\documentclass[12pt,a4paper,reqno]{article}

\usepackage[latin,english]{babel}
\usepackage[utf8]{inputenc}
\usepackage[T1]{fontenc}
\usepackage{eufrak, color, mathrsfs,dsfont}
\usepackage{times}
\usepackage{geometry}
\usepackage{amsmath}
\usepackage{amssymb}
\usepackage{amsthm}
\usepackage{tikz}
\usepackage{mathtools}
\usepackage{mathabx}
\usepackage{cases}
\usepackage{braket}
\usepackage[toc,page]{appendix}
\usepackage{pdfsync}
\usepackage{hyperref}
\usepackage{psfrag}
\mathtoolsset{showonlyrefs}
\usepackage[mathscr]{euscript}
\usepackage{math}
\numberwithin{equation}{section}

\theoremstyle{plain}
\newtheorem{thm}{Theorem}[section]
\newtheorem*{thmno}{Informal Theorem}
\newtheorem{lem}[thm]{Lemma}
\newtheorem{prop}[thm]{Proposition}
\newtheorem{cor}[thm]{Corollary}
\theoremstyle{definition}
\newtheorem{defn}[thm]{Definition}

\theoremstyle{remark}
\newtheorem{rem}[thm]{Remark}

\DeclareMathOperator{\dist}{dist}

\renewcommand{\bar}{\overline}

\newcommand{\sgn}{{\rm sgn}}

\newcommand{\even}{{\rm even}}
\newcommand{\odd}{{\rm odd}}
\newcommand{\ora}[1]{\vec{#1}}
\newcommand{\pa}{\partial}
\newcommand{\vs}{\varsigma}
\newcommand{\vphi}{\varphi}

\newcommand{\normk}[2]{\| #1 \|_{#2}^{k_0,\upsilon}}

\renewcommand{\whi}{\widehat \imath}
\renewcommand{\wti}{\widetilde \imath}

\newcommand{\fy}{{\mathfrak{y}}}

\title{\bf Space quasi-periodic steady Euler flows close to the inviscid
Couette flow }

\author{Luca Franzoi, Nader Masmoudi, Riccardo Montalto}
\begin{document}

\date{}

\maketitle

\noindent
{\bf Abstract.} 
We prove the existence of steady \emph{space quasi-periodic} stream functions,  solutions for the Euler equation in vorticity-stream function formulation in the two dimensional channel $\R\times [-1,1]$. These solutions bifurcate from a prescribed shear equilibrium near the Couette flow, whose profile induces finitely many modes of oscillations in the horizontal direction for the linearized problem. Using a Nash-Moser implicit function iterative scheme, near such equilibrium we construct small amplitude, space reversible stream functions slightly deforming the linear solutions and retaining the horizontal quasi-periodic structure. These solutions exist for most values of the parameters characterizing the shear equilibrium. As a by-product, the streamlines of the nonlinear flow exhibit Kelvin's cat eye-like trajectories arising from the finitely many stagnation lines of the shear equilibrium.
	
\bigskip



\tableofcontents

\section{Introduction}

In the two-dimensional finite channel $\R\times [-1,1]$, we consider a stationary, incompressible inviscid fluid whose stream function $\psi(x,y):\R \times [-1,1 ]\to \R$ solves the stationary Euler equation in vorticity-stream function formulation
\begin{equation}\label{stat.euler.vort}
	\{ \psi,\Delta \psi \} := \psi_x (\Delta \psi)_y - \psi_y (\Delta \psi)_x =0\,,
\end{equation}
coupled with the impermeability condition at the boundary
\begin{equation}\label{imper.cond}
	\psi_x =0 \quad \text{on } \ \{y=\pm 1\}\,.
\end{equation}
In \cite{LZ}, Lin \& Zeng showed on the finite periodic channel, in a $H_{x,y}^s$-neighbourhood of the Couette flow (in the vorticity space), the existence of space periodic steady solutions of \eqref{stat.euler.vort} when $s<\frac32$ and the non-existence of non-parallel traveling solutions when $s>\frac32$. Namely, the regularity threshold $s=\frac32$ discriminates between the presence or not of damping phenomena for the nonlinear evolution of non viscous fluids. The goal of the present paper is to give a new insight to their result when the setting is extended to the quasi-periodic case.
We give now the informal statement of our result.
\begin{thmno}
	Let $\kappa_{0}\in\N$. There exist $\varepsilon_0>0$ small enough and a family of stationary solutions $(\psi_{\varepsilon}(x,y)=\breve{\psi}_{\varepsilon}(\bx,y)|_{\bx=\wt\omega x})_{\varepsilon\in[0,\varepsilon_0]}$ of the Euler equation \eqref{stat.euler.vort} in the finite channel $(x,y)\in\R\times [-1,1]$ that are quasi-periodic in the horizontal direction $x \in \R$ for some frequency vector $\wt\omega\in\R^{\kappa_{0}}$, with $\bx=\wt\omega x\in\T^{\kappa_{0}}$. Such family bifurcates from a shear equilibrium $\psi_{\fm}(y)$ and can be chosen to be arbitrarily close to the stream function of the Couette flow $\psi_{\rm cou}(y):=\tfrac12 y^2$ in $H_{\bx}^{s}H_y^{7/2-}(\T^{\kappa_{0}}\times[-1,1])$, with $s>0$ sufficiently large.
\end{thmno}
The rigorous statement of the result is given in Theorem \ref{main.thm1}.

\medskip

Understanding as much as possible of the fluid behaviour around shear flows is one of the main interests for the hydrodynamic stability research. Around the Couette flow, the simplest among the nontrivial shear flows, Kelvin \cite{Kelvin} and Orr \cite{Orr} proved with experiments and with computations the damping of inviscid flows for the linearized Euler equations at the shear equilibrium, which at first was a surprising result in (apparent) contrast with the (essential) Hamiltonian nature of the equations. The first rigorous justification in the nonlinear case was provided by Bedrossian \& Masmoudi \cite{BeMa} and Deng \& Masmoudi \cite{DeMa}, who proved the asymptotic stability to the planar Couette flow in $\T\times \R$ under perturbations in Gevrey regularity. These results follow the work of Mouhot \& Villani \cite{MouVil} on the nonlinear Landau damping for the Vlasov equation. Other extensions to the inviscid damping near the Couette flow include Ionescu \& Jia \cite{IoJi} in the finite periodic channel,  Yang \& Lin \cite{YangLin},  Bianchini, Coti Zelati \& Dolce \cite{BCzD} for stratified fluids and Antonelli, Dolce \& Marcati \cite{ADM} in the compressible case. For other shear flows, we quote Zillinger \cite{Zill} for monotonic shears and Coti Zelati, Elgindi \& Widmayer \cite{CzEW} for non-monotone shears.

In order to look for quasi-periodic invariant structures, we base our approach on the KAM (Kolmogorov-Arnold-Moser) theory for
Partial Differential Equations. This field started in the
Nineties, with the pioneering papers of Bourgain \cite{Bou}, Craig \& Wayne \cite{CW}, Kuksin \cite{Kuksin}, Wayne \cite{Wayne}. We refer to the recent review article \cite{BertiUMI} for a complete list of references on this topic. In the last years, together with the Nash-Moser implicit function theorem, these techniques have been developed in order to
study time quasi-periodic solutions for PDEs arising from fluid dynamics. For
the two dimensional water waves equations, we mention
Berti \& Montalto \cite{BM}, Baldi, Berti, Haus \& Montalto \cite{BBHM} for time quasi-periodic standing waves and Berti, Franzoi \& Maspero \cite{BFM},
\cite{BFM21}, Feola \& Giuliani \cite{FeoGiu} for time quasi-periodic traveling wave solutions. 
Recently, the existence of time quasi-periodic solutions was proved for the contour dynamics of vortex patches in active scalar equations. We mention Berti, Hassainia \& Masmoudi \cite{BHM} for vortex patches of
the Euler equations close to Kirchhoff ellipses, Hmidi \& Roulley \cite{HmRo} for the 
quasi-geostrophic shallow water equations,  Hassainia, Hmidi \& Masmoudi \cite{HaHmMa} for generalized surface quasi-geostrophic equations, Roulley \cite{Roulley} for Euler-$\alpha$ flows, Hassainia \& Roulley \cite{HaRo} for Euler equations in the unit disk close to Rankine vortices and Hassainia, Hmidi \& Roulley \cite{HaHmRou} for 2D Euler annular vortex patches.
Time quasi-periodic solutions were also constructed for the 3D Euler equations with time quasi-periodic external force \cite{BM20} and for the forced 2D Navier-Stokes
equations \cite{FrMo} approaching in the zero viscosity limit time quasi-periodic solutions of the 2D Euler equations for all times. We finally mention that time quasi-periodic solutions for the Euler equations were constructed also by Crouseilles \& Faou \cite{CF} in 2D, with a very recent extension by Enciso, Peralta-Salas \& de Lizaur \cite{EPsTdL} in 3D and even dimensions: we remark that these latter solutions are engineered so that there are no small divisors issues to deal with, with consequently much easier proofs and a drawback of not having information on the eventual stability of the solutions. 

The paragraph above shows how KAM normal form techniques started very recently to be developed in Fluid Dynamics, in order to construct quasi-periodic solutions in time. On the contrary, there are only few works where the question of the quasi-periodicity \emph{in space} is considered. To the best of our knowledge, the first result of space bi-periodic solutions to PDEs is due to Scheurle \cite{Scheurle} for a semilinear equation on a two-dimensional strip in analytic regularity, whose solutions locally bifurcate from bi-periodic solutions of the linearized system at the equilibrium.
Then, Iooss \& Los \cite{IoLos} proved the bifurcation of stationary solutions in the hydrodynamic stability problem for forced Navier-Stokes equations on cylindrical domains, extended to time-periodic solutions in Iooss \& Mielke \cite{IoMie}. Spatially bi-periodic solutions were studied in Bridges \& Rowlands \cite{BR} for the linear stability analysis of the Ginzburg-Landau equation, in Bridges \& Dias \cite{BD} for stationary 2D gravity-capillary water waves. The general case with more than two spatial frequencies was considered by Valls \cite{Valls} and  Pol\'a\v{c}ik \& Valdebenito \cite{PV} for elliptic equations on $\R^N \times \R$. All these results share the same idea of using one space direction as a temporal one, assuming to have hyperbolic modes for the linearized elliptic operators. The persistence of spatially quasi-periodic oscillations is then proved in \cite{Scheurle} via a Nash-Moser implicit function theorem, in \cite{IoLos} with normalization techniques on the infinite dimensional ``spatial phase space'', while in \cite{Valls}, \cite{PV}  with a center manifold reduction on a finite dimensional system together with a Birkhoff normal form to ensure the application of the standard KAM theorems. As we shall see later, the last two strategies do not look suitable for our problem, since we can only establish the existence of the nonlinear elliptic equation to solve and few properties on the regularity of the nonlinearity, which do not seem enough to check the assumptions for the KAM theorem. Therefore, for our purposes, we preferred to use the Nash-Moser approach as developed by Berti \& Bolle \cite{BB}, which provides also a better description of the final solutions. We conclude by mentioning that spatial dynamics techniques in Fluid Dynamics were applied by Groves \& Wahl\'en \cite{GroWah} to study the existence of small amplitude, solitary gravity-capillary water waves with arbitrary distribution of vorticity.

\subsection{Main result}

Our construction starts with prescribing a potential function $Q_{\fm}(y)$, even in $y$, depending on a parameter $\fm\gg 1$ such that, in the limit $\fm\to\infty$, it uniformly approaches the classical potential well
\begin{equation}\label{Qm.intro}
	Q_{\fm}(y)=Q_{\fm}(\tE,\tr;y) \stackrel{\fm\to\infty}{\rightarrow} Q_{\infty}(\tE,\tr;y) := \begin{cases}
		0 & |y|>\tr \,,\\
		-\tE^2 & |y|<\tr\,,
	\end{cases}
\end{equation}
where $\tr\in(0,1)$ is the width of the well and $\tE>1$ is related to its depth. The potential $Q_{\fm}$ is analytic in all its entries and its derivatives approach the derivatives of $Q_{\infty}$ on compact sets avoiding the points $y=\pm \tr$. The explicit expression of $Q_{\fm}(y)$ is provided in \eqref{Qm.strategy}.  Moreover,  the parameters $\tE$ and $\tr$ are related by the analytic constrain
\begin{equation}\label{constrain}
	\tE\tr = (\kappa_0+\tfrac14) \pi\,.
\end{equation}
The value $\kappa_0\in\N$ is fixed from the very beginning and it prescribes via \eqref{constrain} the exact number of negative eigenvalues $-\lambda_{1,\fm}^2(\tE),...,-\lambda_{\kappa_0,\fm}^2(\tE)<0$ for the operator
\begin{equation}\label{cLm.intro}
	\cL_{\fm} : = - \pa_{y}^2 + Q_{\fm}(y)\,, \quad \text{with eigenfunctions } \quad \cL_{\fm} \phi_{j,\fm} = -\lambda_{j,\fm}^2 \phi_{j,\fm}\,,
\end{equation}
where we imposed Dirichlet boundary conditions on $[- 1, 1]$. The rest of the spectrum $(\lambda_{j,\fm}^2(\tE))_{j\geq \kappa_{0}+1}$ is strictly positive. We remark that the eigenfunctions $(\phi_{j,\fm}(y)=\phi_{j,\fm}(\tE;y))_{j\in \N}$, which form a $L^2$-orthonormal basis with respect to the standard $L^2$-scalar product, depend explicitly  on the parameter $\tE$.

 Shear flows, namely velocity fields of the form $(u,v)=(U(y),0)$ for some function $U(y)$ depending only on $y\in[-1,1]$, are  exact stationary solutions of \eqref{stat.euler.vort} under the boundary conditions in \eqref{imper.cond}.
 The next step is to introduce the shear flow $(\psi_{\fm}'(y),0)$ that plays the role of equilibrium point. In particular, we define the stream function $\psi_{\fm}(y)$ as the solution of the linear ODE
 \begin{equation}\label{ODE.psim.intro}
 	\psi_{\fm}'''(y) = Q_{\fm}(y) \psi_{\fm}'(y)\,, \quad y\in [-1,1]\,.
 \end{equation}
In Section \ref{sec.near_couette} we will construct such  stream function $\psi_{\fm}(y)$, even in $y$ because of the parity of $Q_{\fm}(y)$, so that, for $\fm\gg 1$ large enough, the corresponding velocity field $(\psi_{\fm}'(y),0)$ is close to the well-known Couette flow $(y,0)$  with respect to the width $\tr$ in the $H^3$-topology. Roughly speaking, the function $\psi_{\fm}(y)$ solving \eqref{ODE.psim.intro} behaves almost linearly when $|y| > \tr$ and exhibits oscillations of frequency $\tE$ and amplitude $\tE^{-2}$ in the inner region  $|y| < \tr$. Therefore, the shear flow $(\psi_{\fm}'(y),0)$ is \emph{non-monotone}. In Lemma \eqref{lemma monotono}, we will show that $\psi_{\fm}(y)$ has exactly $2\kappa_{0}+1$ critical points, denoted by $\ty_{0,\fm}:=0$ and $(\pm \ty_{p,\fm})_{p=1,...,\kappa_{0}}$, with $0<\ty_{1,\fm}<...<\ty_{\kappa_0,\fm}<\tr$. These points lead  to divide the interval $[-1,1]$ into the union of stripes $(\tI_{p})_{p=0,1,...,\kappa_{0}}$, where
\begin{equation}\label{sets Ip}
	\begin{aligned}
		& \tI_{p}:=\{ y \in \R :  \ty_{p,\fm} \leq |y| \leq \ty_{p+1,\fm}\}\,, \quad  p=1,...,\kappa_{0}\,, \\
	\end{aligned}
\end{equation}
with $\ty_{\kappa_0+1,\fm}:=1$.  We will also show in Theorem \ref{nonlin_eq} that $\psi_{\fm}(y)$ solves on each set $\tI_{p}$ a second-order nonlinear ODE. Namely, we prove that there exist $\kappa_0+1$ functions $F_{0,\fm}(\psi),F_{1,\fm}(\psi),...,F_{\kappa_0,\fm}(\psi)$ such that, for any $y\in\tI_{p}$, $p=0,1,...,\kappa_{0}$,
\begin{equation}\label{eq.for.psim}
	Q_{\fm}(y) = F_{p,\fm}'(\psi_{\fm}(y)) \ \Rightarrow \ 
	\psi_{\fm}''(y) = F_{p,\fm}(\psi_{\fm}(y)) \,,
\end{equation}
with continuity of finitely many derivatives at the boundaries of each set $\tI_{p}$ with the adjacent problems, meaning that, for a given $S\in\N$ large enough, for any $0\leq n \leq S+1$ and any stripe index $p=1,...,\kappa_0$,
\begin{equation}\label{cont.Fpm.intro}
	\lim_{|y|\to\ty_{p,\fm}^-} \pa_{y}^{n} (F_{p-1,\fm}(\psi_{\fm}(y)) )= \lim_{|y|\to\ty_{p,\fm}^+}\pa_{y}^{n}( F_{p,\fm}(\psi_{\fm}(y)) ) = \psi_{\fm}^{(n+2)}(\ty_{p,\fm}) \,.
\end{equation}
The regularity condition \eqref{cont.Fpm.intro} is ensured by suitable properties of the $Q_{\fm}(y)$: we postponed this explanation to Section \ref{section strategy} ``The shear equilibrium $\psi_{\fm}(y)$ close to Couette and its nonlinear ODE'' and Section \ref{subsec.Q.couette}.

On the two-dimensional channel $\R\times [-1,1]$, we impose quasi-periodic condition in the $x$-direction, that is,
the fluid evolves in the embedded domain
\begin{equation}\label{D_domain}
	\begin{aligned}
		&\cD:= \T^{\kappa_{0}}\times [-1,1] \hookrightarrow \R \times [-1,1]
		 \,, \quad \T^{\kappa_0}:=(\R/2\pi\Z)^{\kappa_0}\,.
	\end{aligned}
\end{equation}
 On the domain $\cD$ we define the Laplacian $\Delta_{\omega}:=(\omega\cdot\pa_\bx)^2+\pa_y^2$ for some frequency vector $\omega\in \R^{\kappa_{0}}$, where $\bx\in\T^{\kappa_0}$.  It is well known that a subclass of solutions of the steady Euler equation \eqref{stat.euler.vort} is given by those stream functions $\psi(x,y)$ that additionally solve semilinear elliptic equations of the form $\Delta \psi:= (\pa_{x}^2+ \pa_{y}^2) \psi= F(\psi)$, for some function $F:\R\to\R$.

The goal of this paper is to construct solutions to the steady Euler equation \eqref{stat.euler.vort}-\eqref{imper.cond}  in the  domain $\cD$ close to the shear equilibrium $(\psi_{\fm}'(y),0)$. In particular, by \eqref{stat.euler.vort} and \eqref{eq.for.psim}, we look for stream functions quasi-periodic in $x$ of the form 
\begin{equation}\label{form.stream}
	\psi(x,y) =\breve{\psi}(\bx,y)_{\bx=\omega x}= \psi_{\fm}(y) +  \vf(\bx,y)|_{\bx=\omega x}\,,    \quad 	\vf(\bx, -1) \!=\! \vf(\bx, 1) = 0 \,, 
\end{equation}
where $\psi_{\fm}(y)$ solves \eqref{eq.for.psim} and $\vf(\bx,y)$  is a solution of
\begin{equation}\label{euler.pert.1}
		\{\psi_{\fm} ,\Delta_{\omega} \vf\} + \{\vf,\psi_\fm''\} + \{\vf,\Delta_{\omega}\vf\} =0\,.
\end{equation}
By a direct computation, we have that a particular class of solutions of \eqref{euler.pert.1} is given by those functions $\vf(\bx, y)$ solving, for any $p=0,1,...., \kappa_0$,
\begin{equation}\label{intro.eq.for.vf}
\Delta_{\omega} \vf(\bx,y) = F_{p,\eta}(\psi_{\fm}(y)+\vf(\bx,y)) - F_{p,\fm}(\psi_{\fm}(y))\,, \quad (\bx,y)\in \T^{\kappa_0}\times \tI_{p}\,.
\end{equation}
The functions $(F_{p,\eta}(\psi))_{p=0,1,...,\kappa_{0}}$ in \eqref{intro.eq.for.vf} are regularized versions of the functions $(F_{p,\fm}(\psi))_{p=0,1,...,\kappa_{0}}$, suitably defined for a small parameter $\eta>0$ as in \eqref{F_eta_reg} of the form 
\begin{equation}\label{Fp.eta.intro}
\begin{footnotesize}
	\begin{aligned}
			F_{p,\eta}(\psi) = \begin{cases}
			\frac12\big(F_{p-1,\fm}(\psi) + F_{p,\fm}(\psi) \big) = F_{p-1,\eta}(\psi)& | \psi - \psi_{\fm}(\ty_{p,\fm}) | \leq \eta \,, \\
			F_{p,\fm}(\psi) &
			\begin{aligned}
			 &	| \psi - \psi_{\fm}(\ty_{p,\fm}) | \geq 2\eta \ \ \text{and}\\
			 & | \psi - \psi_{\fm}(\ty_{p+1,\fm}) | \geq 2\eta\,,
			\end{aligned} \\
			\frac12\big(F_{p+1,\fm}(\psi) + F_{p,\fm}(\psi) \big) = F_{p+1,\eta}(\psi) & | \psi - \psi_{\fm}(\ty_{p+1,\fm}) | \leq \eta \,,
		\end{cases}
	\end{aligned}
\end{footnotesize}
\end{equation}
with smooth connections in the remaining regions,
so that they uniformly converge in the limit $\eta\to 0$ to the functions $F_{p,\fm}(\psi)$  in \eqref{eq.for.psim}-\eqref{cont.Fpm.intro}, see Proposition \ref{prop.F.etax}. Ultimately, in Section \ref{sez.nonlin.lin} we will choose $\eta=\varepsilon^{\frac{1}{S}}$ as in \eqref{link eta varepsilon}, where $\varepsilon$ denotes the size of the perturbation $\vf(\bx,y)$ in \eqref{form.stream} and where $S\in\N$ is the number of derivatives that we have to control in \eqref{cont.Fpm.intro}.

The linearization of the equation \eqref{euler.pert.1} around the equilibrium $\vf =0$ is given by
\begin{equation}\label{euler.pert.lin.1}
		\{\psi_{\fm} ,\Delta_{\omega} \vf\} + \{\vf,\psi_\fm''\}  =0\,.
\end{equation}
By \eqref{cLm.intro}-\eqref{ODE.psim.intro}, a particular class of solutions of \eqref{euler.pert.lin.1} on the  domain $\cD$ is given by
\begin{equation}\label{linear.eq.vf.intro}
	(\omega\cdot\pa_{\bx})^2\vf(\bx,y) = \cL_{\fm}\vf(\bx,y)\,, \quad \vf(\bx,-1)=\vf(\bx,1) =0\,,
\end{equation}
where the self-adjoint Schrödinger operator $\cL_{\fm}$, defined in \eqref{cLm.intro} with Dirichlet boundary conditions on $[-1,1]$, is studied in Proposition \ref{L_operator},

The linearized equation \eqref{euler.pert.lin.1}-\eqref{linear.eq.vf.intro} around the trivial equilibrium $\vf \equiv 0$  admits the family of space quasi-periodic solutions
\begin{equation}\label{linear.sol.intro}
	\vphi(x, y) = \sum_{j=1}^{\kappa_0} A_j \cos(\lambda_{j,\fm}(\tE) x) \phi_{j,\fm}(y)\,, 
\end{equation}
for some nonzero coefficients $A_j\in\R\setminus\{0\}$ with frequency vector
\begin{equation}\label{omega.fm.intro}
	\omega \equiv \ora{\omega}_{\fm}(\tE):=(\lambda_{1,\fm}(\tE),...,\lambda_{\kappa_0,\fm}(\tE))\in\R^{\kappa_0}\setminus\{0\}\,.
\end{equation}
The analysis of the whole linearized systems at the equilibrium and the geometry of the ``spatial'' phase space is postponed to Section \ref{sez.nonlin.lin}.  We will also prove in Proposition \ref{omega.fm.tE.diophantine} that, for most values of $\tE\in [\tE_1,\tE_2]$, with $\tE_1>(\kappa_{0}+\tfrac14)\pi$, the vector $\ora{\omega}_{\fm}(\tE)$ in \eqref{omega.fm.intro} is Diophantine: namely, given $\bar\upsilon\in(0,1)$ and $\bar\tau \gg 1$ sufficiently large, there exists a Borel set
	\begin{equation}\label{diofantea omega vec E imperturbato.intro}
	\bar\cK = \bar\cK(\bar\upsilon,\bar\tau) := \big\{  \tE\in [\tE_1,\tE_2] \,: \, 	|\vec \omega_{\mathtt m}(\mathtt E) \cdot \ell| \geq  \bar\upsilon \braket{\ell}^{-\bar\tau}, \ \forall \,\ell \in \Z^{\kappa_0} \setminus \{ 0 \}  \big\}\,,
\end{equation}
such that $\tE_2-\tE_2 - |\bar\cK| = o(\bar\upsilon)$. This ensures that the linear solutions in \eqref{linear.sol.intro} are quasi-periodic with non-resonant frequency vectors.

The equation \eqref{euler.pert.1} enjoys some symmetries. Since the shear equilibrium $\psi_{\fm}(y)$ is even in  $y\in[-1,1]$ and so are the eigenfunctions of the linear operator $\cL_{\fm}$ in Proposition \ref{L_operator}, we have that \eqref{euler.pert.1}  is invariant with respect to the involution $\vf(\cdot,y)\mapsto \vf(\cdot,-y)$. Moreover, the equation \eqref{euler.pert.1} is also invariant with respect to the involution $\vf(\bx,y)\mapsto \vf(-\bx,-y)$: we refer to such solutions as \emph{space reversible}, or simply \emph{reversible}.  We conclude that we look for solutions
\begin{equation}\label{parity}
	\vf(\bx,y)\in \even(\bx)\even(y)\,.
\end{equation}
The function $\vf(\bx,y)$ is searched in the Sobolev space $H^{s,3}$, as defined in \eqref{sobolev.sp}.

The main result of this paper is the existence of a stream function of the form \eqref{form.stream}, where the functions $\vf(\bx,y)$ are small amplitude, reversible \emph{space quasi-periodic} solutions of the system \eqref{euler.pert.1} with frequency vector $\omega\in\R^{\kappa_{0}}$, bifurcating from a solution \eqref{linear.sol.intro} of the linearization around the trivial equilibria. Such solutions are constructed for a \emph{ fixed valued of the depth} $\tE \in \bar\cK$ in \eqref{diofantea omega vec E imperturbato.intro} and \emph{for most values of an auxiliary parameter} 
\begin{equation}\label{intro parametro mathtt A dim}
	\mathtt A \in {\cal J}_\varepsilon(\mathtt E) := [\tE-\sqrt\varepsilon,\tE+\sqrt\varepsilon]\,.
\end{equation}
This new parameter is introduced to ensure that  frequency vectors $\omega \in\R^{\kappa_0}$, close to the unperturbed frequency vector $\ora{\omega}_{\fm}(\tE)$ in \eqref{omega.fm.intro}, is non-resonant as well.

\begin{thm}\label{main.thm1}
{\bf (Spatial KAM for 2D Euler equations in a channel).}
	Fix $\kappa_0\in\N$ and  $\fm\gg 1$.  Fix also $\tE\in\bar\cK$ as in \eqref{diofantea omega vec E imperturbato.intro} and $\xi = (\xi_1, \ldots, \xi_{\kappa_0}) \in \R_{>0}^{\kappa_0}$. Then  there exist $\bar s>0$, $\varepsilon_0>0$ such that the following hold.
	\\[1.5mm]
	\noindent 
	$1)$ 
	 For any $\varepsilon \in (0, \varepsilon_0)$ there exists a Borel set $\cG_{\varepsilon}=\cG_{\varepsilon}(\tE)\subset \cJ_{\varepsilon}(\tE)$, with$\cJ_{\varepsilon}(\tE)$ as in \eqref{intro parametro mathtt A dim} and with density 1 at $\tE$ when $\varepsilon\to0$, namely $\lim_{\varepsilon \to 0} (2\sqrt\varepsilon)^{-1}|\cG_{\varepsilon}(\tE)| =1$;
	\\[1.5mm]
	\noindent 
	$2)$ There exists $h_\varepsilon=h_{\varepsilon}(\tE) \in H^3_0([- 1, 1])$, $\| h_\varepsilon \|_{H^3} \lesssim \varepsilon$, $h_\varepsilon = {\rm even}(y)$, such that, for any $\tA\in\cG_{\varepsilon}$,  the equation \eqref{euler.pert.1} has a space quasi-periodic solution of the form
	\begin{equation}\label{sqp.sol.main}
	\begin{footnotesize}
		\begin{aligned}
				\vf_{\varepsilon}(\bx,y)|_{\bx=\wt\omega(\tA) x} = h_{\varepsilon}(\tE;y) + \varepsilon \sum_{j=1}^{\kappa_0}\sqrt{\xi_{j}} \cos(\wt\omega_{j}(\tA)x) \phi_{j,\fm}(\tE;y) + r_\varepsilon(\bx,y)|_{\bx=\wt\omega(\tA) x}\,,
		\end{aligned}
	\end{footnotesize}
	\end{equation}
	where $\vf_{\varepsilon}=\vf_{\varepsilon}(\tE,\tA;\bx,y) = {\rm even}({\bf x}) {\rm even}(y)$, $r_{\varepsilon}=r_{\varepsilon}(\tE,\tA;\bx,y)\in H^{\bar s,3}$ (see definition \eqref{sobolev.sp}), with $\lim_{\varepsilon \to 0} \frac{\| r_\varepsilon\|_{\bar s,3}}{{\varepsilon}}=0$, and $\wt\omega=(\wt\omega_{j})_{j=1,...,\kappa_0}\in \R^{\kappa_0}$, depending on $\tA$ and $\varepsilon$, with $| \wt\omega(\tA)-\ora{\omega}_{\fm}(\tE) | \leq C \sqrt\varepsilon $, with $C>0$ independent of $\tE$ and $\tA$.
	Moreover for any $\varepsilon \in [0, \varepsilon_0]$, the stream function
	\begin{equation}\label{form.stream.thm}
		\psi_{\varepsilon}(x,y) =\breve{\psi}_{\varepsilon}(\bx,y)|_{\bx=\wt\omega(\tA)x} = \psi_{\fm}(y) + \vf_{\varepsilon}(\bx,y)|_{\bx=\wt\omega(\tA)x}\,,    
	\end{equation}
	 with  $\vf_{\varepsilon}(\bx,y)$ as in \eqref{sqp.sol.main}, defines a space quasi-periodic solution of the steady 2D Euler equation \eqref{stat.euler.vort} that is close to the Couette flow with estimates
	 \begin{equation}\label{near.cou.thm}
	     \| \breve{\psi}_{\varepsilon} - \psi_{\rm cou} \|_{\bar s , 3} \lesssim_{\bar s} \tfrac{1}{\sqrt\tE} + \varepsilon\,, \quad \psi_{\rm cou}(y):= \tfrac12 y^2\,.
	 \end{equation}
\end{thm}

Let us make some remarks on the result.
\\[1mm]
\noindent 1) {\it Structure of the stationary solutions.} The stream functions \eqref{sqp.sol.main}-\eqref{form.stream.thm} in Theorem \ref{main.thm1} are slight perturbations of the shear equilibrium $\psi_{\fm}(y)$. The first term of $\vf_{\varepsilon}(\bx,y)$ in \eqref{form.stream.thm} is the  shear $h_{\varepsilon}(y)$  and it comes from the forced modification in \eqref{euler.pert.1} of the local nonlinearities $F_{p,\fm}(\psi)$ into $F_{p,\eta}(\psi)$. By Lemma \ref{unforcing.lemma} and \eqref{link eta varepsilon}, it is small with $\eta^S=\varepsilon$ and therefore vanishes in the limit $\varepsilon\to 0$. The second term of $\vf_{\varepsilon}(\bx,y)$ in \eqref{form.stream.thm} retains the space quasi-periodicity of the linearized solutions \eqref{linear.sol.intro} with frequency vectors $\wt\omega$ that are close to the unperturbed frequency vector $\ora{\omega}_{\fm}(\tE)$ in \eqref{omega.fm.intro}.  This term is constructed with a suitable Nash-Moser iterative scheme in order to deal with the eigenfunctions $(\phi_{j,\fm}(\tE;y))_{j\in\N}$ depending on the parameter $\tE$, which is an issue not present in previous papers. Such solutions exist for fixed values of the depth $\tE\in\bar\cK$ so that $\ora{\omega}_{\fm}(\tE)$ is Diophantine and for most values of the auxiliary parameter $\tA\in\cJ_{\varepsilon}(\tE)$ so that $\wt\omega=\wt\omega(\tA,\varepsilon)$ is non-resonant as well. We refer to Section \ref{section strategy} ``A Nash-Moser scheme of hypothetical conjugation with the auxiliary parameter''
 for an extensive discussion.
 \\[1mm]
 \noindent 2) {\it From quasi-periodic stationary to quasi-periodic traveling.}  By changing the frame reference $x \mapsto x -ct$ with an arbitrary speed $c\in\R$, we deduce the existence of \emph{quasi-periodic traveling} solutions, according to \cite{BFM}, of the form
\begin{equation}
\begin{aligned}
    \psi_{\rm tr}(t,x,y)& := c y + \psi_{\varepsilon}(x-ct,y)  \\
    &\ = c y + \psi_{\fm}(y) + \vf_{\varepsilon}(\phi,y)|_{\phi=\bx-\vartheta=\wt\omega(x-ct)}\,, \quad \bx,\vartheta\in\T^{\kappa_0}\,,
    \end{aligned}
\end{equation}
solving the Euler equations in vorticity formulation
\begin{equation}
    (\Omega_{\rm tr})_t+(\psi_{\rm tr})_y (\Omega_{\rm tr})_x - ( \psi_{\rm tr})_x  (\Omega_{\rm tr})_y =0\,, \quad \Omega_{\rm tr}:= \Delta\psi_{\rm tr}\,.
\end{equation}
We read these solutions also as quasi-periodic in time with time frequency vector $c\,\wt\omega$ parallel to the space frequency vector $\wt\omega$. It is of great interest to see whether there exist quasi-periodic solutions to the Euler equations both in time and in space, but with non-collinear frequency vectors.
\\[1mm]
\noindent 3) {\it Generalized Kelvin cat's eyes, lack of damping and regularity thresholds.} The flow generated by the stream function \eqref{sqp.sol.main}-\eqref{form.stream.thm} is a deformation of the near-Couette shear flow $(\psi_{\fm}'(y),0)$ of the form
\begin{equation}\label{cat-eye-flow}
    \begin{pmatrix}
        u(x,y)\\ v(x,y)
    \end{pmatrix}= \begin{pmatrix}
        \psi_{\fm}'(y) + h_{\varepsilon}'(y)\\ 0
    \end{pmatrix} + \varepsilon \sum_{j=1}^{\kappa_0} \sqrt{\xi_j} \begin{pmatrix}
        \cos(\wt\omega_j x) \phi_{j,\fm}'(y) \\ 
        \wt\omega_j \sin(\wt\omega_j x) \phi_{j,\fm}(y)
    \end{pmatrix} + o(\varepsilon)\,.
\end{equation}
We first observe that, since $\psi_{\fm}(y)$, $h_{\varepsilon}(y)$ and the eigenfunctions $\phi_{j,\fm}(y)$ are even in $y$, the streamlines of the perturbed flows have a generalized cat's eyes structure near the stagnation line $\{y=0\}$ of the shear flow $(\psi_{\fm}'(y)+h_{\varepsilon}'(y),0)$, with saddle and center points near the roots of the trigonometric equation $\sum_{j=1}^{\kappa_0}\sqrt{\xi_j}\wt\omega_j \phi_{j,\fm}(0)\sin(\wt\omega_j x) =0$. Possible other cat's eyes-like streamlines may appear near the lines $\{ y = \pm \ty_{p,\fm}\}$, $p=1,...,\kappa_0$, corresponding to the critical points for $\psi_{\fm}(y)$, depending on further properties of the eigenfunctions $(\phi_{j,\fm}(y))_{j=1,...,\kappa_0}$ that we do not investigate in this paper. 

We also observe that, no matter the geometry of the streamlines, the velocity field \eqref{cat-eye-flow} has non-trivial vertical component that is quasi-periodic in $x\in\R$. The presence of such quasi-periodic stationary solutions prevents damping phenomena in the evolution of the dynamics for the Euler equations with quasi-periodic conditions in $x$. Our result agrees with the analysis made for the periodic case in \cite{LZ} and their (vorticity) regularity threshold $s<\frac32$. Indeed in Theorem \ref{proxim.couette} we show that $\psi_{\fm}(y)$ is close to $\psi_{\rm cou}(y)$ with $\sqrt{\tr}$ in the $H_y^3$-topology. At the same time, arguing as in the proof of Theorem \ref{proxim.couette} with an easy computation that we omit here, it is possible to show  that the bound for the $L^2$-norm of $\psi_{\fm}^{(4)}(y)$ diverges with $\tr^{-1/2}$. Therefore, a standard interpolation argument shows that we can construct $\psi_{\fm}(y)$ arbitrarily close to $\psi_{\rm cou}(y)$ in the $H_{y}^\rho$-topology, with the (stream) regularity $\rho< \frac72$.
We remark that here only the regularity for the estimate \eqref{near.cou.thm} in the vertical direction $y$ is below such threshold, whereas the Sobolev regularity in the horizontal direction $\bx = \wt\omega x$ has to be sufficiently large to compensate, during the Nash-Moser iteration, the loss of derivatives coming from the small divisors and the Diophantine conditions, see\eqref{def:DCgt}.

\subsection{Strategy of the proof}\label{section strategy}

We look for stationary solutions of the Euler equation in vorticity-stream function formulation \eqref{stat.euler.vort} as solutions of semilinear elliptic PDEs \eqref{intro.eq.for.vf}. The quasi-periodic solutions in $x$ of Theorem \eqref{main.thm1} are then searched via a Nash-Moser implicit function theorem on such elliptic equations, with initial guess given by the solutions \ref{linear.sol.intro} of the linearized Euler equations at the shear equilibrium \eqref{euler.pert.lin.1}.
The main difficulties and novelties of our results can be summarized as follows:
\\[1mm]
\indent $\bullet$ Each space quasi-periodic function $\vf_{\varepsilon}(\bx,y)$ solve the nonlinear PDE \eqref{intro.eq.for.vf} with nonlinearities explicitly depending of the size $\varepsilon$ of the solution;
\\[1mm]
\indent $\bullet$ The nonlinearity of the semilinear elliptic problem that we solve is actually an "unknown" of the problem and it has to be constructed in such a way that one has a near Couette, space quasi-periodic solution to the Euler equation \eqref{euler.pert.1};
\\[1mm]
\indent $\bullet$ The nonlinearities have finite smoothness and their derivatives lose in size;
\\[1mm]
\indent $\bullet$ The unperturbed frequencies of oscillations are only implicitly defined and their non-degeneracy property relies on an asymptotic expansion for large values of the parameter. It implies that the required non-resonance conditions  are  not trivial to verify;
\\[1mm]
\indent $\bullet$ The basis of eigenfunctions $(\phi_{j,\fm}(y))_{j\in\N}$ of the operator $\cL_{\fm}$ in \eqref{cLm.intro} is not the standard exponential basis and depends explicitly on the parameter $\tE$.
\medskip

\noindent
We now illustrate the main steps to prove Theorem \ref{main.thm1} and how we will overcome the main difficulties.
\\[1mm]
\noindent {\bf The shear equilibrium $\psi_{\fm}(y)$ close to Couette and its nonlinear ODE.} The first issue that we need to solve is to determine which nonlinear differential equation is satisfied by $\psi_{\fm}(y)$ and the regularity properties of the nonlinearity. Our starting point is the linear ODE
\begin{equation}\label{ODE.1.intro}
    \psi_{\fm}'''(y)= Q_{\fm}(y)\psi_{\fm}'(y)\,,
\end{equation}
where $Q_{\fm}(y)$ is a prescribed analytic potential of the form  (see also \eqref{Qgamma})
\begin{equation}\label{Qm.strategy}
\begin{aligned}
    Q_{\fm}(y)= Q_{\fm}(\tE,\tr;y) & := -\tE^2\Big( \Big( \Big(\frac{\cosh(\frac{y}{\tr})}{\cosh(1)}\Big)^\fm +1 \Big)^{-1}+g_{\fm,S}\big(\tfrac{y}{\tr}\big) \Big)\,,
    \end{aligned}
\end{equation}
that approaches the singular finite well potential $Q_{\infty}(y):= -\tE^2\,\chi_{(-1,1)}(\frac{y}{\tr}) $ in \eqref{Qm.intro} with estimates as in Lemma \ref{lemma QmQinfty}. There are some degrees of freedom in the choice of the potential $Q_{\fm}(y)$ that we will take advantage of in the construction of our solutions. 

First, the choice of the parameters $\tE>(\kappa_{0}+\frac14)\pi$ and $\tr\in(0,1)$ controls both the numbers of negative eigenvalues of the Schr\"odinger operator $\cL_{\fm}:= -\pa_y^2 + Q_{\fm}(y)$, via the constrain $\tE\tr = (\kappa_0+\tfrac14)\pi$ in \eqref{constrain}, and their non-resonance properties. These $\kappa_0$ negative eigenvalues determine the frequencies of oscillations in the horizontal direction for the solutions of the linearized system at the equilibrium, see \eqref{linear.sol.intro}-\eqref{omega.fm.intro}. The non-degeneracy of the curve $\tE\mapsto \ora{\omega}_{\fm}(\tE)$, which is needed to ensure Diophantine non-resonance conditions on the frequency vector $\ora{\omega}_{\fm}(\tE)$ and on its perturbations, is proved in Section \ref{deg.KAM}. We remark that an extra difficulty is due to the fact that these $\kappa_0$ linear frequencies are proved initially to be close to the $\kappa_0$ real roots of a transcendental equation, see \eqref{secular_0} in Theorem \ref{L_operator}. This issue is overcome by proving asymptotic expansions of the latter roots, see Lemma \eqref{lemma.asympt}, and then by a perturbative argument.

Back to the second order ODE \eqref{ODE.1.intro} for $\psi_{\fm}'(y)$, its odd solutions, roughly speaking, behave as the affine Couette shear flow in the outer region $|y|> \tr$  and as oscillations of frequency $\tE$ and amplitude $\tE^{-2}$ in the inner region $|y|<\tr$. In particular, the shear flow $(\psi_{\fm}'(y),0)$ has $2\kappa_0+1$ stagnation lines, including the axis $\{y=0\}$ and with the remaining ones symmetric with respect to it. These stagnation lines will be the boundaries of the stripes $\R\times \tI_{p}$ in \eqref{sets Ip}, \eqref{D_domain}, where  $(\pm\ty_{p,\fm})_{p=0,1,...,\kappa_0}$ denote the critical points of $\psi_{\fm}(y)$.

The second-order nonlinear ODE satisfied by $\psi_{\fm}(y)$ is determined via the Cauchy problem solved with a nonlinear vector field $F(\psi)$ induced by $Q_{\fm}(y)$, see \eqref{eq.for.psim}, and with initial datum $F(\psi_{\fm}(0)) = \psi_{\fm}''(0) $. Because $\psi_{\fm}'(y)$ is not monotone, the nonlinearity is constructed locally on each domain where $\psi_{\fm}$ is invertible. Therefore, morally speaking, the nonlinearity $F(\psi)$ globally behaves as a multi-valued function defined on the range domain $\psi_{\fm}([-1,1])$. We now sketchily describe the general idea of the construction of the nonlinearity starting around $\psi_{\fm}(0)$. Since $\psi_{\fm}'(y)$ odd, we can locally solve the initial value problem
\begin{equation}\label{F 0 m intro}
     \begin{cases}
         F_{0,\fm}' (\psi) = Q_{\fm} (\psi_{\fm}^{-1}(\psi))\,, \quad \psi\in \psi_{\fm}([0,T])\,, \\
     F_{0,\fm}(\psi_{\fm}(0)) = \psi_{\fm}''(0)\,.
     \end{cases}
\end{equation}
The local solution of \eqref{F 0 m intro} extends until it meets the next critical point of $\psi_{\fm}(y)$, namely on $\psi_{\fm}([0,\ty_{1,\fm}))$, recalling that $\psi_{\fm}'(\ty_{1,\fm})=0$.  To pass over this critical point, we define a new Cauchy problem
\begin{equation}\label{F 1 m intro}
     \begin{cases}
         F_{1,\fm}' (\psi) = Q_{\fm} (\psi_{\fm}^{-1}(\psi))\,, \quad \psi\in \psi_{\fm}([\ty_{1,\fm},T])\,, \\
     F_{1,\fm}(\psi_{\fm}(\ty_{1,\fm})) = F_{0,\fm}(\psi_{\fm}(\ty_{1,\fm}))\,.
     \end{cases}
\end{equation}
By \eqref{ODE.1.intro}, we deduce $\psi_{\fm}'''(\ty_{1,\fm})=0$. By this latter identity, it is possible to show that $F_{0,\fm}(\psi_{\fm}(y))$ and $F_{1,\fm}(\psi_{\fm}(y))$ agree at $y=\ty_{1,\fm}$ with $\cC^1$-continuity. For the purposes of the Nash-Moser nonlinear iteration, $\cC^1$-regularity for the nonlinearity is definitely not enough to deal with the loss of derivatives coming from the small divisors. It is at this point that we use the extra degrees of freedom coming from the potential $Q_{\fm}(y)$: indeed, in \eqref{Qm.strategy} we can choose the corrector $g_{\fm,S}$ to impose arbitrarily finitely many vanishing conditions on the odd derivatives of $Q_{\fm}(y)$ at $y=\ty_{1,\fm}$, and consequently on $\psi_{\fm}(y)$ by \eqref{ODE.1.intro}, which we use to ensure $\cC^S$-regularity of the nonlinearities $F_{1,\fm}$ and $F_{0,\fm}$ at $\psi=\psi_{\fm}(y)$ for an arbitrarily fixed and large $S\in\N$. The main idea here is that the regularity is determined by the "local evenness" of $Q_{\fm}(y)$ and $\psi_{\fm}(y)$ around the critical point $\ty_{1,\fm}$, meaning that we can write
\begin{equation}
	Q_{\fm}(\ty_{1,\fm}+\delta) = \sum_{n=0}^{S} \frac{Q_{\fm}^{(2n)}(\ty_{1,\fm})}{(2n)!} \delta^{2n} +\frac{Q_{\fm}^{(2S+1)}(\ty_{1,\fm})}{(2S+1)!} \delta^{2S+1} + o(|\delta|^{2(S+1)})\,,
\end{equation}
and similarly for $\psi_{\fm}(y)$, see Lemma \ref{expansion psim crit}, so that we can invert $\psi_{\fm}$ as a function of $\delta^2$ with inverse of finite regularity. 
 The corrector $g_{\fm,S}(z)$ will be a polynomial functions, therefore analytic, that we use to control the local behaviour of $Q_{\fm}(y)$ also at the other critical points, without affecting the global shape of the potential. 

This construction is then iterated when we reach  the remaining finitely many critical points $\ty_{2,\fm}<...<\ty_{\kappa_0,\fm}<\ty_{\kappa_0+1,\fm}:=1$.  All the technical details are provided in Section \ref{sec.near_couette}: in particular, in Theorem \ref{nonlin_eq} we show that there exist $\cC^{S+1}(\R)$ functions $F_{0,\fm}(\psi),...,F_{\kappa_0,\fm}(\psi)$ such that
\begin{equation}\label{nonODE.intro}
    \psi_{\fm}''(y) = F_{p,\fm}(\psi_{\fm}(y)) \quad \forall\, y \in \tI_{p}\,, \quad p=0,1,...,\kappa_{0}\,,
\end{equation}
and with $\cC^{S+1}$-continuity at $\psi=\psi_{\fm}(y)$ as in \eqref{cont.Fpm.intro}.
\\[1mm]
\noindent{\bf A forced elliptic PDE for the perturbation of the shear equilibrium.}
Now that we have a good description of the shear equilibrium $\psi_{\fm}(y)$, we look for solutions depending on $x\in\R$ and ask what problems are solved by stream functions of the form $\psi(x,y)=\breve{\psi}(\bx,y)|_{\bx=\omega x}=\psi_{\fm}(y)+\vf(\bx,y)_{\bx=\omega x}$. The first na\"ive attempt would be to look for solutions of the nonlinear PDE
\begin{equation}
	\Delta_{\omega} \breve\psi ( \bx,y) = F_{p,\fm} (\breve \psi (\bx, y)) \,, \quad (\bx,y) \in \T^{\kappa_{0}}\times \tI_{p}\,, \quad p=0,1,...,\kappa_{0}
\end{equation}
with the same nonlinearities $F_{p,\fm}(\psi)$ as in \eqref{nonODE.intro} and insert the ansatz for $\breve{\psi}(\bx,y)$. The equations for the perturbation $\vf(\bx,y)$ would be, for $(\bx,y)\in \T^{\kappa_{0}} \times \tI_{p}$, $p=0,1,...,\kappa_0$,
\begin{equation}
    (\omega\cdot \pa_{\bx})^2\vf - \cL_{\fm}\vf - \big( F_{p,\fm}(\psi_{\fm}(y)+\vf) - F_{p,\fm}(\psi_{\fm}(y)) \big) = 0\,,
\end{equation}
with $\cL_{\fm}$ as in \eqref{cLm.intro}. This approach fails immediately because the continuity at $\pm\ty_{p,\fm}$ for $F_{p-1,\fm}(\psi_{\fm}(y)+\vf(\bx,y))$ and $F_{p,\fm}(\psi_{\fm}(y)+\vf(\bx,y))$ already does not hold any more in general (unless we require $\vf(\bx,\pm\ty_{p,\fm})=0$ for any $\bx\in\T^{\kappa_{0}}$, which is a too strong conditions, not even satisfied by the eigenfunctions of $\cL_{\fm}$). Recalling that our ultimate goal is to solve the Euler equation \eqref{stat.euler.vort}, the idea is to slightly change the nonlinear functions $F_{p,\fm}(\psi)$ and ``make enough room'' in neighbourhoods of the critical values $\psi_{\fm}(\ty_{p,\fm})$ to ensure enough smoothness of the new nonlinearities when $\psi(x,y)$ is evaluated close to the stagnation lines $\{y=\pm \ty_{p,\fm}\}$. In particular, for a small parameter $\eta\ll 1$, we will use the regularized nonlinearity $F_{p,\eta}(\psi)$ as in \eqref{Fp.eta.intro} instead of $F_{p,\fm}(\psi)$.
With this (non-unique) choice of the modified nonlinearities, we have that $F_{p,\eta}(\psi) = F_{p-1,\eta}(\psi)$ when $\psi$ belongs to the open neighbourhood $ B_{\eta}(\psi_{\fm}(\ty_{p,\fm}))$, $p=1,...,\kappa_0$.  The finite smooth continuity at the stagnation line $\{y=\pm \ty_{p,\fm}\}$ between $F_{p-1,\eta}(\psi_{\fm}(y)+\vf(\bx,y))$ and $F_{p,\eta}(\psi_{\fm}(y)+\vf(\bx,y))$ is then easily satisfied, as soon as the perturbation $\vf(\bx,y)$ is small enough. 

The new question that arises now is to estimate how close the two nonlinearities $F_{p,\fm}(\psi)$ and $F_{p,\eta}(\eta)$ (together with their derivatives) are with respect to the small parameter $\eta\ll 1$. Generally speaking, one only gets uniformly convergence in the limit $\eta\to 0$ and the derivatives of $F_{p,\eta}(\psi)$ exploding when $\eta \to 0$, due to presence of shrinking cut-off functions. The good news here is that, thanks to the "local evenness" that we were able to impose earlier on $Q_{\fm}(y)$ and $\psi_{\fm}(y)$ at the critical points $\pm\ty_{p,\fm}$, we can prove estimates (see Proposition \ref{prop.F.etax}), for $n=0,1,...,S+1$, 
\begin{equation}\label{stima.eta.intro}
    \sup_{y\in[-1,1]}\sup_{p=0,1,...,\kappa_{0}}| F_{p,\eta}^{(n)}(\psi_{\fm}(y)) - F_{p,\fm}^{(n)}(\psi_{\fm}(y)) | \lesssim_{n} \eta^{S+\frac32 -n}\,.
\end{equation}
We finally conclude that the equation for the perturbation $\vf(\bx,y)$ that we are going to solve is \eqref{euler.pert.1}, which implies, by \eqref{eq.for.psim}, that $\Delta(\psi_{\fm}(y)+\vf(\bx,y))=F_{p,\fm}(\psi_{\fm}(y)+\vf(\bx,y))$ and that $\psi_{\fm}(y)+\vf(\bx,y)$ is a solutions of Euler equation \eqref{stat.euler.vort}. We point out that, by expanding
\begin{equation}
    \begin{aligned}
        & F_{p,\eta}(\psi_{\fm}(y)+\vf(\bx,y)) - F_{p,\fm}(\psi_{\fm}(y)) - F_{p,\fm}'(\psi_{\fm}(y))\vf(\bx,y)\\
        & \quad = F_{p,\eta}(\psi_{\fm}(y)) - F_{p,\fm}(\psi_{\fm}(y))\\
        & \quad +\big( F_{p,\eta}'(\psi_{\fm}(y)) - F_{p,\fm}'(\psi_{\fm}(y)))\vf(\bx,y) + \tfrac12 F_{p,\eta}''(\psi_{\fm}(y)) \vf^2(\bx,y) + ...
    \end{aligned}
\end{equation}
the equation for $\vf$ in \eqref{euler.pert.1} contains the forcing term $F_{p,\eta}(\psi_{\fm}(y)) - F_{p,\fm}(\psi_{\fm}(y))$ and a correction at the linear level $\big( F_{p,\eta}'(\psi_{\fm}(y)) - F_{p,\fm}'(\psi_{\fm}(y)))\vf(\bx,y)$. By \eqref{stima.eta.intro}, both of them are actually arbitrarily small with respect to (powers of) $\eta$, therefore they will be treated as perturbative terms in the Nash-Moser nonlinear iteration. In particular, by a small shifting of the unknown $\vf(\bx,y)$ with the $\bx$-independent function $h_{\eta}(y)$ in Lemma \ref{unforcing.lemma}, it is possible to remove the forcing term and include its contribution directly into the nonlinearity. 
\\[1mm]
\noindent{\bf A Nash-Moser scheme of hypothetical conjugation with the auxiliary parameter.}
Finally, the construction of the quasi-periodic solutions in spatial variable $x$, treated here as a temporal one, follows in the same approach of other KAM papers in fluid dynamics, see for instance \cite{BM}, \cite{BBHM}, \cite{BFM}, \cite{BFM21}. The main points are the splitting of the phase space (here "spatial phase space") into tangential and normal invariant subspaces, the introduction of action-angle coordinates on the tangential subspace and the definition of the nonlinear functional to implement the Nash-Moser iteration. The solutions are searched as embeddings $i:\T^{\kappa_{0}}\to \T^{\kappa_{0}}\times \R^{\kappa_{0}}\times  \cX_{\perp}^s$ in the phase space of the form $\bx \mapsto i(\bx) = (\theta(\bx),I(\bx),z(\bx))$, where $\cX_{\perp}^{\perp}$ is the restriction of the functional space $H^{s,3}\times H^{s,1}$ to the normal subspace. The embedding is searched as the zero of the nonlinear functional
\begin{equation}\label{nonlinear funct intro}
	\bF(i) := \omega\cdot\pa_{\bx}i(\bx) - X_{\cH_{\varepsilon}}(i(\bx))\,,
\end{equation}
where $\cH_{\varepsilon}$ is the Hamiltonian in action-angle coordinate (see also \eqref{H_epsilon})
\begin{equation}\label{cH eps intro}
	\cH_{\varepsilon}(\theta,I,z) = \ora{\omega}_{\fm}(\tE) \cdot I + \tfrac12 \big( z, \Big(\begin{smallmatrix}
		-\cL_{\fm} & 0 \\ 0 & {\rm Id}
	\end{smallmatrix}\Big) \big)_{L^2} + \sqrt\varepsilon P_{\varepsilon}( A(\theta,I,z)) \,,
\end{equation}
with $\ora{\omega}_{\fm}(\tE)$ as in \eqref{omega.fm.intro}, $P_{\varepsilon}$ a perturbative contribution from the nonlinear terms and $A(\theta,I,z)$ the action-angle map as in \eqref{aa_coord}. The frequency vector $\omega\in\R^{\kappa_{0}}$ becomes a parameter to determine in order to get a solutions of $\bF(i)=0$. Here a significant difficulty that was not present in previous works appears. 

We first recall the strategy used in the previous works. In the spirit of analysis of Hamiltonian dynamics of Herman-F\'ejoz \cite{fejoz}, one usually relaxes the problem by introducing a counterterm $\alpha \in \R^{\kappa_{0}}$ and modifying the Hamiltonian $\cH_{\varepsilon}$ in \eqref{cH eps intro} as
\begin{equation}
	\wt\cH_{\alpha} := \alpha \cdot I + \tfrac12 \big( z, \Big(\begin{smallmatrix}
		-\cL_{\fm} & 0 \\ 0 & {\rm Id}
	\end{smallmatrix}\Big) \big)_{L^2} + \sqrt\varepsilon P_{\varepsilon}(A(\theta,I,z)) \,.
\end{equation}
The counterterm $\alpha$ becomes an unknown of the problem together with the embedding $i(\bx)$ and one searches for solutions of 
\begin{equation}\label{nonlinear funct nomore}
	\wt\bF(i,\alpha) = \wt\bF(\omega,\tE,\varepsilon;i,\alpha) := \omega\cdot\pa_{\bx}i(\bx) - X_{\wt\cH_{\alpha}}(i(\bx))=0\,.
\end{equation}
One then obtains a solution $(i_\infty,\alpha_{\infty})(\omega,\tE,\varepsilon)$, defined for all parameters $(\omega,\tE)\in \R^{\kappa_{0}}\times [\tE_1,\tE_2]$, such that \eqref{nonlinear funct nomore} is solved whenever the parameters satisfy the Diophantine non-resonance condition 
\begin{equation}\label{diophan intro}
	| \omega\cdot \ell | \geq \upsilon \braket{\ell}^{-\tau} \quad \forall\, \Z^{\kappa_{0}}\setminus\{0\}\,.
\end{equation}
The original equation $\bF(i_{\infty})=0$ is then solved if $\alpha_{\infty}=\alpha_{\infty}(\omega,\tE,\varepsilon) = \ora{\omega}_{\fm}(\tE)$ and, since $\alpha_{\infty}(\,\cdot\,,\tE,\varepsilon)$ is expected to be invertible for any fixed $\tE$, this should fix $\wt\omega=\wt\omega(\tE,\varepsilon) = \alpha_{\infty}^{-1}(\ora{\omega}_{\fm}(\tE),\tE,\varepsilon)$. To ensure that $\wt\omega=\wt\omega(\tE)$ satisfies the non-resonance condition \eqref{diophan intro}, we need to control a finite number of derivatives in the parameter $\tE\in [\tE_1,\tE_2]$. However, the basis of eigenfunctions $\phi_{j,\fm}(y)$ of the operator $\cL_{\fm}$ in \eqref{cLm.intro} is not the standard exponential basis and depends explicitly on the parameter $\tE$, with the consequence that also the Sobolev phase spaces $\cX_{\perp}^s$ vary with respect to the parameter. We do not have a clear and explicit control on the variations of the eigenfunctions with respect to the parameter $\tE$, as well as of the local nonlinearities. This may be a potential source of divergences in the estimates that may prevent the imposition of the non-resonance conditions.
 
 We apply here a new strategy.  We solve \eqref{nonlinear funct intro}-\eqref{cH eps intro} for a fixed value of the depth $\tE\in (\tE_1,\tE_2)$ such that $\ora{\omega}_{\fm}(\tE)$ is a Diophantine non-resonant frequency vector. We will prove in Proposition \ref{omega.fm.tE.diophantine} that this property holds for most values of $\tE$ in any compact interval $[\tE_1,\tE_2]$. For any $\varepsilon>0$, we introduce an auxiliary parameter
 \begin{equation}
 	\tA \in \cJ_{\varepsilon}(\tE):= [\tE -\sqrt\varepsilon,\tE+\sqrt\varepsilon]
 \end{equation}
 so that $\ora{\omega}_{\fm}(\tA)$ is close to the unperturbed frequency vector $\ora{\omega}_{\fm}(\tE)$ with estimates
 \begin{equation}\label{stima piccola per intro}
 	\big| \pa_{\tA}^n \big( \ora{\omega}_{\fm}(\tE)- \ora{\omega}_{\fm}(\tA) \big) \big| \lesssim \sqrt\varepsilon\,, \quad \forall\,n\in\N_0\,, \quad \forall\,\tA\in\cJ_{\varepsilon}(\tE)\,.
 \end{equation}
We remark that the properties of non-degeneracy and transversality for the vector $\ora{\omega}_{\fm}(\tE)$ (see Theorem \ref{non deg omega gamma} and Proposition \ref{prop:trans_un}) hold also for $\ora{\omega}_{\fm}(\tA)$ and its perturbations on the whole interval $[\tE_1,\tE_2]$ with  constants that are independent of $\varepsilon>0$. 
By adding and subtracting the term $\ora{\omega}_{\fm}(\tA)\cdot I$ in \eqref{cH eps intro}, we now introduce the counterterm $\alpha\in \R^{\kappa_{0}}$ and we consider the modified Hamiltonian
\begin{equation}\label{cH vare alpha intro}
	\begin{aligned}
		\cH_{\varepsilon,\alpha}& := \alpha\cdot I + \tfrac12 \big( z, \Big(\begin{smallmatrix}
			-\cL_{\fm} & 0 \\ 0 & {\rm Id}
		\end{smallmatrix}\Big) \big)_{L^2} \\
	& \ \ \ \ \ + \sqrt\varepsilon \big( P_{\varepsilon}(A(\theta,I,z) + \tfrac{1}{\sqrt\varepsilon} \big(\ora{\omega}_{\fm}(\tE)-\ora{\omega}_{\fm}(\tA)  \big)\cdot I \big)\,.
	\end{aligned}
\end{equation}
The great advantage of this procedure is that the modified Hamiltonian $\cH_{\varepsilon,\alpha}$ directly depends on the new parameter $\tA$ only through the correction term $ \big(\ora{\omega}_{\fm}(\tE)-\ora{\omega}_{\fm}(\tA)  \big)\cdot I $, which is perturbative because of the estimates \eqref{stima piccola per intro}. As before, the counterterm $\alpha$ becomes an unknown of the problem together with the embedding $i(\bx)$ and we search for solutions of 
\begin{equation}\label{nonlinear funct correct}
	\cF(i,\alpha) = \cF(\omega,\tA,\tE,\varepsilon;i,\alpha) := \omega\cdot\pa_{\bx}i(\bx) - X_{\cH_{\varepsilon,\alpha}}(i(\bx))=0\,.
\end{equation}
We stress once more that the nonlinear functional in \eqref{nonlinear funct correct} still depends on $\tE$, but its value is fixed during the Nash-Moser estimate and we are not interest in how the solutions vary with respect to it. The Nash-Moser scheme is not affected by this modification and we will obtain a solution $(i_\infty,\alpha_{\infty})(\omega,\tA,\varepsilon)$, defined for all parameters $(\omega,\tA)\in \R^{\kappa_{0}}\times \cJ_{\varepsilon}(\tE)$, such that \eqref{nonlinear funct correct} is solved whenever the parameters satisfy the Diophantine non-resonance condition \eqref{diophan intro}. 
The modified Hamiltonian $\cH_{\varepsilon,\alpha_{\infty}}$ in \eqref{cH vare alpha intro} will therefore coincide again with $\cH_{\varepsilon}$ in \eqref{cH eps intro}, and consequently
the original equation $\bF(i_{\infty})=0$ in \eqref{nonlinear funct intro}  will be solved, if $\alpha_{\infty}=\alpha_{\infty}(\omega,\tA,\varepsilon) = \ora{\omega}_{\fm}(\tA)$. Since $\alpha_{\infty}(\,\cdot\,,\tA,\varepsilon)$ will be invertible for any fixed $\tA$, this  will fix $\wt\omega=\wt\omega(\tA,\varepsilon) = \alpha_{\infty}^{-1}(\ora{\omega}_{\fm}(\tA),\tA,\varepsilon)$. It will finally be possible to prove that the perturbed frequency vector $\wt\omega(\tA,\varepsilon)$ is Diophantine for most values of $\tA \in\cJ_{\varepsilon}(\tE)$: we will prove this in Theorem \ref{MEASEST}. We remark that the Diophantine conditions for $\wt\omega(\tA;\varepsilon)$ will be weaker than the ones for $\ora{\omega}_{\fm}(\tE)$.

Among the reasons why our modified scheme actually works, we identified the following factors that certainly help the convergence to our result: the equation \eqref{euler.pert.1} for $\vf$ at the end of the day is semilinear, so there is no need to perform any regularization of the linearized vector field; the linearized operator in the normal direction is directly invertible without any reducibility to a diagonal operator and, therefore, no Melnikov non-resonance conditions are needed; the only non-resonance conditions that appear are the Diophantine conditions on the frequency vectors when we invert the operator $\omega\cdot \pa_{\bx}$ on functions with zero average in $\bx\in\T^{\kappa_{0}}$.  We surely find of great interest to see if our strategy still works or can be further improved when these cases are not met and still we have a parament-dependent basis of eigenfunctions.
\\[1mm]
\noindent {\bf A final comment on the parameters and their interdependences.} The construction of the space quasi-periodic stream functions in Theorem \ref{main.thm1} requires several parameters that we have to  tune and match appropriately  throughout the entire paper. For sake of clarity, we list all of them:
\\[1mm]
\indent $\bullet$ $\kappa_0\in\N$ is the number of frequencies of oscillations and ultimately the dimension of the quasi-periodicity. It is fixed once for all at the very beginning;
\\[1mm]
\indent $\bullet$ $\tE \in [\tE_1,\tE_2]$, $\tE_1\gg (\kappa_0+\tfrac14)\pi$, and $\tr\in(0,1)$ parametrize the potential $Q_{\fm}(y)$ in \eqref{Qm.intro}, \eqref{Qm.strategy}, affecting its depth and its width, respectively. They are related by the constraint \eqref{constrain}. The threshold $\tE_1$ will be chosen sufficiently large in Section \ref{deg.KAM}, depending only on $\kappa_0$. The parameter $\tr$ will measure the proximity to the Couette flow, see Proposition \ref{proxim.couette};
\\[1mm]
\indent $\tA\in\cJ_{\varepsilon}(\tE):=[\tE-\sqrt\varepsilon,\tE+\sqrt\varepsilon]$ is the auxiliary parameter close to a fixed value of $\tE \in (\tE_1,\tE_2)$ with $\sqrt\varepsilon$, where $\varepsilon\in(0,\varepsilon_0)$ is the size of the perturbation $\vf_{\varepsilon}$ in \eqref{form.stream.thm}. This parameter will be used to prove non-resonance condition for the final frequency of oscillations;
\\[1mm]
\indent $\bullet$ $\fm\geq \bar\fm\gg1$ is a large parameter measuring how close the analytic potential $Q_{\fm}(y)$ in \eqref{Qm.strategy} is to the singular well potential in \eqref{Qm.intro}. The threshold $\bar\fm$ will be chosen sufficiently large, depending on $\tE$, $\tr$ and $\kappa_0$. Once this large threshold is determined, the value $\fm$ can be arbitrarily fixed once for all.
\\[1mm]
\indent $\bullet$ $S \in \N$ is  large, but finite, regularity that we impose on the local nonlinearities in \eqref{euler.pert.1}. Its value is ultimately fixed when we estimate the Nash-Moser iterations in Section \ref{sez.proof.NMT} and it will depend only on $\kappa_0$ and the loss of derivatives coming from the Diophantine conditions in \eqref{def:DCgt};
\\[1mm]
\indent $\bullet$ $\eta\in [0,\bar\eta]$, with $\bar\eta\ll 1$, is a small parameter parametrizing the modification of the local nonlinearities around the critical values $\psi_{\fm}(\ty_{p,\fm})$, $p=1,..,\kappa_0$. The threshold $\bar\eta$ will depend on $\fm$ and $S$, once the value of $\fm\geq \bar\fm$ is fixed. Ultimately, the value of $\eta$ is linked to the size $\varepsilon$ of the quasi-periodic perturbation $\vf_{\varepsilon}$ in \eqref{form.stream.thm} by \eqref{link eta varepsilon}.


\paragraph{Outline of the paper.}

The rest of the paper is organized as follows.
In Section \ref{functiona.setting}, we recall the functional setting and the basic lemmata that we will use in the following. Section \ref{sec.near_couette} is devoted to the analysis of the shear equilibrium $\psi_{\fm}(y)$. In particular, we estimate the proximity to the Couette flow in Proposition \ref{proxim.couette}, we determine the local nonlinearities for the second order ODE satisfied by the stream function $\psi_{\fm}(y)$ in Theorem \ref{nonlin_eq} and we analyse the spectral properties of the linear operator $\cL_{\fm}= -\pa_{y}^2 + Q_{\fm}(y)$ in Proposition \ref{L_operator}. In Section \ref{sez.nonlin.lin} we set the partial differential equation that we will solve, its Hamiltonian formulation with the action-angle variables and the nonlinear functional map for the Nash-Moser Theorem \ref{NMT}. In Section \ref{deg.KAM} we prove the non-degeneracy and the transversality properties for the negative eigenvalues of the operator $\cL_{\fm}$ that are needed to impose Diophantine non-resonance conditions on them, together with the measure estimates for the final frequencies in Theorem \ref{MEASEST} of Section \ref{subsec:measest}. In Section \ref{sec:approx_inv} we study the approximate inverse of the linearized vector field at any approximate solution of the Nash-Moser nonlinear iteration. We conclude with Section \ref{sez.proof.NMT} with the Proof of Theorem \ref{NMT}, which directly implies the validity of Theorem \ref{main.thm1}.

\paragraph{Acknowledgements.} The authors warmly thank Diego Noja for the useful discussions we had during the study of the problem. The authors also warmly thank Luca Tasin for discussing some issues of algebraic nature and for pointing out some references.
The work of the authors Luca Franzoi and Nader Masmoudi is  supported by Tamkeen under the NYU Abu Dhabi Research Institute grant CG002. The work of the author Nader Masmoudi is also supported by  NSF grant DMS-1716466 . The work of the author Riccardo Montalto is funded by the European Union, ERC STARTING GRANT 2021, "Hamiltonian Dynamics, Normal Forms and Water Waves" (HamDyWWa), Project Number: 101039762. Views and opinions expressed are however those of the authors only and do not necessarily reflect those of the European Union or the European Research Council. Neither the European Union nor the granting authority can be held responsible for them. 
Riccardo Montalto is also supported by INDAM-GNFM.

\section{Functional setting}\label{functiona.setting}

In this paper we consider functions in the following Sobolev space
\begin{align}
	&H^{s,\rho} :=  H^s(\T^{\kappa_0},H_0^\rho([-1,1])) \label{sobolev.sp}\\
	& := \Big\{ u(\bx,y)=\!\sum_{\ell\in\Z^{\kappa_0}} u_{\ell}(y)e^{\im \ell\cdot\bx} \,:\, \| u \|_{s,\rho}^2:=\! \sum_{\ell\in\Z^{\kappa_0}} \braket{\ell}^{2s}\| u_\ell \|_{H_0^\rho([-1,1])}^2<\infty  \Big\}\,, \nonumber
\end{align}
where $\braket{\ell}:=\max\{1,|\ell|\}$ and, recalling \eqref{parity},
\begin{equation}
	H_0^\rho([-1,1]) := \big\{ \,u(y) \in H^\rho([-1,1]) \,:\,  u(-1)=u(1)=0\,, \ u(-y)=u(y)\, \big\}\,.
\end{equation}
For $s \geq \frac{\kappa_0}{2}+1$, $\rho \geq 1$, we have that $ H^{s,\rho}\subset \cC(\T^{\kappa_0}\times[-1,1])$ and that $ H^{s, \rho}$ is an algebra.
\paragraph{Whitney-Sobolev functions.}
We consider  families of Sobolev functions $$\lambda=(\omega,\tA)\mapsto u(\lambda)=u(\lambda;\bx,y) \in H^{s,\rho}$$ 
which are $k_0$-times differentiable in the sense of Whitney with respect to the parameter
$ \lambda = (\omega,\tA) \in F \subset \R^{\kappa_0+1}$
where $F\subset \R^{\kappa_0+1}$ is a closed set. 
We refer to Definition 2.1 in \cite{BBHM}, for the 
definition of Whitney-Sobolev functions.
Given $ \upsilon \in (0,1) $,  by the Whitney extension theorem (e.g. Theorem B.2, \cite{BBHM}), we have the equivalence
\begin{equation}\label{weighted_norm}
	\normk{u}{s,\rho,F}  \sim_{\kappa_0,k_0} 
	{\mathop \sum}_{\abs\alpha\leq k_0} \upsilon^{\abs\alpha} \| \pa_{\lambda}^\alpha \cE_{k}u \|_{L^\infty(\R^{\kappa_0+1},H^{s-|\alpha|,\rho})}\,,
\end{equation}
where $\cE_{k}u$ denotes an extension of $u$ to all the parameter space $\R^{\kappa_{0}+1}$. 
For  simplicity, we  
denote  $ \| \ \|_{s,\rho,F}^{k_0,\upsilon} =  \| \ \|_{s,\rho}^{k_0,\upsilon} $, we use the right hand side of \eqref{weighted_norm} as definition of the norm itself, we denote $\cE_{k}u$ by $u$ and we still denote the function spaces by $H^{s,\rho}$. In particular, we shall deal with functions with Sobolev regularity $s\geq s_0$, where the threshold regularity $s_0$ is chosen as
\begin{equation}\label{s0.def}
	s_0:=s_0(\kappa_0,k_0):=\lfloor\tfrac{\kappa_0}{2}\rfloor+1+k_0 \in\N\,.
\end{equation}

\begin{lem}\label{prod.lemma}
	$(i) $ For all $s\geq s_0$, $\rho \geq 1$ and any $u,v\in H^{s,\rho}$,
	\begin{equation}\label{prod}
		\normk{u v}{s, \rho} \leq C(s,k_0) \normk{u}{s, \rho}\normk{v}{s_0, \rho} + C(s_0,k_0)\normk{u}{s_0, \rho}\normk{v}{s, \rho}\,.
	\end{equation}
	\noindent
	$(ii)$ Let $s \geq s_0$, $a \in H^{s, 1}$, $u \in H^{s, 0}$. Then $a u \in H^{s, 0}$ and 
	\begin{equation}\label{prod2}
		\| a u \|_{s, 0}^{k_0, \upsilon} \lesssim_s \| a \|_{s_0, 1}^{k_0, \upsilon} \| u \|_{s, 0}^{k_0, \upsilon} +  \| a \|_{s, 1}^{k_0, \upsilon} \| u \|_{s_0, 0}^{k_0, \upsilon}\,. 
	\end{equation}
	Similarly if $a \in H^s(\T^{\kappa_0})$ and $u \in H^{s, 0}$, then $a u \in H^{s, 0}$ and
	\begin{equation}\label{prod3}
		\| au \|_{s, 0}^{k_0, \upsilon} \lesssim_s \| a \|_{s_0}^{k_0, \upsilon} \| u \|_{s, 0}^{k_0, \upsilon} +  \| a \|_{s}^{k_0, \upsilon} \| u \|_{s_0, 0}^{k_0, \upsilon}\,.
	\end{equation}
\end{lem}
\begin{proof}
	{\sc Proof of $(i)$.} For parameter independent Sobolev functions, i.e. $k_0=0$, the tame estimates \eqref{prod} follows from standard tame estimates arguments, using the algebra property for functions in $H_0^{\rho}([-1,1])$ for $\rho\geq 1$, see Lemma 2.9 in \cite{BKMdnls}. In general, the tame estimates \eqref{prod} follows as in \cite{BM} and \cite{BM20} with respect to the definition of the weighted norm in \eqref{weighted_norm} and the choice of  $s_0$ in \eqref{s0.def} (here, as in \cite{BM20} and \cite{BFM21}, the norm of $\pa_{\omega}^\alpha u$, $|\alpha|\leq k_0$, is estimated in $H^{s-|\alpha|, \rho}$, whereas in \cite{BM}, \cite{BBHM}, \cite{BFM} is estimated just in $H^s$).
	
	\noindent
	{\sc Proof of $(ii)$.}  To simplify notations, we write $\| \cdot \|_{s, \rho}$ instead of $\| \cdot \|_{s, \rho}^{k_0, \upsilon}$. By expanding $a$ and $u$ in Fourier series with respect to ${\bf x} \in \T^{\kappa_0}$, we have 
	$
	a({\bf x}, y) = \sum_{\ell \in \Z^{\kappa_0}} a_\ell(y) e^{\im \ell \cdot {\bf x}}$ and $ u({\bf x}, y) = \sum_{\ell \in \Z^{\kappa_0}} u_\ell(y) e^{\im \ell \cdot {\bf x}}
	$,
	implying that 
	$$
	a({\bf x}, y) u({\bf x}, y) = \sum_{\ell, \ell' \in \Z^{\kappa_0}} a_{\ell - \ell'}(y) u_{\ell'}(y) e^{\im \ell \cdot {\bf x}}\,. 
	$$
	Therefore
	\begin{equation}\label{macete0}
	\| a u \|_{s, 0}^2 \leq \sum_{\ell \in \Z^{\kappa_0}} \langle \ell \rangle^{2 s} \Big( \sum_{\ell'\in \Z^{\kappa_0}} \| a_{\ell - \ell'} u_{\ell'} \|_{L^2_y} \Big)^2\,.
	\end{equation}
	Using the embedding $H_0^1([- 1, 1]) \subset L^\infty([- 1, 1])$, $\| \cdot \|_{L^\infty_y} \lesssim \| \cdot \|_{H^1_y}$, we have that, for any $\ell , \ell' \in \Z^{\kappa_0}$,
	$$
	\| a_{\ell - \ell'} u_{\ell'}\|_{L^2_y} \leq \| a_{\ell - \ell'} \|_{L^\infty_y} \| u_{\ell'} \|_{L^2_y} \lesssim \| a_{\ell - \ell'} \|_{H^1_y} \| u_{\ell'} \|_{L^2_y}\,.
	$$
	Therefore, the inequality \eqref{macete0} leads to
	\begin{equation}\label{macete1}
	\begin{aligned}
	\| a u \|_{s, 0}^2 & \lesssim \sum_{\ell \in \Z^{\kappa_0}} \langle \ell \rangle^{2 s} \Big( \sum_{\ell'\in \Z^{\kappa_0}} \| a_{\ell - \ell'} \|_{H^1_y} \| u_{\ell'} \|_{L^2_y} \Big)^2 \lesssim_s {\cal I}_1 + {\cal I}_2\,, \\
	{\cal I}_1  & := \sum_{\ell \in \Z^{\kappa_0}} \Big( \sum_{\ell'\in \Z^{\kappa_0}} \langle \ell - \ell' \rangle^{s} \| a_{\ell - \ell'} \|_{H^1_y} \| u_{\ell'} \|_{L^2_y} \Big)^2  \,, \\
	{\cal I}_2 & := \sum_{\ell \in \Z^{\kappa_0}}  \Big( \sum_{\ell'\in \Z^{\kappa_0}} \langle \ell' \rangle^s\| a_{\ell - \ell'} \|_{H^1_y} \| u_{\ell'} \|_{L^2_y} \Big)^2
	\end{aligned}
	\end{equation}
	where we have used the trivial fact $\langle \ell \rangle^s \lesssim_s \langle \ell' \rangle^s + \langle \ell - \ell' \rangle^s$, for any $\ell, \ell' \in \Z^{\kappa_0}$.  By multiplying and dividing by $\langle \ell' \rangle^{s_0}$, using that $\sum_{\ell'} \langle \ell' \rangle^{- 2 s_0} < + \infty$ and by the Cauchy-Schwartz inequality, we estimate the term ${\cal I}_1$ with
	\begin{equation}\label{macete3}
	\begin{aligned}
	{\cal I}_1 & \lesssim \sum_{\ell \in \Z^{\kappa_0}} \sum_{\ell' \in \Z^{\kappa_0}} \langle \ell - \ell' \rangle^{2s} \| a_{\ell - \ell'} \|_{H^1_y}^2 \langle \ell' \rangle^{2 s_0} \| u_{\ell'} \|_{L^2_y}^2  \\
	& \lesssim \sum_{\ell' \in \Z^{\kappa_0}} \langle \ell' \rangle^{2 s_0} \| u_{\ell'} \|_{L^2_y}^2  \sum_{\ell \in \Z^{\kappa_0}} \langle \ell - \ell' \rangle^{2 s} \| a_{\ell - \ell'} \|_{H^1_y}^2  \lesssim \| a \|_{s, 1}^2 \| u \|_{s_0, 0}^2\,. 
	\end{aligned}
	\end{equation}
	By similar arguments, one can show that ${\cal I}_2 \lesssim \| a \|_{s_0, 1}^2 \| u \|_{s, 0}^2$ which imply the claimed interpolation estimate for $\| a u \|_{s, 0}$. In order to estimate $\| au \|_{s, 0}^{k_0, \upsilon}$ one has to estimate for any $\alpha \in \N^{\kappa_0}$ with $|\alpha| \leq k_0$, 
	$$
	\| \partial_\omega^\alpha (a u) \|_{s - |\alpha|, 0} \lesssim_\alpha \sum_{\alpha_1 + \alpha_2 = \alpha} \| (\partial_\omega^{\alpha_1}a)(\partial_\omega^{\alpha_2} u) \|_{s - |\alpha|, 0} 
	$$
	and every term of the latter sum is estimated as above. 
	
	\noindent
	The estimate \eqref{prod3} is proved similarly to \eqref{prod2} (it is actually easier since $a$ does not depend on $y$). 
	\end{proof}

For any $K>0$, we define the smoothing projections 
\begin{equation}\label{pro:N}
	u(\bx,y) =\! \! \sum_{\ell\in\Z^{\kappa_0}} u_\ell(y) e^{\im\,\ell\cdot\bx} \mapsto \Pi_{K}u(\bx,y):=\!\! \sum_{|\ell|\leq K} u_\ell(y) e^{\im\,\ell\cdot\bx}\,, \quad \Pi_{K}^\perp := {\rm Id} - \Pi_{K}\,.
\end{equation}
The following  estimates hold for the smoothing operators defined 
in \eqref{pro:N}:
\begin{equation}\label{SM12}
	\normk{\Pi_K u}{s, \rho}  \leq K^\alpha \normk{u}{s-\alpha, \rho}\,, \  0\leq \alpha\leq s \,, \quad 
	\normk{\Pi_K^\perp u}{s, \rho}  \leq K^{-\alpha} \normk{u}{s+\alpha, \rho}\,, \ \alpha\geq 0 \,.
\end{equation}
We also recall the standard Moser tame estimate for the nonlinear composition operator
$$
u(\bx,y)\mapsto \tf(u)(\bx,y) :=f(\bx,y,u(\bx,y))\,.
$$
For the purposes of this paper, we state this result in the case of finite regularity of the nonlinear function. 
\begin{lem}{\bf (Composition operator)}\label{compo_moser}
	Let $s \geq s_0$, $\rho \in \N$. Then there exists $\sigma = \sigma( \rho) > 0$ such that for any $f\in\cC^{s + \sigma}(\T^{\kappa_0}\times [-1,1]\times\R,\R)$, if $u(\lambda)\in H^{s,\rho}$ is a family of Sobolev functions satisfying $\normk{u}{s_0, \rho}\leq 1$, then
	$$
	\normk{\tf(u)}{s, \rho}\leq C(s,k_0) \| f \|_{{\cal C}^{s + \sigma}}\big( 1+\normk{u}{s, \rho} \big) \,.
	$$
	If $ f(\bx,y, 0) = 0 $, then 
	$	  \normk{\tf(u)}{s, \rho}\leq C(s,k_0)  \| f \|_{{\cal C}^{s + \sigma}} \normk{u}{s, \rho} $. Moreover, if $f\in \cC^\infty$, then the same result holds for any $s\geq s_0$.
\end{lem}
\begin{proof}
	See e.g.  Lemmata 2.14, 2.15 in \cite{BKMdnls} and Lemma 2.6 in \cite{BBHM}. The proof relies on the multilinear Leibniz rule, on the Faà di Bruno formula, on the tame estimates in \eqref{prod} and on interpolation inequalities.
\end{proof}

\paragraph{Diophantine equation.} 
If $\omega $ is a Diophantine vector in $ \tD\tC(\upsilon,\tau)$,
defined by
\begin{equation}\label{def:DCgt}
	\tD\tC(\upsilon,\tau):=\Big\{ \omega\in\R^{\kappa_0}\, : \, |\omega\cdot\ell|\geq \upsilon\braket{\ell}^{-\tau} \ \forall\,\ell\in\Z^{\kappa_0}\setminus\{0\}  \Big\}\,,
\end{equation}
then the equation $\omega\cdot \pa_\bx v = u$, where $u(\bx)$ has zero average with respect to $\bx\in\T^{\kappa_0}$, has the periodic solution
$$
(\omega\cdot\pa_\bx)^{-1} u(\bx) := \sum_{\ell\in\Z^{\kappa_0}\setminus\{0\}} \frac{u_{\ell}}{\im\,\omega\cdot \ell} e^{\im\ell\cdot\bx} \,.
$$
For $F\subseteq \tD\tC(\upsilon,\tau)$, one has
\begin{equation}\label{lem:diopha.eq}
	\normk{(\omega\cdot\pa_\bx)^{-1}u}{s,\rho,F} \leq C(k_0)\upsilon^{-1}\normk{u}{s+\mu,\rho,F}\,, \quad  \mu:=k_0+\tau(k_0+1) \,.
\end{equation}

\paragraph{Reversible and reversibility preserving conditions.}

In the next sections we will consider (spatial) reversible and reversibility maps in order to preserve the symmetry of the solutions \eqref{parity}. To do so, let $\cS$ be the involution acting on the real variables $\zeta=(\zeta_1,\zeta_2)\in\R^2$ defined by
\begin{equation}\label{invo.std}
	\cS: \begin{pmatrix}
		\zeta_1(y) \\ \zeta_2(y) 
	\end{pmatrix}\mapsto \begin{pmatrix}
		\zeta_1(-y) \\ -\zeta_2(-y)
	\end{pmatrix}\,.
\end{equation}
In action-angle variables $\zeta = (\theta,I,w) \in \T^{\kappa_0}\times \R^{\kappa_0}\times \R^2$, which will be introduced in \eqref{aa_coord}, we consider the following involution
\begin{equation}\label{invo.aa}
	\ora{\cS} : (\theta,I,z) \mapsto (-\theta,I,\cS z)\,.
\end{equation}

Let $\bx\in\T^{\kappa_0}$. A function $\zeta(\bx,\,\cdot\,)$ is called \emph{(spatial) reversible} if $\cS\zeta(\bx,\,\cdot\,)= \zeta(-\bx,\,\cdot\,)$ and \emph{anti-reversible} if  $-\cS\zeta(\bx,\,\cdot\,)= \zeta(-\bx,\,\cdot\,)$. The same definition holds in action-angle variables $(\theta,I,z)$ with the involution $\cS$ in \eqref{invo.std} replaced by $\ora{\cS}$ in \eqref{invo.aa}.

A $\bx$-dependent family of operators $\cR(\bx)$, $\bx \in \T^{\kappa_0}$, is \emph{reversible} if $\cR(-\bx) \circ \cS = - \cS \circ \cR(\bx)$ for all $\bx\in\T^{\kappa_0}$ and it is \emph{reversibility preserving} if $\cR(-\bx) \circ \cS = \cS \circ \cR(\bx)$ for all $\bx\in\T^{\kappa_0}$. A reversibility preserving operator maps reversible, respectively anti-reversible, functions into reversible, respectively anti-reversible, functions, see Lemma 3.22 in \cite{BFM}. We remark also that, if $X$ is a reversible vector field, namely $X\circ \cS = - \cS\circ X$, and $\zeta(\bx,\,\cdot\,)$ is a reversible function, then the linearized operator $\di_\zeta X(\zeta(\bx,\,\cdot\,))$ is reversible, see e.g. Lemma 3.22 in \cite{BFM}.

\section{A shear equilibrium close to Couette with oscillations}\label{sec.near_couette}

The goal of this quite lengthy section is to construct and analyse the shear flow $(\psi_{\fm}'(y),0)$ with stream function $\psi_{\fm}(y)$ that in an equilibrium configuration from which we will bifurcate the space quasi-periodic solutions of the stationary Euler equations. The properties of the stream function $\psi_{\fm}(y)$ are essentially dictated by the analytic potential function $Q_{\fm}(y)$ in \eqref{Qgamma}, which is suitably engineered to meet regularity properties, smallness estimates and parameter dependences. First, in Section \ref{subsec.Q.couette} we show that the stream function $\psi_{\fm}(y)$ is close to the Couette stream function $\psi_{\rm cou}(y):=\tfrac12 y^2$ with respect to the $H^3$-norm, see Proposition \ref{proxim.couette} . In Section \ref{subsec.FF} we prove Theorem \ref{nonlin_eq}, which states the existence of local nonlinearities such that $\psi_{\fm}(y)$ solves the second-order nonlinear ODE \eqref{ell_psi_gamma}. We conclude in Section \ref{subsec.spectrum} with the spectral analysis of the linear operator $\cL_{\fm}:=-\pa_y^2 +Q_{\fm}(y)$ in Proposition \eqref{L_operator}

\subsection{The potential $Q_{\fm}(y)$ and the proximity to the Couette flow}\label{subsec.Q.couette}

We consider a shear flow 
$(u,v)=(\psi_{\fm}'(y),0)$
on $\R\times[-1,1]$, where $\psi_{\fm}'(y)$ is the odd solution of the second-order ODE
\begin{equation}\label{Qgamma}
	\psi_{\fm}'''(y) = Q_\fm(y) \psi_{\fm}'(y)\,, \quad Q_{\fm} (y):= -\tE^2\Big(h_{\fm}\big( \tfrac{y}{\tr} \big) +g_{\fm,S}\big(\tfrac{y}{\tr}\big) \Big)\,,
\end{equation}
with $\tr\in (0,1) $, $\tE>0$, $\fm,S \in \N$ and the function $h_{\fm}(z)$ given by
\begin{equation}\label{function.h}
	h_{\fm}(z):= \Big( \Big(\frac{\cosh(z)}{\cosh(1)}\Big)^\fm +1 \Big)^{-1}\,.
\end{equation}
The function $g_\fm(\tfrac{y}{\tr})$ is a polynomial corrector which is chosen to satisfy the following property.
\begin{lem}\label{lemma zero odd der}
	Let $(\ty_{p,\fm})_{\fm\in\N}$, $p=1,...,\kappa_0$ be arbitrary sequences in $(0,\tr)$ converging to $\ty_{p,\infty}= \frac{p\pi}{\tE}$. Then, for any $S\in\N$ and for $\fm \gg 1$ sufficiently large, there exists a polynomial function $g_{\fm,S}(z)$ even in $z:=\frac{y}{\tr}$ and of degree $d(S,\kappa_0)\in\N$ such that the analytic potential $Q_{\fm}(y)$ satisfies the following finitely many conditions:
	\begin{equation}\label{zero odd deriv}
		\pa_y^{2n-1} Q_\fm(\pm\ty_{p,\fm}) = 0 \quad \forall\, n=1,...,S\,, \quad p=1,..., \kappa_0\,.
	\end{equation}
	Moreover, for any $n\in\N_0$, the following estimate holds
	\begin{equation}\label{bound.corrector}
		\sup_{y\in[-1,1]}|\pa_{z}^n g_{\fm,S}(\tfrac{y}{\tr}) | \lesssim_{n}  \sum_{p=1}^{\kappa_0}\sum_{k=1}^{S} |h_{\fm}^{(2k-1)}(\ty_{j,\fm})|\,. 
	\end{equation}
\end{lem}

\begin{proof}
	We use a classical Hermite interpolation argument, also referred as Lagrangian interpolation with derivatives. Let $\wt\ty_{p,\fm}:=\frac{\ty_{p,\fm}}{\tr}$, $p=1,...,\kappa_0$. We search for a polynomial $g_{\fm,S}(z)$ of the form
	\begin{equation}\label{struttura.g}
		g_{\fm,S}(z) = \sum_{j=\pm1,...,\pm\kappa_0} F_{j}(z)\,, \quad \text{with} \quad 	F_{j}^{(2n-1)}(\wt\ty_{j',\fm}) = 0 \quad \forall\, j'\neq j\,.
	\end{equation}
	Let $j\in \{\pm1,...,\pm \kappa_0\}$ be fixed. We construct $F_{j}(z)$ as the linear combination
	\begin{equation}\label{befana1}
		F_{j}(z)= \sum_{k=1}^{S} a_{j,2k-1} f_{j,2k-1}(z)\,, \quad  f_{j,2k-1}(z) = \frac{(z-\wt\ty_{j,\fm})^{2k-1}}{(2k-1)!} \whf_{j}(z)\,,
	\end{equation}
	where, for $T>1$ arbitrarily large, we define
	\begin{equation}\label{befana2}
		\whf_{j}(z):= (z^2-T^2)^{2S} \prod_{j'=\pm1,...,\pm \kappa_0\atop j'\neq j} (z-\wt\ty_{j',\fm})^{2S}
	\end{equation}
	The structure of \eqref{befana1}-\eqref{befana2} allows both to impose the condition in \eqref{struttura.g} and to control later the estimates \eqref{bound.corrector} by choosing $T\gg1$.
	We search now for the coefficients $a_{j,1}, a_{j,3},...,a_{j,2S-1}$ so that the conditions in \eqref{zero odd deriv} are satisfied, which, by \eqref{Qgamma} and \eqref{struttura.g}, amounts to ask
	\begin{equation}\label{befana3}
		F_{j}^{(2n-1)}(\wt\ty_{j,\fm}) = c_{j,2n-1} := - h_{\fm}^{(2n-1)}(\wt\ty_{j,\fm})\,, \quad 
		\forall \,n=1,...,S\,.
	\end{equation}
	By differentiating \eqref{befana1} at each order $2n-1$, $n=1,...,S$, solving \eqref{befana3} amounts to solving the linear system
	\begin{equation}\label{befana.system}
		\begin{pmatrix}
			f_{j,1}^{(1)}(\wt\ty_{j,\fm}) & f_{j,3}^{(1)}(\wt\ty_{j,\fm}) & \cdots & f_{j,2S-1}^{(1)}(\wt\ty_{j,\fm}) \\
			f_{j,1}^{(3)}(\wt\ty_{j,\fm}) & f_{j,3}^{(3)}(\wt\ty_{j,\fm}) & \cdot & f_{j,2S-1}^{(3)}(\wt\ty_{j,\fm}) \\
			\vdots & \vdots & \ddots & \vdots \\
			f_{j,1}^{(2S-1)}(\wt\ty_{j,\fm}) & f_{j,3}^{(2S-1)}(\wt\ty_{j,\fm}) & \cdots & f_{j,2S-1}^{(2S-1)}(\wt\ty_{j,\fm})
		\end{pmatrix}\begin{pmatrix}
			a_{j,1} \\ a_{j,3} \\ \vdots \\ a_{j,2S-1}
		\end{pmatrix} = \begin{pmatrix}
			c_{j,1} \\ c_{j,3} \\ \vdots \\ c_{j,2S-1}
		\end{pmatrix}\,.
	\end{equation}
	The system above is lower triangular. Indeed, by \eqref{befana1} and \eqref{befana2}, we have
	\begin{equation}
		f_{j,2n-1}^{(2n-1)}(\wt\ty_{j,\fm}) = \whf_{j}(\wt\ty_{j,\fm}) \quad \forall\, n=1,...,S\,, \quad f_{j,2k-1}^{(2n-1)}(\wt\ty_{j,\fm})=0 \quad \forall\,k>n\,.
	\end{equation}
	We solve directly the system \eqref{befana3} and we get
	\begin{equation}\label{befana4}
		\begin{aligned}
			a_{j,1} & = \frac{c_{j,1}}{\whf_{j}(\wt\ty_{j,\fm})}\,, \\
			a_{j,2k-1} & = \frac{1}{\whf_{j}(\wt\ty_{j,\fm})} \Big(  c_{j,2k-1} - \sum_{n=1}^{k-1} a_{j,2n-1} f_{j,2n-1}^{(2k-1)}(\wt\ty_{j,\fm}) \Big)\,, \quad  k=2,...,S\,.
		\end{aligned}
	\end{equation}
	We finally show the estimates in \eqref{bound.corrector}. By \eqref{struttura.g}-\eqref{befana3}, taking $T\gg \frac{1}{\tr} >1 $ sufficiently large, the estimates for the coefficients in \eqref{befana4} can be made arbitrarily small and  we have, for any $n\in\N_0$,
	\begin{equation}
		\sup_{z\in[-\frac{1}{\tr},\frac{1}{\tr}]}|\pa_{z}^n g_{\fm,S}(z) | \lesssim \sum_{p=1}^{\kappa_0}\sum_{k=1}^{S} |c_{\pm p, 2k-1}| = \sum_{p=1}^{\kappa_0}\sum_{k=1}^{S} |h_{\fm}^{(2n-1)}(\wt\ty_{j,\fm})|\,. 
	\end{equation}
	This concludes the proof of the claim.
\end{proof}

The potential function $Q_\fm(y) = Q_{\fm}(\tE,\tr;y) $ is analytic in all its entries and, in the limit $\fm\to+\infty$, it approaches the singular potential
\begin{equation}\label{Qinfty}
	Q_\infty(y) = Q_\infty(\tE,\tr;y) := \begin{cases}
		0 & |y| >\tr \,, \\
		-\tE^2 & |y|< \tr \,.
	\end{cases}
\end{equation}

\begin{lem}\label{lemma QmQinfty}
	{\bf (Estimates for  $Q_{\fm}(y)$).}
		We have
		\begin{equation}\label{uniform.Qm}
			\sup _{\fm \gg 1}\| Q_\fm\|_{L^\infty([-1,1])} \lesssim \| Q_\infty \|_{L^\infty([-1,1])} \lesssim \tE^2\,.
		\end{equation}
		Moreover, for any fixed $\gamma>0$ sufficiently small, we have, for some  constant $C>1$ and for any $n\in\N_0$,
\begin{equation}\label{focus.est}
		\sup_{y\in  [-1,1] \atop y\notin B_\gamma(\tr)\cup B_\gamma(-\tr)}   \! | \pa_{y}^{n}(Q_{\fm}(y)-Q_\infty(y))| \leq \tE^2 C^n \Big( \frac{\fm^n}{\tr^n} + \sum_{j=1}^{S} \frac{\fm^{2k-1}}{\tr^{2k-1}} \Big) e^{-\frac{\fm}{2}\gamma}\,. 
\end{equation}
		As a consequence, as $\fm \to\infty$, we have that $| \pa_y^n(Q_\fm(y) -Q_\infty(y))| \to 0$  for any $n\in \N_0$,  uniformly in $y\in \R\setminus(B_\gamma(\tr)\cup B_\gamma(-\tr))$, and $\| Q_\fm - Q_\infty \|_{L^p([-1,1])} \to 0$  for any $p\in[1,\infty)$.
\end{lem}
\begin{proof}
 Recalling \eqref{Qgamma}, \eqref{function.h} and \eqref{Qinfty}, we write
 \begin{equation}\label{tilda.Qs}
 	\begin{aligned}
 		&Q_\fm(y)=-\tE^2\wtQ_{\fm}\big(\tfrac{y}{\tr}\big)\,, \quad Q_\infty(y)=-\tE^2\wtQ_{\infty}\big(\tfrac{y}{\tr}\big)\,, \quad \text{where}\\
 		&\wtQ_{\fm}(z):= h_{\fm}(z)+g_{\fm,S}(z) \,, \quad \wtQ_{\infty}(z):=\begin{cases}
 			0 & |z| > 1 \,, \\
 			1 & |z|< 1 \,.
 		\end{cases}
 	\end{aligned}
 \end{equation}
 
 We start by proving \eqref{focus.est}. First, by the estimate \eqref{bound.corrector} in Lemma \eqref{lemma zero odd der} and the fact that $| \ty_{p,\fm}| < \tr$ for any $p=1,...,\kappa_{0}$, we note that, for any $\gamma>0$ sufficiently small,
 \begin{equation}\label{fame}
 	\sup_{y\in[-1,1]}|g_{\fm,S}^{(n)}(\tfrac{y}{\tr}) | \lesssim_{n}  \sum_{k=1}^{S} \sup_{|y| \leq \tr -\gamma}  |h_{\fm}^{(2k-1)}\big(\tfrac{y}{\tr}\big)|\,, \quad  \forall\, n\in\N_0 \,.
 \end{equation}    
It implies that $g_{\fm}\big( \tfrac{y}{\tr} \big)$ and its derivatives are bounded on $[-1,1]$ by finitely many derivatives of $h_{\fm}\big( \tfrac{y}{\tr}\big)$ on $[-\tr+\gamma,\tr-\gamma]$. Also, because $g_{\fm,S}$ is a polynomial, its derivatives of order $n$ will identically vanish for any $n\in \N$ large enough. 
By \eqref{tilda.Qs}, we  estimate
\begin{equation}\label{colapesce}
	\begin{aligned}
		\sup_{y\in  [-1,1] \atop y\notin B_\gamma(\tr)\cup B_\gamma(-\tr)} \big| \wtQ_{\fm}^{(n)}\big(\tfrac{y}{\tr}\big) - \wtQ_{\infty}^{(n)}\big(\tfrac{y}{\tr}\big)\big| \leq  \tJ_{1} + \tJ_{2}\,, 
	\end{aligned}
\end{equation}
where
\begin{equation}\label{dimartino}
	\tJ_{1} :=  \sup_{y\in  [-1,1] \atop y\notin B_\gamma(\tr)\cup B_\gamma(-\tr)} \big| h_{\fm}^{(n)}\big(\tfrac{y}{\tr}\big) - \wtQ_{\infty}^{(n)}\big(\tfrac{y}{\tr}\big)\big| \,, \quad \tJ_{2} :=\sup_{y\in  [-1,1] \atop y\notin B_\gamma(\tr)\cup B_\gamma(-\tr)} \big| g_{\fm,S}^{(n)}\big(\tfrac{y}{\tr}\big) \big|\,.
\end{equation}
By direct computations on the derivatives of $h_{\fm}(z)$ in \eqref{bound.corrector} , 
 we obtain the estimate, for some constant $\wtC>1$,
    \begin{equation}\label{sonno}
		\begin{aligned}
				| h_{\fm}^{(n)}(z)-\wtQ_{\infty}^{(n)}(z)| & \leq \wtC^n \fm^n  \max\Big\{ \Big(\tfrac{\cosh(1-\gamma)}{\cosh(1)}\Big)^{\fm}, \Big(\tfrac{\cosh(1)}{\cosh(1+\gamma)}\Big)^{\fm} \Big\} \\
				& \leq \wtC^n \fm^n e^{-\frac{\fm}{2}\gamma}\,, \quad  \forall \,|z| \notin B_\gamma(1)\,;
		\end{aligned}
\end{equation}
here we used the following estimates
\begin{equation}
	\begin{aligned}
	 &	\frac{f^n}{(f+1)^{n+1}} \leq \begin{cases}
			f \,, & 0\leq f < 1 \,,\\
			f^{-1}\,, & f >1 \,,
		\end{cases}\,, \ n\in \N_0\,, \\
	& \frac{\cosh(1-\gamma)}{\cosh(1)}, \frac{\cosh(1)}{\cosh(1+\gamma)} \leq e^{-\gamma/2}\,, \quad \forall\, 0 \leq \gamma \ll 1\,.
	\end{aligned}
\end{equation}
We deduce that 
\begin{equation}\label{tananai1}
	\tJ_{1}\leq \wtC^n \frac{\fm^{n}}{\tr^{n}}e^{\frac{\fm}{2}\gamma}\,.
\end{equation}
On the other hand, by \eqref{fame} and \eqref{sonno}, we estimate $\tA_{2}$ by 
\begin{equation}\label{tananai2}
	\tJ_{2} \leq \sum_{k=1}^{S}  \sup_{|y|\leq \tr-\gamma} \big| h_{\fm}^{(n)}\big(\tfrac{y}{\tr}\big) - \wtQ_{\infty}^{(n)}\big(\tfrac{y}{\tr}\big)\big| \leq  \sum_{k=1}^{S} \wtC^{2k-1} \frac{\fm^{2k-1}}{\tr^{2k-1}} e^{-\frac{\fm}{2}\gamma}\,.
\end{equation}
 Therefore, by \eqref{tilda.Qs}, \eqref{colapesce}, \eqref{dimartino}, \eqref{tananai1}, \eqref{tananai2}, we conclude that the estimate \eqref{focus.est} holds for any $n\in\N_0$, with $C:=\wtC^{2S-1}$.

 We now prove \eqref{uniform.Qm}. We have that
\begin{equation}
	\begin{aligned}
		& h_{\fm}'(z) = - \fm \Big( \Big(\frac{\cosh(z)}{\cosh(1)}\Big)^\fm +1 \Big)^{-2} \Big(\frac{\cosh(z)}{\cosh(1)}\Big)^\fm \tanh(z) < 0 \quad \forall\, z\geq 0\,, \\
		& h_{\fm}(0) = \Big( \Big(\frac{\cosh(0)}{\cosh(1)}\Big)^\fm +1 \Big)^{-1} < 1\,.
	\end{aligned}
\end{equation}
It implies that $\| h_{\fm} \|_{L^{\infty}(\R)} \leq \| \wtQ_{\infty}\|_{L^\infty(\R)}= 1$. The estimate $\| g_{\fm,S}(\frac{\cdot}{\tr})\|_{L^{\infty}[-1,1]} \lesssim \| \wtQ_{\infty}(\frac{\cdot}{\tr})\|_{L^{\infty}[-1,1]} $ follows by \eqref{fame} with $n=0$ and \eqref{sonno} for $|z|\leq 1-\gamma$. Therefore, together with \eqref{tilda.Qs}, we deduce \eqref{uniform.Qm}.

The claim for the $L^{p}$-convergence follows by the pointwise convergence of $Q_{\fm}(y)$ to $Q_{\infty}(y)$ for any $y\in [-1,1]\setminus\{ \pm \tr\}$,  the estimate $\| Q_\fm\|_{L^\infty([-1,1])} \lesssim \| Q_\infty \|_{L^\infty([-1,1])}\lesssim\tE^2$ uniformly in $\fm \gg 1$ and the integrability of $Q_{\infty}$ in the compact interval $[-1,1]$.
\end{proof}


Among all the possible solutions of \eqref{Qgamma}, we look for those that are odd on $\R$ and satisfy the "Couette condition" $\psi_{\fm}'(y)\sim y$ as $|y|\to \infty$. In particular, we prove that there exists a solution that, on compact sets excluding the singular points $y=\pm \tr$ of the potential $Q_{\infty}(y)$ in \eqref{Qinfty}, approaches uniformly
\begin{equation}\label{compact.beha}
    \psi_{\fm}'(y) \to \begin{cases}
        y -A \,\sgn(y)  & y \in K_{\rm out} \subset\subset \R\setminus [-\tr,\tr]\,,\\
        B \sin(\tE y) & y \in K_{\rm in} \subset\subset (-\tr,\tr)\,,
    \end{cases} \quad  \text{as } \ \ \fm\to \infty\,.
\end{equation}

The constants $A,B\in\R$ are actually free and we fix them in Proposition \ref{proxim.couette}.
First, we prove the uniform limit in \eqref{compact.beha} for the stream function $\psi_{\fm}(y)$ together with its derivatives.
\begin{lem}\label{lemma_approx_comp}
	For any $T\geq 1$,
$\gamma\in (0,\frac12)$  and $n\in\N_0$, we have
		\begin{equation}\label{compact.limit}
			\begin{aligned}
			    & \sup_{ |y| \leq \tr-\gamma }  |\pa_y^n(\psi_{\fm}(y) + \tfrac{B}{\tE} \cos(\tE y) ) | \to 0  \quad \text{ as } \ \ \fm \to \infty\,, \\
			    & \sup_{ |y| \in [\tr+\gamma,T]  }  |\pa_y^n(\psi_{\fm}(y)-\tfrac12(y-A\,\sgn(y))^2 ) | \to 0  \quad \text{ as } \ \  \fm \to \infty\,.
			\end{aligned}
		\end{equation}
\end{lem}
\begin{proof}
 We start with the first limit in \eqref{compact.limit}, with $|y|\leq \tr-\gamma$. Once the claim is proved for $n\geq 1$, then the claim for $n=0$ follows by integration. Thus, we start to prove the claim for $n=1,2$.
We write the second order equations $u''(y) = Q_\fm(y) u(y)$ and $u''(y) = Q_\infty(y) u(y)$ as first order systems, namely 
\begin{equation}\label{bla bla 0}
\begin{aligned}
& \Phi'(y) = A_\fm (y) \Phi(y), \quad \Phi'(y) = A_\infty (y) \Phi(y), \\
& \Phi := \begin{pmatrix}
u \\
u'
\end{pmatrix}, \quad A_\fm(y) := \begin{pmatrix}
0 & 1\\
Q_\fm(y) & 0
\end{pmatrix}, \quad A_\infty(y) := \begin{pmatrix}
0 & 1\\
Q_\infty(y) & 0
\end{pmatrix}\,.
\end{aligned}
\end{equation} 
Let $\Phi_{\fm} (y):= (\psi_{\fm}'(y), \psi_{\fm}''(y))$, $ \Phi_{\infty,{\rm in}}(y) :=( B\sin(\tE y),B\tE \cos(\tE y))$ (recall \eqref{Qgamma}-\eqref{compact.beha}) and $\cF_{\fm} := \Phi_{\fm}(y) - \Phi_{\infty,{\rm in}}(y)$.  
We have
\begin{equation}
	\begin{aligned}
		\cF_{\fm}'(y) & = \Phi_{\fm}'(y) - \Phi_{\infty,{\rm in}}'(y) = A_\fm(y) \Phi_{\fm}(y) - A_\infty(y) \Phi_{\infty,{\rm in}}(y) \\
		& = A_{\fm}(y) \cF_{\fm}(y) + \big( A_\fm(y) - A_\infty(y) \big) \Phi_{\infty,{\rm in}}(y)\,,
	\end{aligned}
\end{equation}
implying that $\cF_{\fm}$ solves the Cauchy problem
\begin{equation}\label{eq cal F gamma}
\begin{aligned}
& \cF_{\fm}'(y) = A_\fm(y)\cF_{\fm}(y) + \cR_{\fm}(y)\,, \quad \cF_{\fm}(0) = \Phi_{\fm}(0)-\Phi_{\infty,{\rm in}}(0)\,, \\
& \text{where} \quad  \cR_{\fm}(y) :=  \big( A_\fm(y) - A_\infty(y) \big) \Phi_{\infty,{\rm in}}(y)\,. 
\end{aligned}
\end{equation}
Note that, by \eqref{Qgamma}, \eqref{Qinfty}, \eqref{compact.beha} and the oddness of $\psi_{\fm}'(y)$, we get $\psi_{\fm}'''(0) + B\tE^2 \sin(\tE y) \to 0$ as $\tm\to+\infty$. By integration, we deduce that $\cF_{\fm}(0)\to 0$ as $\fm\to+\infty$.
For $|y|\leq \tr-\gamma$, one has that 
\begin{equation}
	\cF_{\fm}(y) = \cF_{\fm}(0) +  \int_0^y A_\fm(z) \cF_{\fm}(z) \wrt z  + \int_0^y \cR_{\fm}(z) \wrt z \,.
\end{equation}
By the definitions of $A_\fm$, $A_\infty$ in \eqref{bla bla 0} and by Lemma \ref{lemma QmQinfty}, one has $ \| A_\fm \|_{L^\infty} \leq \tE^2 $ and
\begin{equation}
	\begin{aligned}
		\Big| \int_0^y \cR_{\fm}(z) \wrt z \Big| & \leq 
		 \int_{-\tr+\gamma}^{\tr-\gamma} |A_\fm(z) - A_\infty(z) |  |\Phi_{\infty,{\rm in}}(z)| \wrt z \\ 
		& \leq  \| \Phi_{\infty,{\rm in}} \|_{L^\infty(B_{\tr-\gamma}(0))} \| Q_\fm - Q_\infty \|_{L^1(B_{\tr-\gamma}(0))}\,.
	\end{aligned}
\end{equation}
Therefore, for any $|y|\leq \tr-\gamma$, one obtains the estimate 
\begin{equation}
	\begin{aligned}
		|\cF_{\fm}(y)| \leq |\cF_{\fm}(0)| +  \| \Phi_{\infty,{\rm in}} \|_{L^\infty(B_{\tr-\gamma}(0))} \| Q_\fm - Q_\infty \|_{L^1(B_{\tr-\gamma}(0))} + \tE^2 \int_0^y |\cF_{\fm}(z) | \wrt z \,.
	\end{aligned}
\end{equation}
By Gronwall inequality, Lemma \ref{lemma QmQinfty} and $\cF_{\fm}(0)\to 0$ as $\fm \to 0$, we get
\begin{equation}\label{chiusa.gronw.1}
	\sup_{|y|\leq \tr-\gamma} | \cF_{\fm}(y)| \leq e^{\tE^2 (\tr-\gamma) } \big( |\cF_{\fm}(0)|  + \tE^2 \| Q_\fm - Q_\infty \|_{L^1} \big) \to 0 \quad \text{as } \ \fm \to 0\,,
\end{equation}
which proves the first claim in \eqref{compact.limit} for $n=1,2$. For $n\geq 3$, 
 using the same previous notations, we have that, for any $\ell\in\N$,  $\cF_{\fm}^{(\ell)}(y)$ solves iteratively the Cauchy problem
\begin{equation}\label{eq cFell fm}
	\begin{aligned}
		& (\cF_{\fm}^{(\ell)})'(y) = A_\fm(y)\cF_{\fm}^{(\ell)}(y) + \cR_{\fm,\ell}(y)\,, \quad \cF_{\fm}^{(\ell)}(0) = \Phi_{\fm}^{(\ell)}(0)-\Phi_{\infty,{\rm in}}^{(\ell)}(0)\,, \\
		& \text{where} \quad  \cR_{\fm,\ell}(y) := A_\fm'(y) \cF_{\fm}^{(\ell-1)}(y) + \cR_{\fm,\ell-1}'(y)\,, \quad \cF_{\fm}^{(0)}(y):= \cF_{\fm}(y) \,. 
	\end{aligned}
\end{equation}
Then, the similar conclusion in \eqref{chiusa.gronw.1} holds with the same arguments of before. 

Finally, when $\tr+\gamma \leq |y| \leq T$, the second claim in \eqref{compact.limit} is proved with the same scheme as before, replacing $\Phi_{\infty,{\rm in}}(y)$ with $\Phi_{\infty,{\rm out}}(y):=(y-A\,\sgn(y),1)$, setting as initial data $\cF_{\fm}^{(\ell)}(\tr+\gamma) = \cF_{\fm}^{(\ell)}(-(\tr+\gamma)) =  \Phi_{\fm}^{(\ell)}(y+\tr)-\Phi_{\infty,{\rm out}}^{(\ell)}(\tr+\gamma)$ for any $\ell\in\N_0$ and having $\|A_{\fm} \|_{L^\infty(\R\setminus B_{\tr+\gamma}(0))}\leq \tE^2 e^{-\tm\gamma}$ by Lemma \ref{lemma QmQinfty}. We therefore omit further details.
\end{proof}

We can now prove the proximity result for the shear flow $(\psi_{\fm}'(y),0)$ to the Couette flow $(y,0)$ in the Sobolev regularity $H^1$ (in vorticity space). The main tool is the approximation Lemma \ref{lemma_approx_comp}. To this end, we fix, \emph{independently of $\fm\gg 1$}, the constants
\begin{equation}\label{choice.AB}
    A:= \tr-\frac{1}{\tE^2}\,, \quad B:= \frac{1}{\tE^2 \cos(\tE\tr)}\stackrel{\eqref{constrain}}{=}= \frac{1}{\tE^2 \sin(\tE\tr)}\,,
\end{equation}
and the small radius $\gamma \ll \tr$
\begin{equation}\label{def.smaller.scale}
    \gamma = \gamma(\tr)= \tr^{5}\,.
\end{equation}

\begin{prop}\label{proxim.couette}
	{\bf  (Proximity to the Couette flow).}
	There exists $\bar\fm =\bar\fm(\tr) \gg  1$ large enough such that, for $\fm \geq \bar \fm$, there exists a stream function  $\psi_{\fm}(y)$, with $\psi_{\fm}'(y)$ odd solution of \eqref{Qgamma}, that is close to the stream function of the Couette flow $\psi_{\rm cou}(y):=\tfrac12 y^2$ in the $H^3$-norm, with estimate
		\begin{equation}\label{proxim.estim}
		\| \psi_{\fm} -\psi_{\rm cou} \|_{H^3[-1,1]} \lesssim \sqrt\tr\,.
	\end{equation}
\end{prop}
\begin{proof}
 By interpolation for Sobolev norms, it is enough to prove estimates for $\| \psi_{\fm}- \psi_{\rm cou} \|_{L^2[-1,1]}$ and for $\| \psi_{\fm}''' \|_{L^2[-1,1]}$.
We start with the estimate for $\psi_{\fm}-\psi_{\rm cou}$. By parity of the integrand, we split
\begin{equation}\label{der0.split}
\begin{aligned}
     \| \psi_{\fm} - \psi_{\rm cou} \|_{L^2[-1,1]}^2 & = 2\int_{\tr+\gamma}^{1}\big| \psi_{\fm}(y) - \tfrac12 y^2 \big|^2 \wrt y + 2\int_{B_{\gamma}(\tr)}\big| \psi_{\fm}(y) - \tfrac12 y^2 \big|^2 \wrt y \\
    & + 2\int_{0}^{\tr-\gamma}\big| \psi_{\fm}(y) - \tfrac12 y^2 \big|^2 \wrt y =: I_{0,{\rm out}} + I_{0,\gamma} + I_{0,{\rm in}}\,.
\end{aligned}
\end{equation}
On the compact intervals $[\tr+\gamma,1]$ and $[0,\tr-\gamma]$, we use the approximations in Lemma \ref{lemma_approx_comp}. In particular, by \eqref{choice.AB}, \eqref{constrain}, for any $\fm\geq \bar\fm$, with $\bar\fm \gg 1$ sufficiently large, we estimate:
    \begin{align}
         \tfrac12 I_{0,{\rm out}} &\leq \int_{\tr+\gamma}^{1}\big| \psi_{\fm}(y) - \tfrac12(y -A\,\sgn(y) ) \big|^2 \wrt y  + \int_{\tr+\gamma}^{1}\big| \tfrac12 (y-A\,\sgn(y))^2 - \tfrac12 y^2 \big|^2 \wrt y  \nonumber\\
         & = \int_{\tr+\gamma}^{1}\big| \psi_{\fm}(y) - \tfrac12(y -A\,\sgn(y) ) \big|^2 \wrt y  + A^2\int_{\tr+\gamma}^{1} \big| |y| - \tfrac12 A \big|^2 \wrt y \nonumber \\
         & \leq \int_{\tr+\gamma}^{1}\big| \psi_{\fm}(y) - \tfrac12(y -A\,\sgn(y) ) \big|^2 \wrt y  + A^2(1-\tfrac12 A)^2 (1-(\tr+\gamma)) \nonumber \\
         & \lesssim \int_{\tr+\gamma}^{1}\big| \psi_{\fm}(y) - \tfrac12(y -A\,\sgn(y) ) \big|^2 \wrt y  + A^2 \lesssim 2A^2 \lesssim \tr^2\,;  \label{I0.out} \\ 
         \tfrac12 I_{0,{\rm in}} & \leq \int_{0}^{\tr-\gamma}\big| \psi_{\fm}(y) + \tfrac{B}{\tE} \cos(\tE y) \big|^2 \wrt y + \int_{0}^{\tr-\gamma} \big| \tfrac{B}{\tE}\cos(\tE y) + \tfrac12 y^2 \big|^2 \wrt y \nonumber \\
         & \lesssim \int_{0}^{\tr-\gamma}\big| \psi_{\fm}(y) + \tfrac{B}{\tE} \cos(\tE y) \big|^2 \wrt y + \max\{ \tfrac{B}{\tE},(\tr-\gamma)^2 \}^2 (\tr-\gamma) \nonumber \\
         & \lesssim \int_{0}^{\tr-\gamma}\big| \psi_{\fm}(y) + \tfrac{B}{\tE} \cos(\tE y) \big|^2 \wrt y +  (\tr-\gamma)^5 \lesssim 2\tr^5\,.\label{I0.in}
    \end{align}
On the compact neighbourhood $\{ ||y|-\tr| \leq \gamma  \}$, the approximation with the singular case $\fm\to\infty$ fails. However, being $Q_{\fm}(y)$ uniformly bounded by $\tE^2$ for any $\fm\in\N$  (see Lemma \ref{lemma QmQinfty}), all the solutions of \eqref{Qgamma} are uniformly bounded with respect to $\fm\in\N$ on $[-1,1]$, as well as $\psi_{\fm}(y)$ by integration. Therefore, by \eqref{der0.split} and \eqref{def.smaller.scale}, we have the estimate
\begin{equation}\label{I0.gamma}
    \tfrac12 I_{0,\gamma} \lesssim \max\big\{ \|\psi_{\fm}\|_{L^\infty(B_{\gamma}(\tr))}, (\tr+\gamma)^2 \big\}^2  \gamma \lesssim \tr^{5} \,.
\end{equation}
By \eqref{der0.split}, \eqref{I0.out}, \eqref{I0.in}, \eqref{I0.gamma}, \eqref{choice.AB}, we conclude that $\| \psi_{\fm} - \psi_{\rm cou} \|_{L^2[-1,1]} \lesssim \tr$.

We move now to estimate the $L^2$-norm of $\psi_{\fm}'''(y)$. Similarly as in \eqref{der0.split}, we split
    \begin{align}
        \| \psi_{\fm}''' \|_{L^2[-1,1]}^2 & \leq 2 \int_{\tr+\gamma}^{1}|\psi_{\fm}'''(y) |^2 \wrt y + 2 \int_{B_{\gamma}(\tr)}|\psi_{\fm}'''(y) |^2 \wrt y  + 2 \int_{0}^{\tr-\gamma}|\psi_{\fm}'''(y) |^2 \wrt y \nonumber \\
        & =: I_{3,{\rm out}} + I_{3,\gamma} + I_{3,{\rm in}} \,.\label{der3.split}
    \end{align}
As before, we use the approximation Lemma \ref{lemma_approx_comp} on the compact intervals $[\tr+\gamma,1]$ and $[0,\tr-\gamma]$, with the same choice of the constants $A$, $B$ as in \eqref{choice.AB} and for $\fm\geq \bar\fm$, with a possibly larger $\bar\fm= \bar\fm(\tr)\gg 1$:
    \begin{align}
    \tfrac12 I_{3,{\rm in}} & \leq \int_{0}^{\tr-\gamma} | \psi_{\fm}'''(y) + B\,\tE^2 \sin(\tE y) |^2 \wrt y + B^2 \,\tE^4 \int_{0}^{\tr-\gamma} | \sin(\tE y)|^2 \wrt y\nonumber  \\
    & \leq \int_{0}^{\tr-\gamma} | \psi_{\fm}'''(y) + B\,\tE^2 \sin(\tE y) |^2 \wrt y + B^2 \,\tE^4 (\tr-\gamma) \nonumber \\
    & \lesssim \int_{0}^{\tr-\gamma} | \psi_{\fm}'''(y) + B\,\tE^2 \sin(\tE y) |^2 \wrt y + \tr  \lesssim 2\tr\,; \label{I3.in} \\
    \tfrac12 I_{3,{\rm out}} & = \int_{\tr+\gamma}^{1}| \psi_{\fm}'''(y) |^2\wrt y \lesssim 2\tr \,. \label{I3.out}
    \end{align}
To estimate finally $I_{3,\gamma}$, we recall that $\psi_{\fm}'(y)$ solves \eqref{Qgamma} and, as for \eqref{I0.gamma}, that it is uniformly bounded with respect to $\fm\in\N$ on the whole interval $[-1,1]$. Therefore, by \eqref{der3.split}, \eqref{def.smaller.scale}, \eqref{constrain} and Lemma \eqref{lemma QmQinfty}, we get
\begin{equation}\label{I3.gamma}
   \begin{aligned}
       \tfrac12 I_{3,\gamma} & = \int_{B_{\gamma}(\tr)} | Q_{\fm}(y) |^2  |\psi_{\fm}'(y)|^2 \wrt y 
       \leq \tE^4 \int_{B_{\gamma}(\tr)} |\psi_{\fm}'(y)|^2 \wrt y 
       \lesssim \tE^4 \gamma \lesssim \tr\,.
   \end{aligned}
\end{equation}
By \eqref{der3.split}, \eqref{I3.in}, \eqref{I3.out}, \eqref{I3.gamma}, we conclude that $\| \psi_{\fm}''' \|_{L^2[-1,1]}\lesssim \sqrt\tr$. This concludes the proof of the proposition.
\end{proof}

\subsection{The existence of the local nonlinearities}\label{subsec.FF}

The goal of this section is to determine which nonlinear ODE is locally solved by the stream function $\psi_{\fm}(y)$, starting from the linear ODE \eqref{Qgamma}.
By Lemma \ref{lemma_approx_comp}, in the limit $\fm\to\infty$, the stream function $\psi_{\fm}(y)$ converges locally  on compact sets excluding the singular values $y=\pm\tr$ to a limit function $\psi_{\infty}(y)$ solving locally the second-order semilinear ODE
\begin{equation}\label{limit.nonlinODE}
	\psi_{\infty}''(y) = F_\infty(\psi_{\infty}(y)) := \begin{cases}
		1 & y\in K_{\rm out} \subset\subset \R \setminus[-\tr,\tr] \,, \\
		-\tE^2  \psi_{\infty}(y) & y \in K_{\rm in} \subset\subset (-\tr,\tr) \,.
	\end{cases}
\end{equation}
We want to show that the even function $\psi_{\fm}(y)$ solves an ODE similar to \eqref{limit.nonlinODE} but on the whole interval $[-1,1]$, including the neighbourhoods of the singularities $y=\pm\tr$, which are smoothed out thanks to the analyticity of the potential $Q_{\fm}(y)$ in \eqref{Qgamma}. Outside the neighbourhoods of these singular values, the nonlinearity is expected to be a slight modification of the one in \eqref{limit.nonlinODE} and, morally speaking, to behave locally as a single function and globally as multivalued function.
The construction of the local nonlinearities is carried out in Theorem \ref{nonlin_eq}. To this end, we need a couple of preliminary results.
First, as a corollary of Lemma \ref{lemma_approx_comp}, we deduce monotonicity properties for the stream function $\psi_{\fm}(y)$.
\begin{lem}\label{lemma monotono}
Let $\bar\fm\gg 1$ be fixed as in Proposition \ref{proxim.couette}. Then, for any $\fm \geq \bar\fm$, the following hold:
\\[1mm]
\noindent $(i)$ $\psi_{\fm}(|y|)$ is strictly monotone for $|y| \geq \tr +\gamma$;
\\[1mm]
\noindent  $(ii)$ $\psi_{\fm}(|y|)$ is strictly monotone for $|y| \in B_\gamma(\tr)$;
\\[1mm]
\noindent  $(iii)$ $\psi_{\fm}(y)$ has $2\kappa_0+1$ critical points $0=\ty_{0,\fm} < |\pm\ty_{1,\fm}| < ... < |\pm\ty_{\kappa_0,\fm}| < \tr -\gamma$ and no saddle points when $|y|\leq \tr-\gamma$. In particular, $\psi_{\fm}''$ does not vanish at these critical points.
\end{lem}

\begin{proof}
	{\sc Proof of $(i)$}. By Lemma \ref{lemma_approx_comp} and \eqref{choice.AB}, \eqref{def.smaller.scale}, we have, for any $|y|\geq \tr+\gamma$,
	\begin{equation}
	 \begin{aligned}
	        | \psi_{\fm}'(y) | & \geq |y -A\,\sgn(y)| - | \psi_{\fm}'(y) - (y-A\,\sgn(y)) |\\
	        & \geq \gamma + \tfrac{1}{\tE^2} - | \psi_{\fm}'(y) - (y-A\,\sgn(y)) | \geq \tfrac12\big( \gamma + \tfrac{1}{\tE^2} \big) >0\,,
	 \end{aligned}
	\end{equation}
	having $\fm \geq \bar\fm$.
	This proves item $(i)$ and ensures that $\psi_{\fm}(y)$ is invertible in this region.
	
	{\sc Proof of $(ii)$}. By parity, we prove just that $\psi_{\fm}(y)$ is strictly monotone for $|y-\tr|\leq \gamma$.
	First, by Lemma \ref{lemma_approx_comp} and \eqref{choice.AB}, \eqref{def.smaller.scale}, \eqref{constrain}, we have the pointwise estimate, for $\fm\geq \bar\fm$ and up to subsequences,
	\begin{equation}
	    \begin{aligned}
	        & \psi_{\fm}'(\tr-\gamma)  = B \sin(\tE(\tr-\gamma)) + \big( \psi_{\fm}'(\tr-\gamma) - B\sin(\tE (\tr-\gamma)) \big)\\
	        & = \frac{1}{\tE^2} \Big( 1 + \frac{-\tE\gamma}{\sin(\tE\tr)} \frac{\sin(\tE(\tr-\gamma))-\sin(\tE\tr)}{-\tE\gamma} \Big) +  \big( \psi_{\fm}'(\tr-\gamma) - B\sin(\tE (\tr-\gamma)) \big)
	        \\
	        & \geq \frac{1}{\tE^2 }\Big( 1 - \frac{\tE \gamma}{|\sin(\tE \tr) |}
	        \Big) -  \big| \psi_{\fm}'(\tr-\gamma) - B\sin(\tE (\tr-\gamma))  \big| \\
		& \geq \frac{1}{2 \tE^2 }\big( 1 - \sqrt{2}\tE \gamma \big) \gtrsim \mathtt r^2>0
	    \end{aligned}
	\end{equation}
	and, similarly,
	\begin{equation}\label{der2.posit}
	    \begin{aligned}
 \psi_{\fm}''(\tr-\gamma) \geq \frac{1}{2 \tE }\big( 1 - \sqrt{2}\tE \gamma \big) \gtrsim \mathtt r > 0 
	    \end{aligned}
	\end{equation}
	for $\mathtt r \ll 0$ small enough, since $\mathtt r \simeq \frac{1}{\mathtt E}$ and $\gamma = \mathtt r^5$. 
	Therefore, the claim that $\psi_{\fm}'(y) >0$ for $ |y-\tr|\leq \gamma $ follows if we show that $\psi_{\fm}''(y)>0$ in the same interval. By the mean value Theorem, \eqref{Qgamma}, \eqref{def.smaller.scale}, \eqref{constrain} and Lemma \eqref{lemma QmQinfty}, we have, for any $y\in[\tr-\gamma,\tr+\gamma]$,
	\begin{equation}\label{mean.value.psisec}
	   \begin{aligned}
	         | \psi_{\fm}'' (y) &- \psi_{\fm}''(\tr-\gamma)|  \leq 2\gamma \| \psi_{\fm}''' \|_{L^\infty[-1,1]} \\
	        & \leq 2\gamma \| Q_{\fm} \|_{L^\infty[-1,1]} \| \psi_{\fm}' \|_{L^\infty[-1,1]} \lesssim \tE^2 \gamma \lesssim \tr^{3}\,.
	   \end{aligned}
	\end{equation}
	Therefore, for  $\tr \ll 1$, by \eqref{mean.value.psisec}, \eqref{der2.posit} we deduce
	\begin{equation}
\psi_{\fm}''(y) \geq \psi_{\fm}''(\tr-\gamma) - | \psi_{\fm}'' (y) - \psi_{\fm}''(\tr-\gamma)| \gtrsim \mathtt r - O(\tr^3)  \gtrsim \mathtt r > 0\,.
	\end{equation}
	This concludes the proof of item $(ii)$.

{\sc Proof of $(iii)$}.Under the constrain \eqref{constrain}, let $\ty_{j,\infty}=\frac{j\pi}{\tE}$, $j=0,\pm 1, ...,\pm\kappa_0$ be the $2\kappa_0 + 1$ zeroes of $\psi_{\infty}'(y):=B \sin(\tE y)$ in $[-\tr,\tr]$, and therefore critical points for $\psi_{\infty}(y)=\frac{B}{\tE} \cos(\tE y)$. It is also trivial, by \eqref{choice.AB}, \eqref{constrain}, that $|\psi_{\infty}''(\ty_{j,\infty})|= \frac{\sqrt 2}{\tE}>0$ and that $|\ty_{\pm \kappa_0,\infty}|= \frac{\kappa_0 \pi}{\tE}< \tr$. By Lemma \ref{lemma_approx_comp} and \eqref{def.smaller.scale}, for $\fm\geq \bar\fm$, we have that $\psi_{\fm}'(y)$ has $2\kappa_0+1$ zeroes on $[-(\tr-\gamma),\tr-\gamma]$, denoted by $\ty_{j,\fm}$ for $j=0,\pm 1, ..., \pm \fm$, each one sufficiently close to $\ty_{j,\infty}$ and satisfying $|\psi_{\fm}''(\ty_{j,\fm})|\geq \frac{\sqrt 2}{2 \tE}$. Moreover, by the parity of $\psi_{\fm}(y)$, we deduce that $\ty_{0,\fm}=0$ and that $\ty_{-p,\fm}=-\ty_{p,\fm}$ for any $p=1,...,\kappa_0$. This proves the claim in item $(iii)$ and concludes the proof of the lemma.
\end{proof}


The following result follows from Lemma \ref{lemma zero odd der} and it is a key tool to prove the claimed regularity properties in Theorem \ref{nonlin_eq}.

\begin{lem}\label{expansion psim crit} 
	Let $\bar\fm \gg 1$ be fixed as in Proposition \ref{proxim.couette}.
	 Let $\ty_{p,\fm}$ be a critical point for $\psi_{\fm}(y)$, for a stripe index $p=1,...,\kappa_0$.  For $\fm\geq \bar\fm$, the following hold
	\begin{equation}
		\begin{aligned}
			& \psi_{\fm}^{(2n-1)}(\ty_{p,\fm}) = 0 \,, \quad \forall\, n=1,...,S+1\,.
		\end{aligned}
	\end{equation}
	As a consequence, we have the expansion, for $|\delta|$ small enough,
	\begin{equation}
		\psi_{\fm}(\ty_{p,\fm}+\delta) = \sum_{n=0}^{S+1} \frac{\psi_{\fm}^{(2n)}(\ty_{p,\fm})}{(2n)!} \delta^{2n} + \frac{\psi_{\fm}^{(2S+3)}(\ty_{p,\fm})}{(2S+3)!} \delta^{2S+3} + o(|\delta|^{2(S+2)})\,. 
	\end{equation}
\end{lem}
\begin{proof}
	The expansion is a direct consequence of Lemma \ref{lemma zero odd der} together with iterative derivatives of \eqref{Qgamma}.
	Indeed, by Lemma \ref{lemma zero odd der}-$(iii)$, we have $\psi_{\fm}'(\ty_{p,\fm})=0$, whereas, for $n\geq 2$, we have, by Leibniz rule,
	$$
	\psi_{\fm}^{(2n-1)}(y)=\pa_y^{2(n-2)}\psi_{\fm}'''(y)= \sum_{j=0}^{2(n-2)}\binom{2(n-2)}{j} Q_{\fm}^{(j)}(y) \psi_{\fm}^{(2n-3-k)}(y)\,.
	$$
	The claim $\psi_{\fm}^{(2n-1)}(\ty_{p,\fm}) =0$, for $n=1,...,S+1$, follows consequently by an induction argument.
\end{proof}

We are now ready to prove the main result of this section. We introduce the following notation for the left and the right neighbourhood of a given point, respectively: for any $r > 0$, we define
\begin{equation}
	B_{r}^{-}(\ty) := \{ y\in\R \, : \, y\in (\ty -r, \ty) \}\,, \quad B_{r}^{+}(\ty) := \{ y\in \R \, : \, y \in (\ty,\ty+r) \}\,. 
\end{equation}
For any $S \in \N$, we denote by ${\cal C}^S_0(\R)$ the space of ${\cal C}^S$ functions $f : \R \to \R$ with compact support. In order to state the next theorem, we also recall the definition of the interval ${\mathtt I}_p = \{ y \in \R :  \ty_{p,\fm} \leq |y| \leq \ty_{p+1,\fm}\}$ given in \eqref{sets Ip}. 
\begin{thm}\label{nonlin_eq}
	{\bf (Local nonlinearities).}
	Let $S\in\N$ and let $\fm\geq \bar\fm$, with $\bar\fm \gg 1$ fixed as in Proposition \ref{proxim.couette}. For any $p=0,1,...,\kappa_0$, there exists  a nonlinear function $F_{p,\fm}\in \cC_{0}^{S+1}(\R)$, $\psi\to F_{p,\fm}(\psi)$, such that $\psi_{\fm}(y)$ solves the nonlinear ODE
	\begin{equation}\label{ell_psi_gamma}
		\psi_\fm''(y) = F_{p,\fm}(\psi_\fm(y))\,, \quad y\in \tI_{p}\,.
	\end{equation}
In particular, the derivative of $F_{p,\fm}$ evaluated at $\psi=\psi_{\fm}(y)$ satisfies, for any $y\in\R$,
\begin{equation}\label{F_Q_gamma}
	(\pa_\psi F_{p,\fm})(\psi_\fm(y))=Q_\fm(y)\,, \quad \forall \,p=0,1,...,\kappa_0\,.
\end{equation}
We have $\cC^{S+1}$-continuity at $\psi=\psi_{\fm}(y)$ at the critical points $y=\pm \ty_{p,\fm}$, $p=1,...,\kappa_0$, meaning that, for any $n=0,1,...,S+1$,
\begin{equation}\label{cont Fp critical pts}
	\lim_{|y|\to\ty_{p,\fm}^-} \pa_{y}^{n} (F_{p-1,\fm}(\psi_{\fm}(y)) )= \lim_{|y|\to\ty_{p,\fm}^+}\pa_{y}^{n}( F_{p,\fm}(\psi_{\fm}(y)) ) = \psi_{\fm}^{(n+2)}(\ty_{p,\fm}) \,.
\end{equation}
\end{thm}
\begin{proof}

Each function $F_{p,\tm}(\psi)$ is constructed on the interval of monotonicity for the stream function $\psi_{\fm}(y)$ by solving Cauchy problems that lead to \eqref{F_Q_gamma}. By Lemma \ref{expansion psim crit}, the behaviour around the critical points $\pm\ty_{1,\fm},...,\pm\ty_{\kappa_0,\fm}$ both for the stream function $\psi_{\fm}(y)$ and the potential $Q_{\fm}(y)$ will determine the regularity of the functions. The construction is carried out in several steps.

{\bf Step 1) Behaviour of $\psi_{\fm}(y)$ around the critical points.} It is convenient to rewrite both the stream function $\psi_{\fm}(y)$ and the potential $Q_{\fm}(y)$ as quadratic functions with finite regularity locally around each critical point of $\psi_{\fm}(y)$.

 We start with $\ty_{0,\fm}=0$. Since both $Q_{\fm}(y)$ and $\psi_{\fm}(y)$ are even in $y$, we write
\begin{equation}
	Q_{\fm}(y) = K_{0,\fm}(y^2) \,, \quad \psi_{\fm}(y) = G_{0,\fm}(y^2)\,, \quad |y| < \ty_{1,\fm}:=\tr_{0,+}\,.
\end{equation}
By Lemma \ref{lemma monotono}-$(iii)$ we have $\psi_{\fm}''(0)\neq 0$ and  $|G_{0,\fm}'(y^2)| = \big|\frac{\psi_{\fm}'(y)}{2y}\big|>0$ for any $y\in \tI_{0}$. Therefore, $G_{0,\fm}(z)$ is invertible for $|z|<\sqrt{\ty_{1,\fm}}$. In the same region, because both $Q_{\fm}(y)$ and $\psi_{\fm}(y)$ are analytic and even, we have that $K_{0,\fm}(z)$ and $G_{0,\fm}(z)$ are in $\cC^{\infty}$, as well as $G_{0,\fm}^{-1}$.

We move now around $|y|=\ty_{p,\fm}$, $p=1,..,\kappa_0-1$. By Lemma \ref{lemma zero odd der}, we deduce that we can write, for $||y|-\ty_{p,\fm}| < \tr_{p,\pm}$, with $\tr_{p,\pm}:=|\ty_{p,\fm} - \ty_{p\pm1,\fm}|$,
\begin{equation}
	Q_{\fm}(y) = K_{p,\fm,\pm}\big((|y|-\ty_{p,\fm})^2\big) \,, \quad K_{p,\fm,\pm} \in \cC^{S}(B_{\sqrt{\tr_{p,\pm}}}(0)) \,.
\end{equation}
Similarly, by Lemma \ref{expansion psim crit},
we write, for $||y|-\ty_{p,\fm}| <\tr_{p,\pm}$, 
\begin{equation}
	\psi_{\fm}(y) = G_{p,\fm,\pm}\big((|y|-\ty_{p,\fm})^2\big) \,, \quad G_{p,\fm,\pm} \in \cC^{S+1}(B_{\sqrt{\tr_{p,\pm}}}(0))\,.
\end{equation}
By Lemma \ref{lemma monotono}-$(iii)$, we have 
\begin{equation}
	\begin{aligned}
		& \big|G_{p,\fm,-}'\big( (|y|-\ty_{p,\fm})^2 \big)\big| = \Big| \frac{\psi_{\fm}'\big( |y|-\ty_{p,\fm} \big)} {2(|y|-\ty_{p,\fm})} \Big| >0 \quad \forall \, y \in \tI_{p-1}\,, \\
		& \big|G_{p,\fm,+}'\big( (|y|-\ty_{p,\fm})^2 \big)\big| = \Big| \frac{\psi_{\fm}'\big( |y|-\ty_{p,\fm} \big)}{2(|y|-\ty_{p,\fm})} \Big| >0 \quad \forall \, y \in \tI_{p}\,, \\
		& \lim_{|y|\to\ty_{p,\fm}^{\pm}} G_{p,\fm,\pm}'\big( (|y|-\ty_{p,\fm})^2 \big) = \tfrac12 \psi_{\fm}''(\ty_{p,\fm}) \neq 0\,.
	\end{aligned}
\end{equation}
Therefore, $G_{p,\fm,-}(z)$ is invertible for $|z|<\sqrt{\tr_{p,-}}$ and $G_{p,\fm,+}(z)$ is invertible for $|z|<\sqrt{\tr_{p,+}}$, with inverses $G_{p,\fm,\pm}^{-1}$ being in $\cC^{S+1}$ in the respective regions.

Finally, we consider the critical points $|y|=\ty_{\kappa_0,\fm}$. With the same previous arguments, we have, for $\ty_{\kappa_0-1,\fm}< |y| \leq \ty_{\kappa_0,\fm}$, with $\tr_{\kappa_0,-}:=\ty_{\kappa_0,\fm}-\ty_{\kappa_0-1,\fm}$
\begin{equation}
	\begin{aligned}
		& Q_{\fm}(y) = K_{\kappa_0,\fm,-} \big( (|y|-\ty_{\kappa_0,\fm})^2 \big)\,, \quad K_{\kappa_0,\pm,-} \in \cC^{S}(B_{\sqrt{\tr_{\kappa_0,-}}}(0))\,, \\
		& \psi_{\fm}(y) = G_{\kappa_0,\fm,-} \big( (|y|-\ty_{\kappa_0,\fm})^2 \big)\,, \quad Q_{\kappa_0,\pm,-} \in \cC^{S+1}(B_{\sqrt{\tr_{\kappa_0,-}}}(0))\,,
	\end{aligned}
\end{equation}
and, for $\ty_{\kappa_0,\fm}\leq |y|\leq 1$, with $\tr_{\kappa_0,+}:=1-\ty_{\kappa_0,\fm}$,
\begin{equation}
	\begin{aligned}
		& Q_{\fm}(y) = K_{\kappa_0,\fm,+} \big( (|y|-\ty_{\kappa_0,\fm})^2 \big)\,, \quad K_{\kappa_0,\pm,+} \in \cC^{S}(B_{\sqrt{\tr_{\kappa_0,+}}}(0))\,,  \\
		& \psi_{\fm}(y) = G_{\kappa_0,\fm,+} \big( (|y|-\ty_{\kappa_0,\fm})^2 \big)\,, \quad Q_{\kappa_0,\pm,+} \in \cC^{S+1}(B_{\sqrt{\tr_{\kappa_0,+}}}(0))\,.
	\end{aligned}
\end{equation}
Also in this case, by Lemma \ref{lemma monotono}, $G_{\kappa_0,\fm,-}(z)$ is invertible for $|z|<\sqrt{\tr_{\kappa_0,-}}$ and $G_{\kappa_0,\fm,+}(z)$ is invertible for $|z|<\sqrt{\tr_{\kappa_0,+}}$, with inverses $G_{\kappa_0,\fm,\pm}^{-1}$ being in $\cC^{S+1}$ in the respective regions.

{\bf Step 2) Existence and smoothness of $F_{p,\fm}(\psi)$.} We start with the stripe indexes $p=0,1,..., \kappa_0-1$. We look for $F_{p,\fm}(\psi)$ of the form $F_{p,\fm}(\psi)=-\tE^2 \psi + Z_{p,\fm}(\psi)$. Let $0<\gamma_{0}\ll \tr_{0,+}$ and $0<\gamma_{p}\ll \min\{\tr_{p,-},\tr_{p,+}\}$, $p\geq 1$, small enough.
First, we define $Z_{p,\fm,+}(\psi)$ as the solution of the Cauchy problem
\begin{equation}
	\begin{cases}
		\pa_\psi Z_{p,\fm,+}(\psi) = K_{p,\fm,+}(G_{p,\fm,+}^{-1}(\psi))+\tE^2\,, \quad \psi \in \psi_{\fm}(B_{\tr_{p,+}-\gamma_{p+1}}^{+}(\ty_{p,\fm}))\,,   \\
		Z_{p,\fm,+} (\psi_{\fm}(\ty_{p,\fm})) = \psi_{\fm}''(\ty_{p,\fm}) + \tE^2 \psi_{\fm}(\ty_{p,\fm})\,.
	\end{cases}
\end{equation}
Then, we define $Z_{p,\fm,-}(\psi)$ as the solution of the Cauchy problem
\begin{equation}
	\begin{cases}
		\pa_\psi Z_{p,\fm,-}(\psi) = K_{p+1,\fm,-}(G_{p+1,\fm,-}^{-1}(\psi))+\tE^2\,, \quad \psi \in \psi_{\fm}(B_{\tr_{p,-}-\gamma_{p}}^{-}(\ty_{p+1,\fm}))\,,   \\
		Z_{p,\fm,-} (\psi_{\fm}(\ty_{p+1,\fm})) = \psi_{\fm}''(\ty_{p+1,\fm}) + \tE^2  \psi_{\fm}(\ty_{p+1,\fm})\,.
	\end{cases}
\end{equation}
By Step 1 and the boundedness of the vector fields following from Lemma \ref{lemma QmQinfty}, both problems are well defined, with $Z_{p,\fm,+}\in \cC^{S+1}\big( \psi_{\fm}(B_{\tr_{p,+}-\gamma_{p+1}}^{+}(\ty_{p,\fm})) \big)$ and $Z_{p,\fm,-}\in \cC^{S+1}\big( \psi_{\fm}(B_{\tr_{p+1,-}-\gamma_{p}}^{-}(\ty_{p+1,\fm})) \big)$.  We claim that, when $\psi \in \psi_{\fm}(B_{\tr_{p,+}-\gamma_{p+1}}^{+}(\ty_{p,\fm})) \cap  \psi_{\fm}(B_{\tr_{p+1,-}-\gamma_{p}}^{-}(\ty_{p+1,\fm}))  $, then $Z_{p,\fm,+}(\psi)= Z_{p,\fm,-}(\psi)$. Indeed, using the respective initial values, the two functions satisfy, for $|y| \in B_{\tr_{p,+}-\gamma_{p+1}}^{+}(\ty_{p,\fm}) \cap B_{\tr_{p+1,-}-\gamma_{p}}^{-}(\ty_{p+1,\fm})$,
\begin{equation}
	(\pa_\psi Z_{p,\fm,+})(\psi_\fm(y))  = Q_{\fm} (y) +\tE^2= (\pa_\psi Z_{p,\fm,-} (\psi_{\fm}(y)) )  \,.
\end{equation}
By uniqueness of the solution of the Cauchy problems, the claim follows. We denote both solutions by $Z_{p,\fm}(\psi)$ when $\psi\in \psi_{\fm}( B_{\tr_{p,+}}^{+}(\ty_{p,\fm})) = \psi_{\fm}(B_{\tr_{p+1,-}}^{-}(\ty_{p+1,\fm}))$ and we conclude that $Z_{p,\fm} \in \cC^{S+1}\big( \psi_{\fm}( B_{\tr_{p,+}}^{+}(\ty_{p,\fm}))  \big)$, as well as for $F_{p,\fm}(\psi)=-\tE^2\psi + Z_{p,\fm}(\psi)$.

Finally, let $p=\kappa_0$.
We define $F_{\kappa_0,\fm}(\psi)$ as the solution of the Cauchy problem
\begin{equation}
	\begin{cases}
		\pa_\psi F_{\kappa_0,\fm}(\psi) = K_{\kappa_0,\fm,+} ( G_{\kappa_0,\fm,+}^{-1}(\psi) ) \,, \quad \psi \in \psi_{\fm}(B_{\tr_{\kappa_0,+}}(\ty_{\kappa_0,\fm}))  \,, \\
		F_{\kappa_0,\fm} (\psi_{\fm}(\ty_{\kappa_0,\fm})) = \psi_{\fm}''(\ty_{\kappa_0,\fm})\,. 
	\end{cases}
\end{equation}

\noindent
By Step 1 and the boundedness of the vector field following from Lemma \ref{lemma QmQinfty}, the problem is well defined, with $F_{\kappa_0,\fm} \in \cC^{S+1}\big( \psi_{\fm}(B_{\tr_{\kappa_0,+}}^{+}(\ty_{\kappa_0,\fm})) \big)$.

{\bf Step 3) Global $\cC^{S+1}$-continuity.} By Whitney extension Theorem, we extend all the functions $F_{p,\fm}(\psi)$, $p=0,1,...,\kappa_0$, from their domains of definition to global functions in $\cC_{0}^{S+1}(\R)$. For sake of simplicity in the notation, we keep denoting the extensions by $F_{p,\fm}(\psi)$. Note that, by Step 1 and the construction of the Cauchy problems in Step 2, for any $p=0,1,....,\kappa_0$ and any $y\in\tI_{p}$, we have 
\begin{equation}
	(\pa_\psi F_{p,\fm})(\psi_{\fm}(y)) = Q_{\fm}(y)\,.
\end{equation}
By the smoothness of $Q_{\fm}(y)$ and the choice of the initial values in the Cauchy problems, we obtain that both \eqref{F_Q_gamma} and \eqref{cont Fp critical pts} hold. This concludes the proof of the theorem. 
\end{proof}

\begin{rem}
	During the proof of Theorem \ref{nonlin_eq}, for $p=0,1,...,\kappa_0-1$ we constructed the nonlinearities $F_{p,\fm}(\psi) = -\tE^2 \psi + Z_{p,\fm}(\psi)$  by defining Cauchy problems for $Z_{p,\fm}(\psi)$, whereas for $p=\kappa_0$ we directly consider the Cauchy problem for $F_{\kappa_0,\fm}(\psi)$. There is no conceptual difference between the two kinds of constructions. The reason behind this choice is purely expository: we just wanted to highlight that the nonlinearity $F_{p,\fm}(\psi)$ when $p < \kappa_0$ is a perturbation of the linear function $-\tE^2\psi$. It is also possible to show that, when $|y| \geq \tr + \gamma $, we have $F_{\kappa_0,\fm}(\psi_{\fm}(y)) = 1 + Z_{\kappa_0,\fm}(\psi_{\fm}(y))$ for some small function $Z_{\kappa_0,\fm}(\psi)$. Morally speaking, the nonlinearities constructed in Theorem \ref{nonlin_eq} are slight local modifications of the nonlinearity $F_{\infty}(\psi)$ in \eqref{limit.nonlinODE}.
\end{rem}

The following corollary of Theorem \ref{nonlin_eq} and Lemmata \ref{lemma zero odd der}, \ref{expansion psim crit} will be used at the beginning of Section \ref{sez.nonlin.lin}.

\begin{cor}\label{cor:asimptotics.close.pc}
For any $n=1,...,S$ and for any $p=1,...,\kappa_0$, we have
\begin{equation}\label{around.cp}
	\begin{aligned}
		&(\pa_{\psi}^n F_{p,\fm})(\psi_{\fm}(y)) = (\pa_{\psi}^n F_{p-1,\fm})(\psi_{\fm}(y)) = \fP_{n}(y)\,,
	\end{aligned}
\end{equation}
where the function $\fP_{n}(y)$ is independent of the strip index $p$. Moreover, there exists $\bar\delta=\bar\delta(p,\fm)>0$ small enough such that, for any $\delta\in (0,\bar\delta)$, the function $\fP_{n}(y)$ satisfies the expansions
\begin{equation}\label{expan.around.cp}
    \fP_{n}(\ty_{p,\tm}+\delta) = \sum_{k=0}^{S-n} \frac{\fp_{2k}^{(n)}}{(2k)!}\delta^{2k} + \frac{\fp_{2(S-n)+1}^{(n)}}{(2(S-n)+1)!}\delta^{2(S-n)+1} + o(|\delta|^{2(S-n+1)})\,.
\end{equation}
\end{cor}

\begin{proof}
 We argue by induction. For $n=1$,  we have that \eqref{around.cp} and \eqref{expan.around.cp} hold true by \eqref{F_Q_gamma} and Lemma \ref{lemma zero odd der}, setting $\fP_{0}(y):=Q_{\fm}(y)$. We now assume by induction that the claim holds for a fixed $n\in\{1,..,S-1\}$ and we show it for $n+1$. By differentiating  \eqref{F_Q_gamma} iteratively in $y$, we have
\begin{equation}\label{peppa 1}
	(\pa_{\psi}^{n+1} F_{p,\fm})(\psi_{\fm}(y)) = \frac{1}{\psi_{\fm}'(y)}  \pa_y \big( (\pa_{\psi}^n F_{p,\fm})(\psi_{\fm}(y)) \big)\,,
\end{equation}
with similar formula for $F_{p-1,\tm}$.
By the induction assumption, $(\pa_{\psi}^n F_{p,\fm})(\psi_{\fm}(y))$ and $(\pa_{\psi}^n F_{p-1,\fm})(\psi_{\fm}(y))$ satisfy \eqref{around.cp} at the step $n$. This proves \eqref{around.cp} at the step $n+1$, with $\fP_{n+1}(y)$ given by the right hand side of \eqref{peppa 1}. Moreover, by the induction assumption, the expansion in \eqref{expan.around.cp} holds at the step $n$ and we compute
    \begin{align}
         & \pa_y \big( (\pa_{\psi}^n F_{p,\fm})(\psi_{\fm}(y)) \big) = \pa_y \fP_{n}(y) \label{peppa 2} \\
        & = \tfrac{\fp_{2}^{(n)}}{2}(y-\ty_{p,\tm}) + \sum_{k=2}^{S-n}\tfrac{\fp_{2k}^{(n)}}{(2k-1)!}(y-\ty_{p,\tm})^{2k-1} \notag \\
        & \quad  + \tfrac{\fp_{2(S-n)+1}^{(n)}}{(2(S-n))!}(y-\ty_{p,\tm})^{2(S-n)} + o(|y-\ty_{p,\tm}|^{2(S-n)+1})\notag \\
        & = \tfrac{\fp_{2}^{(n)}}{2}(y-\ty_{p,\tm})\Big( 1 + \sum_{k=1}^{S-n-1}\tfrac{\wt\fp_{2k}^{(n)}}{(2k)!}(y-\ty_{p,\tm})^{2k}\notag  \\ 
        & \quad  + \tfrac{\wt\fp_{2(S-n)-1}^{(n)}}{(2(S-n)-1)!}(y-\ty_{p,\tm})^{2(S-n)-1} + o(|y-\ty_{p,\tm}|^{2(S-n)})  \Big)\,, \quad \wt\fp_{j}^{(n)}:= \tfrac{2\,\fp_{j+2}^{(n)}}{(j+1)\fp_{2}^{(n)}}\,. \notag
    \end{align}
By Lemma \ref{expansion psim crit}, a similar computation leads to
\begin{equation}
    \begin{aligned}
        \psi_{\tm}'(y) & = \tfrac{\psi_{\fm}''(\ty_{p,\tm})}{2} (y-\ty_{p,\tm})\Big( 1 + \sum_{k=1}^{S} \tfrac{\wt\psi_{\tm,2k}}{(2k)!}(y-\ty_{p,\tm})^{2k} \\
        & \quad + \tfrac{\wt\psi_{\tm,2S+1}}{(2S+1)!}(y-\ty_{p,\tm})^{2S+1}+o(|y-\ty_{p,\tm}|^{2(S+1)}) \Big)\,, \quad \wt\psi_{\tm,j}:=\tfrac{2\psi_{\tm}^{(j+2)}(\ty_{p,\tm})}{(j+1)\psi_{\tm}''(\ty_{p,\tm})}\,.
    \end{aligned}
\end{equation}
By Lemma \ref{lemma monotono}-$(iii)$, we get that $\psi_{\tm}'(y)$ is invertible for $y$ sufficiently close to $\ty_{p,\tm}$, with $(\psi_{\tm}')^{-1}$ having a similar expansion as above, with different coefficients provided by the Neumann series. Combining such expansion in \eqref{peppa 1} together with \eqref{peppa 2}, the claim in \eqref{expan.around.cp} holds at the step $n+1$. This concludes the proof.
\end{proof}

\subsection{Spectral analysis of the linear operator}\label{subsec.spectrum}

We now want to study the spectrum of the operator $\cL_{\fm} := - \partial_y^2 + Q_{\fm}(y)$. We shall emphasize that this linear operator depends on the parameter $\tE \in [\tE_1,\tE_2]$, with $1 < \tE_1 < \tE_2$, hence we often write $\cL_{\fm} \equiv \cL_{\fm}(\tE)$ and $Q_{\fm}(y) \equiv Q_{\fm}(\tE;y)$. We shall prove that $\cL_{\fm}(\tE)$ has a finite number of negative eigenvalues, which we will use them in Section \ref{deg.KAM} in order to impose some Diophantine conditions by cutting away some resonance zones in the parameter space $[\tE_1,\tE_2]$. Such property will be inferred from the limit operator $\cL_{\infty}=\cL_{\infty}(\tE)$.  
%

\begin{prop}\label{L_operator}
	The Schrödinger  operator $\cL_{\fm}:=-\pa_y^2 +Q_{\fm}(y) $, with $Q_{\fm}(y)$ as in \eqref{Qgamma}, is self-adjoint in $L^2_0([-1,1])$ on the domain 
$$D(\cL_{\fm}):= \big\{\phi\in H_0^1([-1,1]) \, : \, \phi(y)=\phi(-y) \big\}\,,$$
with a countable $L^2$-basis of eigenfunctions $(\phi_{j,\fm}(y))_{j\in\N}\subset \cC^{\infty}[-1,1]$ corresponding to the  eigenvalues $(\mu_{j,\fm})_{j\in\N}$. Moreover, under the constrain \eqref{constrain},
	with respect to the order $\mu_{1,\fm} < \mu_{2,\fm} < ...  < \mu_{j,\fm} < ...$, 
	there exists $\bar\fm= \bar\fm(\tE_1,\tE_2,\kappa_0)\gg 1$, possibly larger than the threshold in Proposition \ref{proxim.couette}, such that, for any $\fm\geq \bar\fm$, the first $\kappa_0$ eigenvalues are strictly negative and larger than $-\tE^2$, whereas all the others are strictly positive: we write
	\begin{equation}
		\mu_{j,\fm}=\begin{cases}
			-\lambda_{j,\fm}^2\in (-\tE^2,0) & j=1,... ,\kappa_0\,,\\
			\lambda_{j,\fm}^2>0 & j\geq \kappa_0+1\,.
		\end{cases}
	\end{equation}
	In particular,  for any $j=1,...,\kappa_0$, we have that $\lambda_{j,\fm}$ is close to $\lambda_{j,\infty}$, with the latter being the $j$-th root out of $\kappa_0$ of the transcendental equation in the region $\lambda\in(0,\tE)$
	\begin{equation}\label{secular_0}
		\fF(\lambda):= \lambda \cos\big(\tr \sqrt{\tE^2-\lambda^2}\big)\coth((1-\tr)\lambda) - \sqrt{\tE^2-\lambda^2}\sin\big(\tr \sqrt{\tE^2-\lambda^2}\big) = 0 \,.
\end{equation}
\end{prop}
\begin{rem}\label{rem.secular}
	The existence of the $\kappa_0$ roots of the equation \eqref{secular_0} is established in Lemma \ref{lemma.asympt}, where an asymptotic for such zeroes when $\tE\to\infty$ is provided as well.
\end{rem}
\begin{proof}
	We split the proof in several steps.
	\\[1mm]
	\indent {\bf Step 1) 0 is not an eigenvalue.}
	The self-adjointness of $\cL_\fm$ and its spectral resolution follow by standard arguments of functional analysis. The smoothness of the eigenfunctions $(\phi_{j,\fm}(y))_{j\in\N}$ follows from standard regularity properties for the Sturm-Liouville problem, since $Q_{\fm}\in\cC^{\infty}[-1,1]$ (it is actually analytic in $[-1,1]$). 
	
	We claim now that $0$ is not an eigenvalue. Indeed, if we know that $\mu_{\kappa_0,\fm}<0$, we claim that $\mu_{\kappa_0+1,\fm}>0$. We argue by contradiction and we assume that $\mu_{\kappa_0+1,\fm}\leq 0$. We have that the odd function $\psi_\fm'(y)$ is a generalized eigenfunction for $\cL_\fm=-\pa_y^2 +Q_\fm(y)$, since $\cL_\fm \psi_{\fm}'=0$ by \eqref{Qgamma}. Combining it with $\cL_{\fm}\phi_{\kappa_0+1,\fm}= \mu_{\kappa_0+1,\fm}\phi_{\kappa_0+1,\fm}$, we get
	\begin{equation}\label{ferrari1}
		\mu_{\kappa_0+1,\fm} \psi_{\fm}' (y)\phi_{\kappa_0+1,\fm} (y)= \big( \psi_{\fm}''(y) \phi_{\kappa_0+1,\fm}(y) - \psi_{\fm}' (y)\phi_{\kappa_0+1,\fm}'(y) \big)'\,.
	\end{equation}
	Without loss of generality, we assume $\phi_{\kappa_0+1,\fm}(0)>0$ and, by symmetry, we work on $y\in[0,1]$.
	In this region, by Lemma \ref{lemma monotono}-(iii), $\psi_{\fm}'(y)$ vanishes at the $\kappa_0+1$ points $\ty_{0,\fm}=0$, $(\ty_{p,\fm})_{p=1,...,\kappa_0}$. At the same time, the even eigenfunction $\phi_{\kappa_0+1,\fm}(y)$ vanishes at $\kappa_0$ nodes in $(0,1)$ (see for instance Theorem 9.4 in \cite{teschl}). We deduce that one of the following cases have to happen:
	\begin{itemize}
		\item All the points $\ty_{p,\fm}$, $p=1,...,\kappa_0$, are also nodes for $\phi_{\kappa_0+1,\fm}$: $\psi_{\fm}'(\ty_{p,\fm})=\phi_{\kappa_0+1,\fm}(\ty_{p,\fm})=0$;
		\item There exists at least one $p=1,...,\kappa_0$ such that $\ty_{p,\fm}$ is not a zero of $\phi_{\kappa_0+1,\fm}$ and lies between two if its zeroes, say $\bar\ty_1 $ and $\bar\ty_2$, with $0<\bar\ty_1 < \ty_{p,\fm}<\bar\ty_2\leq 1$.
	\end{itemize}
	We assume that the second case holds. Furthermore, without any loss of generality, we also assume that 
	$\phi_{\kappa_0+1,\fm}(y)>0$ and $\psi_{\fm}'(y)<0$ for any $y\in (\bar\ty_1,\ty_{p,\fm})$.
	We integrate \eqref{ferrari1} on $[\bar\ty_1,\ty_{p,\fm}]$ and we get
	\begin{equation}\label{ferrari2}
		\begin{aligned}
			\mu_{\kappa_0+1,\fm}&  \int_{\bar\ty_1}^{\ty_{p,\fm}}  \psi_{\fm}'(y) \phi_{\kappa_0+1,\fm}(y)\wrt y  = \big[  \psi_{\fm}''(y) \phi_{\kappa_0+1,\fm}(y) - \psi_{\fm}'(y) \phi_{\kappa_0+1,\fm}'(y) \big]_{\bar\ty_1}^{\ty_{p,\fm}} \\
			& = \psi_{\fm}''(\ty_{p,\fm})\phi_{\kappa_0+1,\fm}(\ty_{p,\fm}) + \psi_{\fm}'(\bar\ty_1)\phi_{\kappa_0+1,\fm}(\bar\ty_1)\,.
		\end{aligned}
	\end{equation}
	By construction, $\phi_{\kappa_0+1,\fm}(\ty_{p,\fm})>0$ and $\phi_{\kappa_0+1,\fm}'(\bar\ty_1)>0$, whereas $\psi_{\fm}'(\bar\ty_1)\leq0$ and $\psi_{\fm}''(\ty_{p,\fm})>0$. We conclude that the right hand side of \eqref{ferrari2} is strictly negative and the integral on the left hand side is strictly negative by construction as wel. But we assumed at the beginning that $\mu_{\kappa_0+1,\fm}\leq 0$, therefore we have reached a contradiction and the claim is proved.
	\\[1mm]
	\indent {\bf Step 2) Negative eigenvalues in the limit case $\fm\to\infty$.}
	We consider first the limit case $\fm\to \infty$, with $Q_\infty(y)$ given in \eqref{Qgamma}.  Let $\lambda>0$. We look for solutions $\phi\in \cC^1([-1,1])$ of the eigenvalue problem
	\begin{equation}\label{bc_eigen}
		\begin{cases}
			\cL_{\infty} \phi(y):-\phi''(y) + Q_\infty(y)\phi(y) = -\lambda^2\phi(y)\,, & y\in[-1,1]\setminus\{\pm\tr\}\,, \\
			\phi(y)=\phi(-y)\,, \quad \phi(-1)=\phi(1)=0\,, \\
			\lim_{y\to\pm\tr^{-}}\phi(y)=\lim_{y\to\pm\tr^{+}}\phi(y)\,, \\ \lim_{y\to\pm\tr^{-}}\phi'(y)=\lim_{y\to\pm\tr^{+}}\phi'(y)\,.
		\end{cases}
	\end{equation}
	The conditions $\phi(y)=\phi(-y)$ and $\phi\in\cC^1$ fail the problem to be solved when $\lambda\geq \tE$. On the other hand,  for $0<\lambda<\tE$, the solutions are given by
	\begin{equation}\label{eigenf.gen}
		\phi(y)=\begin{cases}
			-c_1 \sinh(\lambda(1+y)) & -1\leq y<\tr\,, \\
			c_2\cos(\varsigma(\lambda)y) & |y|\leq \tr\,,\\
			c_1 \sinh(\lambda(1-y)) & \tr<y\leq 1\,,
		\end{cases}
	\end{equation}
	where $\varsigma(\lambda):=\sqrt{\tE^2-\lambda^2}$. The last two conditions in \eqref{bc_eigen} at $y=\tr$ translate into
	\begin{equation}
		\begin{cases}
			c_1 \sinh(\lambda(1-\tr))-c_2\cos(\varsigma(\lambda)\tr)=0\,,\\
			c_1 \lambda\cosh(\lambda(1-\tr))-c_2\varsigma(\lambda)\sin(\varsigma(\lambda)\tr)=0\,.
		\end{cases}
	\end{equation}
	A nontrivial solution $(c_1,c_2)\in\R^2\setminus\{0\}$ exists only when $\lambda$ solves
	\begin{equation}
		\varsigma(\lambda)\sin(\tr\varsigma(\lambda))\sinh((1-\tr)\lambda)=\lambda\cos(\tr\varsigma(\lambda))\cosh((1-\tr)\lambda)\,,
	\end{equation}
	which is equivalent to \eqref{secular_0}. Under the constrain \eqref{constrain}, equation \eqref{secular_0} has exactly $\kappa_0$ distinct zeroes in the interval $\lambda\in(0,\tE)$, see Remark \ref{rem.secular} and Lemma \ref{lemma.asympt}. We denote these zeroes by $\lambda_{j,\infty}=\lambda_{j,\infty}(\tE)$ for any $j=1,...,\kappa_0$. The negative eigenvalues of $\cL_{\infty}$ are then given by $\mu_{j,\infty}=\mu_{j,\infty}(\tE):=-\lambda_{j,\infty}^2(\tE)$, with corresponding eigenfunctions $\phi_{j,\infty}(y)=\phi_{j,\infty}(\tE;y)$ given in \eqref{eigenf.gen} with $\lambda=\lambda_{j,\infty}(\tE)$ and $c_1,c_2$ chosen as normalizing constants in $L^2$. We also have
	\begin{equation}
		\mu_{j,\infty} = \fB_{\infty} (\phi_{j,\infty})\,, \quad \fB_{\infty}(\psi) = (\cL_{\infty} \psi, \psi) _{L^2}\,.
	\end{equation}
	\indent {\bf Step 3) Negative eigenvalues when $\fm\gg 1$.}
	We now analyse the case $\fm\gg 1$. 
	In Step 1 we showed that, assuming only $\kappa_0$ negative eigenvalues of $\cL_{\fm}$, the rest of the spectrum is strictly positive. We now prove that, for any $j=1,...,\kappa_0$,
	\begin{equation}\label{vicinanza autovalori in gamma}
		\sup_{\tE \in [\tE_1, \tE_2]} |\mu_{j, \fm}(\tE) - \mu_{j, \infty}(\tE)| \to 0 \quad \text{as} \quad \tm \to \infty\,,
	\end{equation}
	where $(\mu_{j, \infty})_{j=1,...,\kappa_0}$ are all the negative eigenvalues of the limit operator $\cL_{\infty}:=-\pa_y^2 + Q_{\infty}(y)$ that we characterized in Step 2.
	For any $j = 1,..., \kappa_0$, let us denote by ${\cal E}_{j- 1}$ the set of all finite dimensional subspaces of dimension $j - 1$ of $H_0^1(- 1, 1)$. By the Min-Max Theorem (for instance, see \cite{teschl}), we have that the eigenvalues $\mu_{1,\fm} < \mu_{2,\fm} < ... < \mu_{\kappa_0,\fm}$ of $\cL_\fm$ satisfy the variational formulation, for any $j =1,...,\kappa_0$,
	\begin{equation}\label{muj.gamma.minmax}
	\begin{aligned}
	    	&\mu_{j,\fm} := \sup_{E \in {\cal E}_{j - 1}} \inf\big\{ \fB_{\fm}(\psi) \, : \, \psi\in E^\bot \cap H_0^1(-1, 1), \ \| \psi \|_{L^2} = 1 \big\} \,, \\
	    	& \fB_{\fm}(\psi):= (\cL_\fm \psi,\psi)_{L^2}\,,
	\end{aligned}
	\end{equation}
	Each eigenfunction $\phi_{j,\fm}(y)$ is the solution to the max-min variational problems in \eqref{muj.gamma.minmax}.
	Let $E_j = {\rm span}\{ \phi_{1, \infty}, \ldots, \phi_{j - 1, \infty} \} \in {\cal E}_{j - 1}$ and let $\phi_{j, \infty} \in E_j^\bot \cap H_0^1(- 1, 1)$ where, for $i =1,...,\kappa_0$, the function $\phi_{i, \infty}$ is the i-th eigenfunction of the operator $\cL_{\infty}$.  By standard theory for Sturm Liouville operators and by the Sobolev embedding, for any $\tE \in [\tE_1, \tE_2]$, one has that $\phi_{j, \infty} \equiv \phi_{j, \infty}(\tE)$ satisfies  
	\begin{equation}\label{psi j 0 L infty}
	\| \phi_{j, \infty} \|_{L^\infty} \lesssim \| \phi_{j, \infty} \|_{H^1} \leq C(\tE_1, \tE_2, j), 
	\end{equation}
	for some constant $C(\tE_1, \tE_2, j) > 0$. We compute
	\begin{equation}\label{ferrarino 1}
	\begin{aligned}
	\fB_{\fm}(\phi_{j, \infty})&=({\cal L}_\fm \phi_{j, \infty}, \phi_{j, \infty})_{L^2}   = (\cL_{\infty}\phi_{j, \infty}, \phi_{j, \infty})_{L^2} + \big( ({\cal L}_\fm - \cL_{\infty}) \phi_{j, \infty}\,,\, \phi_{j, \infty} \big)_{L^2} \\
	& \leq \lambda_{j, \infty} + |c_j(\fm, \tE)|\,, \quad c_j(\fm, \tE)  := \big( (Q_\fm(y) - Q_\infty(y)) \phi_{j, \infty}\,,\, \phi_{j, \infty} \big)_{L^2}\,.
	\end{aligned}
	\end{equation}
	By a similar computation, we also have
	\begin{equation}\label{ferrarino 2}
	\lambda_{j, \infty} = \fB_{\infty}(\phi_{j, \infty}) \leq \fB_{\fm}(\phi_{j, \infty}) + |c_j(\fm, \tE)|\,.
\end{equation}
	By taking the infimum over $\psi \in E_j^\bot$ with $\| \psi \|_{L^2} = 1$ one gets the two inequalities 
	\begin{equation}\label{ferrarino 3}
	\begin{aligned}
	\inf\{ \fB_{\fm}(\psi) : \psi \in E_j^\bot, \ \| \psi \|_{L^2} = 1 \}  \leq  \lambda_{j, \infty} + |c_j(\fm, \tE)|\,, \\
	\lambda_{j, \infty} \leq \inf\{ \fB_{\fm}(\psi) : \psi \in E_j^\bot, \ \| \psi \|_{L^2} = 1 \} + |c_j(\fm, \tE)|
	\end{aligned}
	\end{equation}
	and then, by taking the supremum over $E \in {\cal E}_{j - 1}$, one obtains that 
	$$
	\mu_{j, \fm}(\tE) \leq \mu_{j, \infty}(\tE) + |c_j(\fm, \tE)|\,, \quad \mu_{j, \infty}(\tE) \leq \mu_{j, \fm}(\tE) + |c_j(\fm, \tE)|
	$$
	namely 
	$$
	|\mu_{j, \fm}(\tE) - \mu_{j, \infty}(\tE)| \leq |c_j(\fm, \tE)|
	$$
	It remains to estimate the term $c_j(\fm, \tE) = \big( (Q_\infty(y) - Q_\fm(y)) \phi_{j, \infty}\,,\, \phi_{j, \infty} \big)_{L^2}$. By the Cauchy-Schwartz inequality, using that $\| \phi_{j, \infty} \|_{L^2} = 1$, one has 
	\begin{equation}
	\begin{aligned}
	\big| \big( &(Q_\infty(y)  - Q_\fm(y)) \phi_{j, \infty}\,,\, \phi_{j, \infty} \big)_{L^2} \big|   \leq \| (Q_\infty(y) - Q_\fm(y)) \phi_{j, \infty} \|_{L^2} \| \phi_{j, \infty} \|_{L^2}  \\& \leq  \| Q_\fm - Q_\infty\|_{L^2}  \| \phi_{j, \infty} \|_{L^\infty} 
	 \stackrel{\eqref{psi j 0 L infty}}{\leq} C(\tE_1, \tE_2, \kappa_0) \| Q_\fm - Q_0\|_{L^2} \to 0 \quad \text{as} \quad \fm \to 0
	\end{aligned}
	\end{equation}
	by Lemma \ref{lemma QmQinfty}, uniformly with respect to $\tE \in [\tE_1, \tE_2]$. Hence, we deduce \eqref{vicinanza autovalori in gamma} and, by fixing $\bar\fm = \bar\fm(\tE_1,\tE_2,\kappa_0)\gg 1$ sufficiently large, for any $\fm\geq \bar\fm$ we get
	$-\tE^2<\mu_{1, \fm}(\tE) < \ldots < \mu_{\kappa_0, \fm}(\tE) < 0$ for any $\tE \in [\tE_1, \tE_2]$. since $-\tE^2<\mu_{1, \infty}(\tE) < \ldots < \mu_{\kappa_0, \infty}(\tE) < 0$. This concludes the proof.
\end{proof}

\begin{rem}
	The estimate \eqref{vicinanza autovalori in gamma} actually holds when $j\geq \kappa_0+1$. The proof is essentially identical and it is here omitted, since we are interested only in the full characterization of the negative spectrum. We also remark that, by refining the result of the $L^p$ convergence in Lemma \ref{lemma QmQinfty}, it is possible to show an explicit rate of convergence of \eqref{vicinanza autovalori in gamma} with respect to $\fm\gg1$.
\end{rem}

\section{The nonlinear elliptic systems with oscillating modes}\label{sez.nonlin.lin}

In the previous section, we constructed the
 stream function $\psi_{\fm}(y)$,  which is a steady solution of the Euler equation \eqref{stat.euler.vort}, that locally solves the second-order nonlinear ODE in Theorem \ref{nonlin_eq}. We now go back to the search of $x$-dependent solutions that are perturbations of the shear equilibrium $\psi_{\fm}(y)$. First, in Proposition \ref{prop.F.etax} we suitably modify the local nonlinearities of Theorem \ref{nonlin_eq}, leading to to the elliptic systems in \eqref{elliptic_eq}. Then, we analyse the linearized systems at the equilibrium $\vf \equiv 0$  and the parametrization of the ``spatial phase space'' based on the solutions of such linearized systems. This choice of coordinates is then used to search the solutions for the nonlinear elliptic system \eqref{elliptic_eq} as zeroes of the nonlinear functional \eqref{F_op} via a Nash-Moser implicit function Theorem, whose statement is provided at the end of the section.

\subsection{Regularization of the nonlinearity}

In Theorem \ref{nonlin_eq} we constructed functions $F_{p,\fm}(\psi)$ for any stripe index $p=0,1,...,\kappa_0$ so that the unperturbed stream function $\psi_{\fm}(y)$ solves the equation \eqref{ell_psi_gamma} on each stripe $\tI_{p}$. The functions are $\cC^{S+1}$-continuous and they satisfy the regularity conditions \eqref{cont Fp critical pts} at the critical points $(\ty_{p,\fm})_{p=1,...,\kappa_0}$. Even without taking derivatives, these conditions are clearly violated by perturbation of $\psi_{\fm}(y)$, in the sense that, for generic functions $\vf(x,y)$, we have
\begin{equation}
	\lim_{y\to\pm\tr^{-}} F_{p-1,\fm}(\psi_{\fm}(y)+\vf(x,y)) \neq \lim_{y\to\pm\tr^{+}} F_{p,\fm}(\psi_{\fm}(y)+\vf(x,y)) \quad \forall\,x\in\R\,.
\end{equation}
The continuity would be recovered for these nonlinearities only we ask $\vf(x,\ty_{p,\fm})\equiv 0$ in $x\in\R$, which is a too strong condition. To this end, we need to modify the nonlinear functions to avoid this issue around the critical points and accommodate small perturbations of $\psi_{\fm}(y)$. We introduce a small parameter $\eta>0$. For any stripe index $p=0,...,\kappa_0$, we define the functions,
\begin{equation}\label{F_eta_reg}
	\begin{aligned}
		F_{p,\eta}(\psi):= F_{p,\fm}(\psi)& +\tfrac12 \chi_{\eta}(\psi - \psi_{\fm}(\ty_{p,\fm})) \big( F_{p-1,\fm}(\psi)- F_{p,\fm}(\psi) \big) \\
		& +\tfrac12 \chi_{\eta}(\psi - \psi_{\fm}(\ty_{p+1,\fm})) \big( F_{p+1,\fm}(\psi)- F_{p,\fm}(\psi) \big)\,,
	\end{aligned}
\end{equation}
with $F_{-1,\fm}:=F_{0,\fm}$, $\ty_{\kappa_0+1,\fm}:=1$, $F_{\kappa_0+1,\fm}:=F_{\kappa_0,\fm}$, and
where the cut-off function $\chi_\eta$ has the following form 
\begin{equation}\label{cutoff rego}
	\begin{aligned}
	& \chi_\eta(\psi) := \chi(\psi/\eta)\,, \quad \chi \in C^\infty(\R), \quad \chi(\psi) = \chi(- \psi)\,, \\
	& \quad 0 \leq \chi \leq 1\,, \quad \chi\equiv 1 \ \text{ on } \ B_{1}(0)\,, \quad \chi\equiv 0 \ \text{ on } \ \R \setminus B_{2}(0)\,, \\
		& \quad \chi'(|\psi|) \leq 0\,, \quad \forall \psi \in \R\,. 
	\end{aligned}
\end{equation}

\begin{prop}\label{prop.F.etax}
	{\bf (Modified local nonlinearities).}
	The following hold:
	\\[1mm]
	\noindent $(i)$ The functions $F_{p,\eta}$ in \eqref{F_eta_reg} are in $\cC_0^{S+1}(\R)$. Moreover, for any stripe index $p$, we have $\| F_{p,\eta}-F_{p,\fm}\|_{L^\infty(\R)} \to 0$  as $\eta\to 0$;
	\\[1mm]
	\noindent $(ii)$  We have $F_{p-1,\eta}=F_{p,\eta}$ on $ B_{\eta}(\psi_{\fm}(\ty_{p,\fm}))$  for any $p=1,...,\kappa_0$. As a consequence, for any sufficiently smooth function $\vf(\bx,y)$ sufficiently small in $L^\infty$
	the regularity conditions
	\begin{equation}\label{cont.Fp.eps.intro}
		\lim_{|y|\to\ty_{p,\fm}^-} \pa_{(\bx,y)}^{n} (F_{p-1,\eta}(\psi_{\fm}(y) +\vf(\bx,y)) )= \lim_{|y|\to\ty_{p,\fm}^+}\pa_{(\bx,y)}^{n}( F_{p,\eta}(\psi_{\fm}(y)+\vf(\bx,y)) )  \,.
	\end{equation}
	are satisfied  for $n\in\N_0^{\kappa_0+1}$, $0\leq |n|\leq S$, and for any  $\bx\in\T^{\kappa_{0}}$, $p=1,...,\kappa_0$;
\\[1mm]
\noindent $(iii)$
For any $n \in \N_0$, with $n \leq  S + 1$, one has 
\begin{equation}\label{stima brutta F p eta}
\sup_{\psi \in \R} |F_{p, \eta}^{(n)}(\psi)| \lesssim \eta^{- n}\,;
\end{equation}
	\noindent $(iv)$  
	There exists $\bar\eta=\bar\eta(\fm,S)>0$ small enough such that, for any $\eta\in [0,\bar\eta]$ and for any $n=0,1,...,S+1$
		\begin{equation}\label{esti close Fpeta}
			\sup_{y\in [-1,1]}\sup_{p=0,1,...,\kappa_0}|F_{p,\eta}^{(n)}(\psi_{\fm}(y)) - F_{p,\fm}^{(n)}(\psi_{\fm}(y)) | \lesssim_{n} \eta^{S+\tfrac32 -n}  \,.
	\end{equation}
\end{prop}
\begin{proof}
	The proof of items $(i)$ and $(ii)$ are a direct consequence of \eqref{F_eta_reg}, \eqref{cutoff rego} and of Theorem \ref{nonlin_eq}. The item $(iii)$ follows by an explicit calculation, by differentiating the formula \eqref{F_eta_reg} and using that the cut off function $\chi_\eta$ in \eqref{cutoff rego} satisfies $|\partial_\psi^n \chi_\eta(\psi)| \lesssim_n \eta^{- n}$. 
	
	\smallskip
	
	\noindent
	We now prove item $(iv)$. By \eqref{F_eta_reg}, \eqref{cutoff rego}, the claim trivially holds when
	$\psi_{\fm}(y) \notin  B_{2\eta} (\psi_{\fm}(\ty_{p,\fm}))\cup  B_{2\eta} (\psi_{\fm}(\ty_{p+1,\fm}))$.
	 Therefore,  let $\psi_{\fm}(y) \in B_{2\eta} (\psi_{\fm}(\ty_{p,\fm}))$. The case $\psi_{\fm}(y) \in  B_{2\eta} (\psi_{\fm}(\ty_{p+1,\fm}))$ works similarly and we omit it. 
	 By Lemma \ref{lemma monotono}, for $\eta>0$ sufficiently small, there exist $\delta^{\pm}
	 = \delta^{\pm}(p,\fm,\tE)>0$ such that $\psi_{\fm}(y) \in B_{2\eta} (\psi_{\fm}(\ty_{p,\fm}))$ is parametrized by
	\begin{equation}\label{cut off varepsilon}
	\begin{aligned}
		& \psi_{\fm}(y) = \psi_{\fm}(\ty_{p,\fm}-\delta^{-}) = \psi_{\fm}(\ty_{p,\fm}+\delta^{+})\,, \\
		& \ty_{p,\fm}-\delta^{-}\in \tI_{p-1}\,,\quad \ty_{p,\fm}+\delta^{+} \in \tI_{p}\,.
	\end{aligned}
\end{equation}
	In particular, by \eqref{cut off varepsilon}, Lemma \ref{lemma monotono}-$(iii)$ and the mean value Theorem,
	 we get
	\begin{equation}\label{sunday}
		\begin{aligned}
			\vf = \vf(\delta^{\pm}) & := \psi_{\fm}(\ty_{p,\fm}\pm \delta^{\pm}) - \psi_{\fm}(\ty_{p,\fm}) = \int_{0}^{\pm\delta^{\pm}} \psi_{\fm}'(\ty_{p,\fm}+\delta_1)\wrt \delta_1 \\
			& =  \int_{0}^{\pm\delta^{\pm}} \int_{0}^{\delta_1} \psi_{\fm}''(\ty_{p,\fm}+\delta_2)\wrt \delta_2 \wrt \delta_1 \simeq \tfrac{\psi_{\fm}''(\ty_{p,\fm})}{2} (\delta^{\pm})^2\,,
		\end{aligned}
	\end{equation}
	for $\delta^{\pm}$ sufficiently small, from which we deduce that 
	\begin{equation}\label{keyclaim1}
		\delta^{\pm}(\vf) = O(\vf^\frac12) = O(\eta^{\frac12}) \,, \quad \eta \to 0\,.
	\end{equation}
	We claim now that
	\begin{equation}\label{keyclaim2}
		(\delta^{+})^2 - (\delta^{-})^2 = O(\eta^{S+\frac32})\,.
	\end{equation}
	To see this, we recall the expansion in Lemma \ref{expansion psim crit}. By \eqref{cut off varepsilon}, we compute
	\begin{equation}
		\begin{aligned}
		0 & = \psi_{\fm}(\ty_{p,\fm}+\delta^{+}) - \psi_{\fm}(\ty_{p,\fm}-\delta^{-})   = \sum_{n=0}^{S+1} \frac{\psi_{\fm}^{(2n)}(\ty_{p,\fm})}{(2n)!}\big( (\delta^{+})^{2n} -(\delta^{-})^{2n} \big) \\
		& \quad + \frac{\psi_{\fm}^{(2S+3)}(\ty_{p,\fm})}{(2S+3)!}\big( (\delta^{+})^{2S+3}+(\delta^{-})^{2S+3} \big) + o\big( |\delta^{+}|^{2(S+2)} + |\delta^{-}|^{2(S+2)} \big)
		\end{aligned}
	\end{equation}
	It follows that
	\begin{equation}
		\begin{aligned}
			& \frac{\psi_{\fm}^{''}(\ty_{p,\fm})}{2}\big( (\delta^{+})^{2} -(\delta^{-})^{2} \big) \Big( 1 +  o\big( (\delta^{+})^{2} + (\delta^{-})^{2} \big) \Big)  \\
			& \quad = - \frac{\psi_{\fm}^{(2S+3)}(\ty_{p,\fm})}{(2S+3)!}\big( (\delta^{+})^{2S+3}+(\delta^{-})^{2S+3} \big) + o\big( |\delta^{+}|^{2(S+2)} + |\delta^{-}|^{2(S+2)} \big)\,.
		\end{aligned}
	\end{equation}
	Therefore, by Lemma \ref{lemma monotono}-$(iii)$ and \eqref{keyclaim1}, for $\eta>0$ sufficiently small, we can invert the factor $\frac{\psi_{\fm}^{''}(\ty_{p,\fm})}{2}\big( 1 +  o\big( (\delta^{+})^{2} + (\delta^{-})^{2} \big) \big)$ and deduce the claim in \eqref{keyclaim2} as desired.
	
	We are now ready to prove \eqref{esti close Fpeta}. We start with $n=0$.
	By \eqref{F_eta_reg}, \eqref{cut off varepsilon}, \eqref{sunday}, \eqref{Qgamma}, Theorem \ref{nonlin_eq}, Lemma \eqref{expansion psim crit} and the mean value Theorem, we compute
	\begin{equation}\label{compu 1}
		\begin{aligned}
			& F_{p,\eta}(\psi_{\fm}(y)) - F_{p,\fm}(\psi_{\fm}(y)) = \tfrac12\chi_{\eta}(\vf) \big( F_{p-1,\fm} (\psi_{\fm}(y) ) - F_{p,\fm}(\psi_{\fm}(y) \big)\\
			& =\tfrac12\chi_{\eta}(\vf)\big( F_{p-1,\fm} (\psi_{\fm}(\ty_{p,\fm}-\delta^{-}))  - F_{p,\fm}(\psi_{\fm}(\ty_{p,\fm}+\delta^{+}))\big) \\
			& =\tfrac12\chi_{\eta}(\vf)\big( \psi_{\fm}''(\ty_{p,\fm}-\delta^{-})  - \psi_{\fm}''(\ty_{p,\fm}+\delta^{+})\big) \\ 
			& =\tfrac12 \chi_{\eta}(\vf)\int_{\delta^{+}}^{-\delta^{-}} \psi_{\fm}'''(\ty_{p,\fm}+\delta_1) \wrt \delta_1 =\tfrac12 \chi_{\eta}(\vf)\int_{\delta^{+}}^{-\delta^{-}}\big( Q_{\fm} \psi_{\fm}' \big)(\ty_{p,\fm}+\delta_1) \wrt \delta_1\\
			& =\tfrac12 \chi_{\eta}(\vf)\int_{\delta^{+}}^{-\delta^{-}}\int_{0}^{\delta_1} \big( Q_{\fm}\psi_{\fm}'  \big)'(\ty_{p,\fm}+\delta_2)  \wrt \delta_2 \wrt \delta_1 \\
			& =\tfrac12 \chi_{\eta}(\vf)\int_{\delta^{+}}^{-\delta^{-}}\int_{0}^{\delta_1} \big( Q'_{\fm}\psi_{\fm}' + Q_{\fm} \psi_{\fm}'' \big)(\ty_{p,\fm}+\delta_2)  \wrt \delta_2 \wrt \delta_1
		\end{aligned}
	\end{equation}
	By \eqref{compu 1}, \eqref{sunday}, \eqref{cutoff rego} Lemma \ref{lemma QmQinfty} and \eqref{keyclaim2}, we conclude that, for $\eta$ sufficiently small,
	\begin{equation}\label{stima bella 1}
		|	F_{p,\eta}(\psi_{\fm}(y)) - F_{p,\fm}(\psi_{\fm}(y))| \lesssim  	\big| (\delta^{+})^2 - (\delta^{-})^2 \big|  \lesssim \eta^{S+\frac32}\,.
	\end{equation}
	Therefore the claim holds for $n=0$. Let now $n\in\{1,...,S+1\}$.
	First, we look for the derivatives of $F_{p-1,\fm}(\psi_{\fm}(y))-F_{p,\fm}(\psi_{\fm}(y))$, still with $\psi_{\fm}(y)$ as in \eqref{cut off varepsilon}. By Corollary \ref{cor:asimptotics.close.pc} and the mean value Theorem, we have
		\begin{align}
			&F_{p-1,\fm}^{(n)}(\psi_{\fm}(y)) - F_{p,\fm}^{(n)}(\psi_{\fm}(y))  = F_{p-1,\fm}^{(n)}(\psi_{\fm}(\ty_{p,\fm}-\delta^{-})) - F_{p,\fm}^{(n)}(\psi_{\fm}(\ty_{p,\fm}+\delta^{+})) \nonumber \\
			& = \fP_{n-1}(\ty_{p,\fm}-\delta^{-}) - \fP_{n-1}(\ty_{p,\fm}+\delta^{+})  = \int_{\delta^{+}}^{-\delta^{-}}  \fP_{n-1} '(\ty_{p,\fm}+\delta_1) \wrt \delta_1\,. \label{deriva0}
		\end{align}
By \eqref{expan.around.cp} in Corollary \ref{cor:asimptotics.close.pc}, it is clear, for $n=1,...,S$, that $\fP_{n-1}'(\ty_{p,\tm})=0$. 
Therefore, again by the mean value Theorem, we obtain
\begin{equation}\label{deriva2}
    \begin{aligned}
        \fP_{n-1} '(\ty_{p,\fm}+\delta_1) = \int_{0}^{\delta_1} \fP_{n-1} ''(\ty_{p,\fm}+\delta_2) \wrt\delta_2\,. 
    \end{aligned}
\end{equation}
By Lemmata \ref{lemma QmQinfty}, \ref{lemma monotono}, the integrand in \eqref{deriva2} is uniformly bounded in the domain of integration for $\eta>0$ sufficiently small. Therefore, collecting \eqref{deriva0} with \eqref{deriva2}, we have the estimate
\begin{equation}\label{stima bella 2}
	| F_{p-1,\fm}^{(n)}(\psi_{\fm}(y)) - F_{p,\fm}^{(n)}(\psi_{\fm}(y)) | \lesssim_{n} 	\big| (\delta^{+})^2 - (\delta^{-})^2 \big|  \lesssim \eta^{S+\frac32}\,.
\end{equation}
We finally prove the claimed estimate \eqref{esti close Fpeta} for $n\in \{1,...,S+1\}$. By Leibniz rule and by \eqref{cut off varepsilon}, we have
\begin{equation}
		\begin{aligned}
			& \pa_{\psi}^{n}\big(F_{p,\eta}(\psi)-F_{p,\fm}(\psi)\big) = \sum_{j=0}^{n}\binom{n}{j} \chi_{\eta}^{(n-j)}(\vf) \big( F_{p-1}^{(j)}(\psi) - F_{p,\fm}^{(j)}(\psi) \big)
		\end{aligned}
\end{equation}
The claim \eqref{esti close Fpeta} then follows by \eqref{deriva2} and \eqref{cutoff rego}, recalling that $|\chi_{\eta}^{n-j}(\vf)| \leq \eta^{-(n-j)}$. This concludes the proof of the proposition.
\end{proof}

\subsection{The Hamiltonian formulation}

The manifold of the zeroes \eqref{euler.pert.1} admits a formulation as an Hamiltonian vector field. 
First, for fixed $\eta>0$, we rewrites the elliptic equations in \eqref{euler.pert.1} as
the second order forced PDE
\begin{equation}\label{elliptic_eq}
	\begin{cases}
		(\omega\cdot \pa_\bx)^2\vf -\cL_{\fm}\vf -g_{\eta}(y,\vf)= f_{\eta}(y)\,, \quad (\bx,y)\in \T^{\kappa_0} \times [-1,1]\,, \\
		\vf(\bx, -1) = \vf(\bx, 1) = 0 \,, \quad \omega \in\R^{\kappa_0} \,,
	\end{cases}
\end{equation}
where the operator $\cL_{\fm}$ is as in Theorem \ref{L_operator}, the forcing term $f_{\eta}(y)$ is defined by
\begin{equation}\label{fpeta}
	f_{\eta}(y):= \sum_{p=0}^{\kappa_0} \chi_{\tI_{p}}(y) f_{p,\eta}(y)\,, \quad   f_{p,\eta}(y):= F_{p,\eta}(\psi_{\fm}(y)) - F_{p,\fm}(\psi_{\fm}(y))\,,
\end{equation}
 and the nonlinear function $g_{\eta}(y,\vf)$ is defined by
\begin{equation}\label{nonlin_g_eta}
\begin{aligned}
		& g_{\eta}(y,\vf):= \sum_{p=0}^{\kappa_0} \chi_{\tI_{p}}(y) g_{p,\eta}(y,\vf)\,, \\
		&g_{p,\eta}(y,\vf)	:=F_{p,\eta}(\psi_{\fm}(y)+\vf) - F_{p,\eta}(\psi_{\fm}(y))- F_{p,\fm}'(\psi_{\fm}(y)) \vf \\
		& \ \ =  
		\big( F_{p,\eta}'(\psi_{\fm}(y)) - F_{p,\fm}'(\psi_{\fm}(y)) \big) \vf + \int_{0}^{1} (1-\varrho ) F_{p,\eta}''(\psi_{\fm}(y)+\varrho \vf) \wrt\varrho \, \vf^2\,;
\end{aligned}
\end{equation}
here $\chi_{\tI_{p}}$ denotes the characteristic function for the interval $\tI_{p}$,   $F_{p,\eta}(\psi)$ is as in \eqref{F_eta_reg} and $F_{p,\fm}(\psi)$ is as in Theorem \ref{nonlin_eq}. Note that, by \eqref{cont Fp critical pts} and Proposition \ref{prop.F.etax}
the compatibility conditions
\begin{equation}
	\begin{aligned}
			&	\lim_{|y|\to\ty_{p,\fm}^-} \pa_{(\bx,y)}^{n} \big( g_{p-1,\eta}(y,\vf(\bx,y))\big) = \lim_{|y|\to\ty_{p,\fm}^+} \pa_{(\bx,y)}^{n} \big( g_{p,\eta}(y,\vf(\bx,y))\big)\,,  
		\end{aligned}
\end{equation}
hold for any $  n\in\N_0^{\kappa_0+1}$, $ 0\leq |n|\leq S $, and $ p=1,...,\kappa_0$, assuming the smallness condition, for $s_0 \leq s \leq q(S,k_0)$,
\begin{equation}
	\| \vf \|_{s_0,1}^{k_0,\upsilon} \leq C \varepsilon \ll 1\,,
\end{equation}
which implies the small bound in $L^\infty$. We now provide a Lemma in which we estimate the forcing term $f_\eta$ and the nonlinearity $g_\eta$ in \eqref{fpeta}, \eqref{nonlin_g_eta}. 
\begin{lem}\label{estimates f eta g eta}
The following estimates hold:
\\[1mm]
\noindent
$(i)$ {\bf (Estimates in $H^1_y$).} We have $\| f_\eta \|_{H^1} \lesssim \eta^{S + \frac12}$. Moreover, the composition operator $\vphi \in B_1(0) \mapsto g_\eta(y, \vphi(y)) \in H_0^1([- 1, 1])$, with $B_1(0) := \{ \vphi \in H_0^1([- 1, 1]) : \| \vphi \|_{H^1} \leq 1\}$, satisfies the estimates
$$
\begin{aligned}
& \| g_\eta(\cdot, \vphi) \|_{H^1} \lesssim \eta^{S - \frac12}\| \vphi \|_{H^1} + \eta^{- 3} \| \vphi \|_{H^1}^2\,, \quad \\
&  \| \di g_\eta(\cdot , \vphi)[h] \|_{H^1} \lesssim \eta^{S - \frac12} \| h \|_{H^1} + \eta^{- 4} \| \vphi \|_{H^1} \| h \|_{H^1}\,. 
\end{aligned}
$$
\\[1mm]
\noindent
$(ii)$ {\bf (Estimates in $H^s_{\bf x} H^1_y$).} Assume $\| \vphi \|_{s_0, 1}^{k_0,\upsilon} \leq 1$. Then there exists $\sigma > 0$ such that, for $S> s_0 + \sigma$ large enough and for any $s_0 \leq s \leq S- \sigma $, one has 
\begin{equation}\label{stime composizione g p eta}
\begin{aligned}
 \| g_\eta (\cdot , \vphi) \|_{s, 1}^{k_0,\upsilon} & \lesssim_s  \eta^{S - \frac12}\| \vphi \|_{s, 1}^{k_0,\upsilon} + \eta^{- (s + \sigma)} \| \vphi \|_{s_0, 1}^{k_0,\upsilon} \| \vphi \|_{s, 1}^{k_0,\upsilon}\,; \\
 \| \di g_\eta (\cdot , \vphi)[h] \|_{s, 1}^{k_0,\upsilon} &  \lesssim_s  \eta^{S - \frac12}\| h\|_{s, 1}^{k_0,\upsilon} + \eta^{- (s + \sigma)} \| \vphi \|_{s, 1}^{k_0,\upsilon}  \| h \|_{s_0, 1}^{k_0,\upsilon}  \,; \\
 \| \di^2 g_\eta (\cdot , \vphi)[h_1, h_2] \|_{s, 1}^{k_0,\upsilon} & \lesssim_s   \eta^{- (s + \sigma)}\big(  \| h_1 \|_{s_0, 1}^{k_0,\upsilon} \| h_2 \|_{s, 1}^{k_0,\upsilon} +  \| h_1 \|_{s, 1}^{k_0,\upsilon} \| h_2 \|_{s_0, 1}^{k_0,\upsilon} \\
& \quad \quad \quad \quad   + \| \vphi \|_{s, 1}^{k_0,\upsilon} \| h_1 \|_{s_0, 1}^{k_0,\upsilon} \| h_2 \|_{s_0, 1}^{k_0,\upsilon}  \big)\,. 
\end{aligned}
\end{equation}

\end{lem}
\begin{proof}
{\sc Proof of $(i)$.} By the definition of $f_\eta$, $g_\eta$  in \eqref{fpeta}, \eqref{nonlin_g_eta}, it suffices to estimate $f_{p, \eta}$, $g_{p, \eta}$ for any $p = 1, \ldots, \kappa_0$. By  \eqref{esti close Fpeta} applied with $n =  0, 1$, one obtains that 
\begin{equation}\label{bla bla car 10}
\begin{aligned}
\| f_{p, \eta} \|_{H^1} \leq \| F_{p,\eta}(\psi_{\fm} ) - F_{p,\fm}(\psi_{\fm} ) \|_{{\cal C}^1} \lesssim \eta^{S + \frac12}
\end{aligned}
\end{equation}
which implies the claimed bound on $f_{p, \eta}$. In order to estimate $g_{p, \eta}$, we write 
\begin{equation}\label{stima g p eta nel lemma}
\begin{aligned}
 g_{p, \eta}(y, \vphi)&  = {\cal I}_1 + {\cal I}_2 \,, \\
 {\cal I}_1 (\vphi) & := \big( F_{p,\eta}'(\psi_{\fm}(y)) - F_{p,\fm}'(\psi_{\fm}(y)) \big) \vf \,, \quad   \\
{\cal I}_2(\vphi) &  := \int_{0}^{1} (1-\varrho ) F_{p,\eta}''(\psi_{\fm}(y)+\varrho \vf) \wrt\varrho \, \vf^2\,. 
\end{aligned}
\end{equation}
and we estimate ${\cal I}_1$ and ${\cal I}_2$ separately. 

\smallskip

\noindent
{\bf Estimate of ${\cal I}_1(\vphi)$.} By the algebra property of $H_0^1$ and the estimate \eqref{esti close Fpeta} applied with $n =  1, 2$, we have
$$
\begin{aligned}
\|{\cal I}_1(\vphi) \|_{H^1} &  \lesssim  \| \big( F_{p,\eta}'(\psi_{\fm}(y)) - F_{p,\fm}'(\psi_{\fm}(y)) \big) \|_{H^1} \|  \vphi \|_{H^1}  \\
& \lesssim \| \big( F_{p,\eta}'(\psi_{\fm}(y)) - F_{p,\fm}'(\psi_{\fm}(y)) \big) \|_{{\cal C}^1} \|  \vphi \|_{H^1}   \stackrel{\eqref{bla bla car 10}}{\lesssim} \eta^{S - \frac12} \|\vphi \|_{H^1}\,,
\end{aligned}
$$
and similarly $\| \di {\cal I}_1(\vphi) [h] \|_{H^1} \lesssim \eta^{S - \frac12} \| h \|_{H^1}$, since the map $\vphi \mapsto {\cal I}_1(\vphi)$ is linear. 

\smallskip

\noindent
{\bf Estimate of ${\cal I}_2(\vphi)$.} The differential of ${\cal I}_2$ is given by 
\begin{equation}\label{diffe.1.I2}
	\begin{aligned}
		\di {\cal I}_2(\vphi)[h] &= 2 \int_{0}^{1} (1-\varrho ) F_{p,\eta}''(\psi_{\fm}(y)+\varrho \vf) \wrt\varrho \, \vf h \\
		& \ \ \, + \int_{0}^{1} (1-\varrho ) F_{p,\eta}'''(\psi_{\fm}(y)+\varrho \vf) \varrho \wrt\varrho \, \vf^2 h\,. 
	\end{aligned}
\end{equation}
By Proposition \ref{prop.F.etax}-$(iii)$, for $\| \vphi \|_{H^1} \leq 1$, a direct calculation shows that 
\begin{equation}\label{pallavolo 0}
\sup_{\varrho \in [0, 1]} \| F_{p,\eta}^{(k)}(\psi_{\fm}(y)+\varrho \vf) \|_{H^1} \lesssim \| F \|_{{\cal C}^4} \lesssim \eta^{- 4}  \quad k = 2,3\,. 
\end{equation}
The latter estimate, together with the algebra property of $H_0^1$ implies that 
$$
\begin{aligned}
\| {\cal I}_2(\vphi) \|_{H^1} \lesssim \eta^{- 3} \| \vphi \|_{H^1}^2 \,, \quad \| \di {\cal I}_2(\vphi)[h] \|_{H^1} \lesssim \eta^{- 4}  \| \vphi \|_{H^1} \| h \|_{H^1}
\end{aligned}
$$

\noindent
{\sc Proof of $(ii)$.} By \eqref{nonlin_g_eta}, it is enough to estimate $g_{p, \eta}(y, \vphi)$ for any $p = 1, \ldots, \kappa_0$ and, according to \eqref{stima g p eta nel lemma}, we estimate ${\cal I}_1$ and ${\cal I}_2$ separately.  

\medskip

\noindent
{\bf Estimate of ${\cal I}_1(\vphi)$.} By Lemma \ref{prod.lemma} and by applying again the estimate \eqref{esti close Fpeta} with $n = 0, 1$, one obtains that 
$$
\begin{aligned}
 \| {\cal I}_1(\vphi) \|_{s, 1}^{k_0,\upsilon} & =  \|  \big( F_{p,\eta}'(\psi_{\fm}(y)) - F_{p,\fm}'(\psi_{\fm}(y)) \big) \vf  \|_{s, 1}^{k_0,\upsilon} \\
 & \lesssim \| F_{p,\eta}'(\psi_{\fm}(y)) - F_{p,\fm}'(\psi_{\fm}(y)) \|_{H^1} \| \vphi \|_{s, 1}^{k_0, \upsilon} \\
 & \lesssim \| F_{p,\eta}'(\psi_{\fm}(y)) - F_{p,\fm}'(\psi_{\fm}(y)) \|_{{\cal C}^1} \| \vphi \|_{s, 1}^{k_0, \upsilon} \lesssim \eta^{S - \frac12} \| \vphi\|_{s, 1}^{k_0, \upsilon}\,.
\end{aligned}
$$
The estimates for $\di {\cal I}_1(\vphi)$ and $\di^2 {\cal I}_1(\vphi)$ follows similarly since ${\cal I}_1$ is linear with respect to $\vphi$. 

\smallskip

\noindent
{\bf Estimate of ${\cal I}_2(\vphi)$.} The first differential of ${\cal I}_2(\vphi)$ in \eqref{stima g p eta nel lemma} is given in \eqref{diffe.1.I2}, whereas the second differential has the form 
$$
\begin{aligned}
\di^2 {\cal I}_2(\vphi)[h_1, h_2] & = 2 \int_{0}^{1} (1-\varrho ) F_{p,\eta}''(\psi_{\fm}(y)+\varrho \vf) \wrt\varrho \, h_2 h_1\\ & + 3 \int_{0}^{1} (1-\varrho ) F_{p,\eta}''' (\psi_{\fm}(y)+\varrho \vf) \varrho \wrt\varrho \, \vphi  \,h_1 h_2 \\ 
 & \ \ \, + \int_{0}^{1} (1-\varrho ) F_{p,\eta}^{(4)}(\psi_{\fm}(y)+\varrho \vf) \varrho^2 \wrt\varrho \, \vf^2 h_1 h_2  \,. 
\end{aligned}
$$By applying Proposition \ref{prop.F.etax}-$(iii)$ and  the composition lemma \ref{compo_moser}, for $\| \vphi \|_{s_0, 1}^{k_0, \upsilon} \leq 1$, some $\sigma > 0$, one has that for any $s_0 \leq s \leq S - \sigma$, 
\begin{equation}\label{pallavolo.1}
\sup_{\varrho \in [0, 1]} \| F_{p,\eta}^{(k)}(\psi_{\fm}(y)+\varrho \vf) \|_{s, 1}^{k_0, \upsilon} \lesssim_s  \eta^{- (s + \sigma)} (1 + \| \vphi \|_{s, 1}^{k_0, \upsilon})\,, \quad k = 2,3,4\,. 
\end{equation}
Then by the explicit expressions of ${\cal I}_2, \di{\cal I}_2, \di^2{\cal I}_2$, using the estimate \eqref{pallavolo.1}, Lemma \ref{prod.lemma} and $\| \vphi \|_{s_0, 1}^{k_0, \upsilon} \leq 1$, one gets, for any $s_0 \leq s \leq S - \sigma$, the estimates
$$
\begin{aligned}
 \| {\cal I}_2(\vphi) \|_{s, 1}^{k_0,\upsilon} &  \lesssim_s  \eta^{- (s + \sigma)} \| \vphi \|_{s_0, 1}^{k_0,\upsilon} \| \vphi \|_{s, 1}^{k_0,\upsilon}\,, \\
 \| \di{\cal I}_2(\vphi)[h] \|_{s, 1}^{k_0,\upsilon} & \lesssim_s   \eta^{- (s + \sigma)} \big( \| h \|_{s_0, 1}^{k_0,\upsilon} \| \vphi \|_{s, 1}^{k_0,\upsilon} +  \| h \|_{s_0, 1}^{k_0,\upsilon} \| \vphi \|_{s, 1}^{k_0,\upsilon} \big) \,, \\
 \| \di^2 {\cal I}_2(\vphi)[h_1, h_2] \|_{s, 1}^{k_0,\upsilon} & \lesssim_s   \eta^{- (s + \sigma)}\big(  \| h_1 \|_{s_0, 1}^{k_0,\upsilon} \| h_2 \|_{s, 1}^{k_0,\upsilon}  + \| h_1 \|_{s, 1}^{k_0,\upsilon} \| h_2 \|_{s_0, 1}^{k_0,\upsilon}  \\ 
 &  \quad \quad \quad \quad + \| \vphi \|_{s, 1}^{k_0,\upsilon} \| h_1 \|_{s_0, 1}^{k_0,\upsilon} \| h_2 \|_{s_0, 1}^{k_0,\upsilon}  \big)\,. 
\end{aligned}
$$
By the previous arguments, one easily deduces the claimed bound \eqref{stime composizione g p eta}. 
\end{proof}

It is actually possible to perform an affine transformation of the unknown in order to remove the forcing term $f_{\eta}(y)$ in \eqref{elliptic_eq}.
\begin{lem}\label{unforcing.lemma}
	Let $\bar\eta\ll 1$ as in Proposition \ref{prop.F.etax} and $S > 4$. Then, for any $\eta\in[0,\bar\eta]$, there exists a function $h_{\eta}(y) \in H_0^3[-1,1]$, even in $y$ and with $h_{\eta}(0)=0$, such that $\| h_{\eta}\|_{H_y^3}\leq \eta^S$ and
	\begin{equation}\label{unforcing.eq}
		-\cL_{\fm} h_{\eta}(y) - g_{\eta}(y,h_{\eta}(y)) = f_{\eta}(y)\,.
	\end{equation}
\end{lem}
\begin{proof}
	First, by \eqref{fpeta}, we note that $f_{\eta}(y)$ is even and that $f_{\eta}(0)=f_{\eta}(\pm 1)=0$. The same holds for $g_{\eta}(y,h(y))$ in \eqref{nonlin_g_eta}, assuming $h(y)\equiv h_{\eta}(y)$ even in $y$ and with $h(0)=h(\pm 1)=0$. Moreover, by Proposition \ref{L_operator} and 
	 the classical Sturm-Liouville theory for Schr\"odinger operators with smooth potentials,
	$0$ is not an eigenvalue for $\cL_{\fm}$ and  the inverse operator $\cL_{\fm}^{-1}:H_{0}^{1}[-1,1] \to H_{0}^{3}[-1,1]$ is a well defined smoothing operator. Therefore, we reformulate equation \eqref{unforcing.eq} as the fixed point equation
	\begin{equation}\label{fixed.point.h}
		h(x) = T_{\eta}(h(x))\,, \quad T_{\eta}(h):= (-\cL_{\fm})^{-1}\big( f_{\eta}(y) + g_{\eta}(y,h) \big)\,.
	\end{equation}
	We define the domain,
	\begin{equation}
		B_{\eta} := \big\{ h\in H_{0}^{3}[-1,1] \, : \, h(0)=0\,, \ h(-y)=h(y)\,,\ \| h \|_{H_y^3} \leq \eta^S \big\}\,.
	\end{equation}
	By Proposition \ref{prop.F.etax}-$(ii)$ and Lemma \ref{estimates f eta g eta}-$(i)$, for any $h \in B_\eta$ and $\widehat h \in H_0^1([- 1, 1])$, with $\eta \ll 1$ small enough, we have 
	\begin{equation}\label{estimates for contraction in lemma}
	\begin{aligned}
	 \| f_\eta \|_{H^1} &\lesssim \eta^{S + \frac12}\,, \quad \\
	 \| g_\eta(\cdot, h) \|_{H^1} & \lesssim \eta^{S - \frac12}\| h \|_{H^1} + \eta^{- 3} \| h \|_{H^1}^2 \\
	&  \lesssim \eta^{2 S - \frac12} + \eta^{2 S - 3} \lesssim \eta^{2 S - 3}\,, \\
	 \| \di g_\eta(\cdot, h)[\widehat h] \|_{H^1} & \lesssim \eta^{S - \frac12} \| \widehat h \|_{H^1} + \eta^{- 4} \| h \|_{H^1} \| \widehat h \|_{H^1} \\
	& \ \lesssim \big( \eta^{S - \frac12} + \eta^{S - 4} \big) \| \widehat h \|_{H^1}  \lesssim \eta^{S - 4} \| \widehat h \|_{H^1}
	\end{aligned}
	\end{equation}
 By using that $\cL_{\fm}^{-1} : H_{0}^{1}[-1,1] \to H_{0}^{3}[-1,1]$ is linear and continuous and by  \eqref{estimates for contraction in lemma}, one deduces that  the map $T_\eta$ in \eqref{fixed.point.h} satisfies, for any $h \in B_\eta$ and $\widehat h \in H_0^3([- 1, 1])$, 
   $$
 \| T_\eta (h) \|_{H^3} \leq C( \eta^{S + \frac12} + \eta^{2 S - 3})\,, \quad \| d T_\eta(h)[\widehat h] \|_{H^3} \leq C \eta^{S - 4} \| \widehat h\|_{H^3}
 $$  
 for some $C \geq 1$ independent of $\eta¡$. Hence, by the assumption $S>4$, with $\eta \ll 1$ small enough, the map $T_\eta : B_\eta \to B_\eta$ is a contraction, implying that there exists a unique solution $h\in B_{\eta}$ of the equation \eqref{fixed.point.h} by a fixed point argument.
\end{proof}

We introduce the rescaled variable $\zeta := \varepsilon^{-1}\big( \vphi (\bx, y) - h_\eta(y)\big)$.
 At this stage, we also link the parameters $\eta$ and $\varepsilon$ as 
 \begin{equation}\label{link eta varepsilon}
\eta := \varepsilon^{\frac{1}{S}}
\end{equation}
where $S \gg 0$ is the smoothness of the nonlinearity $g_\eta$. 
Hence \eqref{elliptic_eq} in the new rescaled variable becomes 
 \begin{equation}\label{equazione secondo ordine con h eta riscalata}
	\begin{cases}
		(\omega\cdot \pa_\bx)^2\zeta -\cL_{\fm}\zeta - \sqrt{\varepsilon} q_\varepsilon(y, \zeta)= 0\,, \quad (\bx,y)\in \T^{\kappa_0} \times [-1,1],, \\
		\zeta(\bx, -1) = \zeta(\bx, 1) = 0 \,, \quad \omega \in\R^{\kappa_0} \,.
	\end{cases}
 \end{equation} 
 where 
 \begin{equation}\label{definition p eta}
 q_\varepsilon (y, \zeta) := \varepsilon^{- \frac32}\big( g_{\eta}(y,h_\eta(y) + \varepsilon \zeta) - g_\eta(y, h_\eta(y)) \big)\,, \quad \eta = \varepsilon^{\frac{1}{S}}\,. 
 \end{equation}
In the next lemma, we provide some estimates on the rescaled nonlinearity $q_{\varepsilon}$. 
\begin{lem}\label{stime cal Q eta}
Let $C_0 > 0$ and assume $\| \zeta \|_{s_0, 1}^{k_0, \upsilon} \leq C_0$. Let  $S\geq 2(s_0+\sigma)$ with $\sigma>0$ as in Lemma \ref{estimates f eta g eta}-$(ii)$. Then the rescaled nonlinearity $q_\varepsilon$ satisfies the following estimates. For any $s_0 \leq s \leq S/2 - \sigma$, one has 
\begin{equation}\label{stime nonlinearita riscalata}
\begin{aligned}
 \| q_\varepsilon(\cdot, \zeta) \|_{s, 1}^{k_0, \upsilon} & \lesssim_s  \|  \zeta \|_{s, 1}^{k_0, \upsilon}  \,, \\
 \| \di q_\varepsilon(y, \zeta)[\widehat \zeta] \|_{s, 1}^{k_0, \upsilon} &  \lesssim_s   \| \widehat \zeta \|_{s, 1}^{k_0, \upsilon} + \| \zeta \|_{s, 1}^{k_0, \upsilon} \| \widehat \zeta \|_{s_0, 1}^{k_0, \upsilon} \,, \\
 \| \di^2 q_\varepsilon(\cdot, \zeta)[\widehat \zeta_1, \widehat \zeta_2] \|_{s, 1}^{k_0, \upsilon} & \lesssim_s  \| \widehat \zeta_1 \|_{s_0, 1}^{k_0,\upsilon} \| \widehat \zeta_2 \|_{s, 1}^{k_0,\upsilon} +  \| \widehat \zeta_1 \|_{s, 1}^{k_0,\upsilon} \| \widehat \zeta_2 \|_{s_0, 1}^{k_0,\upsilon}  \\
& \ \ + \|  \zeta \|_{s, 1}^{k_0,\upsilon} \| \widehat \zeta_1 \|_{s_0, 1}^{k_0,\upsilon} \| \widehat \zeta_2 \|_{s_0, 1}^{k_0,\upsilon}\,. 
\end{aligned}
\end{equation}
\end{lem}
\begin{proof}
We shall apply the estimates \eqref{stime composizione g p eta} in Lemma \ref{estimates f eta g eta}.  We start by proving the second estimate in \eqref{stime nonlinearita riscalata}. One has
$$
\di q_\varepsilon(y, \zeta)[\widehat \zeta] = \varepsilon^{- \frac12} \di g_{\eta}(y,h_\eta(y) + \varepsilon \zeta)[\widehat \zeta]\,.
$$
Hence, by the second estimate in \eqref{stime composizione g p eta}, Lemma \ref{unforcing.lemma} and \eqref{link eta varepsilon}, one gets that for any $s_0 \leq s \leq S/2 - \sigma$, 
$$
\begin{aligned}
\| \di q_\varepsilon(y, \zeta)[\widehat \zeta]\|_{s, 1}^{k_0, \upsilon} & \lesssim_s  \varepsilon^{- \frac12}\eta^{S - \frac12}\| \widehat \zeta\|_{s, 1}^{k_0,\upsilon} + \varepsilon^{- \frac12} \eta^{- S/2} \| h_\eta + \varepsilon \zeta \|_{s, 1}^{k_0,\upsilon}  \| \widehat \zeta \|_{s_0, 1}^{k_0,\upsilon}  \\
& \lesssim_s \varepsilon^{- \frac12}\eta^{S - \frac12}\| \widehat \zeta\|_{s, 1}^{k_0,\upsilon} + \varepsilon^{- \frac12} \eta^{- S/2}\big( \| h_\eta \|_{H^1} + \varepsilon  \| \zeta \|_{s, 1}^{k_0,\upsilon}  \big) \| \widehat \zeta \|_{s_0, 1}^{k_0,\upsilon}  \\
& \lesssim_s  \varepsilon^{- \frac12}\big( \eta^{S - \frac12} +  \eta^{S/2}  \big) \| \widehat \zeta\|_{s, 1}^{k_0,\upsilon} + \varepsilon^{\frac12} \eta^{- S/2}  \| \zeta \|_{s, 1}^{k_0,\upsilon}  \| \widehat \zeta \|_{s_0, 1}^{k_0,\upsilon}  \\
& \lesssim_s   \varepsilon^{- \frac12} \eta^{S/2}   \| \widehat \zeta\|_{s, 1}^{k_0,\upsilon} + \varepsilon^{\frac12} \eta^{- S/2}  \| \zeta \|_{s, 1}^{k_0,\upsilon}  \| \widehat \zeta \|_{s_0, 1}^{k_0,\upsilon} \\
& \lesssim_s \varepsilon^{- \frac12}(\varepsilon^{\frac{1}{S}})^{\frac{S}{2}} \| \widehat \zeta\|_{s, 1}^{k_0,\upsilon}  + \varepsilon^{\frac12} (\varepsilon^{\frac1S})^{- S/2} \| \zeta \|_{s, 1}^{k_0,\upsilon}  \| \widehat \zeta \|_{s_0, 1}^{k_0,\upsilon} \\
& \lesssim_s  \| \widehat \zeta\|_{s, 1}^{k_0,\upsilon} + \| \zeta \|_{s, 1}^{k_0,\upsilon}  \| \widehat \zeta \|_{s_0, 1}^{k_0,\upsilon}
\end{aligned}
$$
which is the second estimate in \eqref{stime cal Q eta}. Therefore, the first estimate in \eqref{stime nonlinearita riscalata} follows by \eqref{definition p eta},  the mean value theorem and the second estimate in \eqref{stime cal Q eta}. We finally prove the third estimate in \eqref{stime nonlinearita riscalata}. The second derivative of $q_\varepsilon$ is given by 
$$
\di^2 q_\varepsilon(\cdot, \zeta)[\widehat \zeta_1, \widehat \zeta_2] = \varepsilon^{\frac12} \di^2 g_{\eta}(y,h_\eta(y) + \varepsilon \zeta)[\widehat \zeta_1, \widehat \zeta_2]\,.
$$
Hence, by the third estimate in \eqref{stime composizione g p eta}, Lemma \ref{unforcing.lemma} and \eqref{link eta varepsilon}, for $\varepsilon\ll 1$ one gets
$$
\begin{aligned}
\| \di^2 q_\varepsilon(\cdot, \zeta)[\widehat \zeta_1, \widehat \zeta_2] \|_{s, 1}^{k_0, \upsilon} & \lesssim_s \varepsilon^{\frac12} \eta^{- \frac{S}{2}} \Big( \| \widehat \zeta_1 \|_{s_0, 1}^{k_0,\upsilon} \| \widehat \zeta_2 \|_{s, 1}^{k_0,\upsilon} +  \| \widehat \zeta_1 \|_{s, 1}^{k_0,\upsilon} \| \widehat \zeta_2 \|_{s_0, 1}^{k_0,\upsilon}  \\
& \qquad + \| h_\eta + \varepsilon \zeta \|_{s, 1}^{k_0,\upsilon} \| \widehat \zeta_1 \|_{s_0, 1}^{k_0,\upsilon} \| \widehat \zeta_2 \|_{s_0, 1}^{k_0,\upsilon} \Big) \\
& \lesssim_s   \| \widehat \zeta_1 \|_{s_0, 1}^{k_0,\upsilon} \| \widehat \zeta_2 \|_{s, 1}^{k_0,\upsilon} +  \| \widehat \zeta_1 \|_{s, 1}^{k_0,\upsilon} \| \widehat \zeta_2 \|_{s_0, 1}^{k_0,\upsilon}  \\
& \qquad + \|  \zeta \|_{s, 1}^{k_0,\upsilon} \| \widehat \zeta_1 \|_{s_0, 1}^{k_0,\upsilon} \| \widehat \zeta_2 \|_{s_0, 1}^{k_0,\upsilon} 
\end{aligned}
$$
as claimed.
 The proof of the lemma is then concluded. 
\end{proof}
 We now write the rescaled second order equation \eqref{equazione secondo ordine con h eta riscalata} as a second order system. Let 
 \begin{equation}\label{definizione zeta 1 zeta 2 zeta}
 	\zeta_{1}(\bx,y) := \zeta \,, \quad \zeta_{2} := \omega\cdot \pa_{\bx} \zeta (\bx,y)\,, \quad u : = (\zeta_1, \zeta_2)
 \end{equation}
 Hence, solving the equation \eqref{equazione secondo ordine con h eta riscalata}  is equivalent to solving the first order system in the variable $u = (\zeta_1, \zeta_2)$ 
 \begin{equation}\label{first_order}
	\begin{cases}
		\omega \cdot \pa_{\bx} u (\bx,y) - J \nabla_u H_{\varepsilon}(u(\bx,y))\! =\!0 \,, \ \ (\bx,y)\in \T^{\kappa_0}\!\times \![-1,1]\,, \ \omega \in\R^{\kappa_0} \,, \\
		u(\bx, -1) = u (\bx, 1) = 0\,,
	\end{cases}
\end{equation}
where $J=\left(\begin{smallmatrix}
	0 & {\rm Id}\\ -{\rm Id} & 0
 \end{smallmatrix}\right)$ is the standard Poisson tensor and the Hamiltonian $H_\varepsilon$ is given by
\begin{equation}\label{Hami}
	\begin{aligned}
		H_{\varepsilon}(\zeta_1, \zeta_2)&:= \frac12\int_{-1}^1\Big( \zeta_{2}^2 - \zeta_{1}\cL_{\fm}\zeta_{1} \Big)\wrt z - \sqrt{\varepsilon}\int_{-1}^{1}Q_{\varepsilon}(y,\zeta_{1})\wrt y\,,
	\end{aligned}
\end{equation}
with $(\pa_\psi Q_\varepsilon)(y,\psi)= q_\varepsilon(y, \psi)$.
The symplectic 2-form induced by the Poisson tensor
is given by
\begin{equation}\label{sympl_form}
	\cW\Big( \begin{pmatrix}
		\zeta_1 \\ \zeta_2
	\end{pmatrix},\bigg( \begin{matrix}
	\wt\zeta_1 \\ \wt\zeta_2
\end{matrix} \bigg) \Big) := \Big( J^{-1} \begin{pmatrix}
\zeta_1 \\ \zeta_2
\end{pmatrix},\bigg( \begin{matrix}
\wt\zeta_1 \\ \wt\zeta_2
\end{matrix} \bigg) \Big)_{L^2} = -(\zeta_2,\wt\zeta_1)_{L^2} +(\zeta_1,\wt\zeta_2)_{L^2}\,,
\end{equation}
with $J^{-1}$ regarded as an operator acting on $L_0^2([-1,1])\times L_0^2([-1,1])$ into itself. The Hamiltonian field $X_{H_{\varepsilon}}(y,u):= J \nabla_u H_{\varepsilon}(y,u)$ is therefore characterized by the identity
\begin{equation}\label{Hami_vf}
	\di_u H_{\varepsilon}(u)[\whu] = \cW( X_{H_{\varepsilon}}(u), \whu )\quad \forall\, \whu\in L_0^2([-1,1])\times L_0^2([-1,1])\,.
\end{equation}
The "spatial phase space" $\cH :=  H_0^{2}([-1,1])\times L_0^2([-1.1])$ splits into two invariant subspaces for the Hamiltonian $$H_0 (\zeta_1, \zeta_2) :=  \frac12\int_{-1}^1\Big( \zeta_{2}^2 - \zeta_{1}\cL_{\fm}\zeta_{1} \Big)\wrt z$$ (namely \eqref{Hami} at $\varepsilon = 0$), that is, $\cH= \cX \oplus \cX_\perp$, with
\begin{equation}\label{cX.sub}
	\cX:={\rm span} \Big\{ \begin{pmatrix}
		\phi_{j,\fm}(y) \\ 0
	\end{pmatrix}, \begin{pmatrix}
	0 \\ \phi_{j,\fm}(y)
\end{pmatrix} \, : \, j=1,...,\kappa_0 \Big\}\,,
\end{equation}
and 
\begin{equation}\label{cX.perp}
 \cX_\perp:=\Big\{ \sum_{j\geq \kappa_0+1}\begin{pmatrix}
\alpha_j \\ \beta_j
\end{pmatrix}\phi_{j,\fm}(y) \in \cH  \, : \, \alpha_j, \beta_j\in\R \Big\}\,,
\end{equation}
where $(\phi_{j,\fm})_{j\in\N}$ is the basis of eigenfunctions for the self-adjoint operator $\cL_{\fm}$, see Proposition \ref{L_operator}. In the following, we will denote by $\Pi^\perp$ the projection on the invariant subspace $\cX_{\perp}$ in \eqref{cX.perp}.
We note that the symmetry condition \eqref{parity} translates in the unknown $\zeta=(\zeta_1,\zeta_2)$ as follows:
\begin{equation}\label{parity_sys}
	\zeta_1(\bx, y)\in \even(\bx)\even(y)\,, \quad \zeta_2(\bx,y)\in \odd(\bx)\even(x)\,.
\end{equation}

\subsection{Linear solutions near the shear equilibrium}
We want to study all the solutions of the linearized system around the stream function $\psi_{\fm}(y)$ at $\varepsilon = 0$. This amounts to solving the elliptic equation in \eqref{euler.pert.lin.1} (without any quasi-periodic conditions in $x$), which is equivalent to the following first-order systems
\begin{equation}\label{1st_lin_0}
\begin{cases}
	\pa_x \zeta -\bL_{\fm}(\tE) \zeta = 0 \,, \\
	\zeta( x, -1) = \zeta(x, 1) = 0\,, 
\end{cases} \quad \bL_{\fm}(\tE) :=\begin{pmatrix}
0 & {\rm Id} \\ \cL_{\fm}(\tE) & 0
\end{pmatrix}\,. 
\end{equation}
For any $\mathtt E \in [\mathtt E_1, \mathtt E_2]$. The spectrum  of the operator $\bL_{\fm} = \bL_{\mathtt m}(\mathtt E)$ with Dirichlet boundary conditions is given by
\begin{equation}
	\sigma(\bL_{\fm})= \big\{ \pm \sqrt{\mu_{j,\fm}} \,:\, j\in\N \big\} = \big\{ \pm\im \lambda_{j,\fm} \,:\, j=1,...,\kappa_0 \big\} \cup \big\{ \pm \lambda_{j,\fm}\,:\, j\geq \kappa_0+1 \big\}\,,
\end{equation}
where $\{ \mu_{j, \fm}=\mu_{j, \fm}(\tE) \}_{j\in\N}$ are the eigenvalues of $\cL_{\fm}$ as in Proposition \ref{L_operator}.
Solutions of \eqref{1st_lin_0} which satisfy \eqref{parity_sys} are given by
\begin{equation}
	\begin{aligned}
		\begin{pmatrix}
			\zeta_1(x,y) \\ \zeta_2(x,y)
		\end{pmatrix}&= \sum_{j=1}^{\kappa_0}A_{j}\begin{pmatrix}
			\cos(\lambda_{j,\fm}(\tE)x) \\ -\lambda_{j,\fm}\sin(\lambda_{j,\fm}(\tE)x)
		\end{pmatrix}\phi_{j,\fm}(y)\\
	&+ \sum_{j\geq \kappa_0+1}B_{j}\begin{pmatrix}
			\cosh(\lambda_{j,\fm}(\tE)x) \\ \lambda_{j,\fm}\sinh(\lambda_{j,\fm}(\tE)x)
		\end{pmatrix}\phi_{j,\fm}(y)\,,
	\end{aligned}
\end{equation}
for constants $A_j,B_j\in\R$. We deduce that, when $B_j=0$ for any $j\geq \kappa_0+1$, there exist solutions of the linearized system at $\varepsilon = 0$, at the equilibrium that are periodic or quasi-periodic in $x$ with at most $\kappa_0$ frequencies, depending on the non-resonance conditions between the linear frequencies $\ora{\omega}_{\fm}(\mathtt E)$ in \eqref{omega.fm.intro}.
The ultimate goal is to prove that, for amplitudes $0<A_1,...,A_{\kappa_0} \ll 1$ sufficiently small, close to the equilibrium $\psi_{\fm}(y)$ there exist stationary solutions to the nonlinear system \eqref{first_order} bifurcating from the quasi-periodic linear solutions above and still quasi-periodic in the space variable $x$, with frequency vectors $\omega$ close to the unperturbed linear frequency vector $\ora{\omega}_{\fm}(\mathtt E)$ in \eqref{omega.fm.intro}. We argue as follows: we fix $\tE\in \bar\cK \cap (\tE_1,\tE_2)$, with $\bar\cK\subset [\tE_1,\tE_2]$ as in \eqref{diofantea omega vec E imperturbato.intro}, and we consider an auxiliary parameter $\tA\in \cJ_{\varepsilon}(\tE)$ as in \eqref{intro parametro mathtt A dim}.
This parameter $\mathtt A$ will be used to impose the non-resonance conditions for the perturbed frequency vector.

\subsection{Action-angle coordinates on the invariant subspace $\cX$}

Functions in the ``spatial phase space'' $\cH =\cX\oplus \cX_\perp$ are parametrized by
\begin{equation}
	\zeta(y)=\begin{pmatrix}
		\zeta_1(y)\\ \zeta_2(y)
	\end{pmatrix}= \sum_{j=1}^{\kappa_0}\begin{pmatrix}
	a_j \\ b_j
\end{pmatrix}\phi_{j,\fm}(y) +z(y),
\end{equation}
where $z\in\cX_\perp$  and $(a_1,...,a_{\kappa_0},b_1,...,b_{\kappa_0})\in\R^{2\kappa_0}$ are coordinates on the $2\kappa_0$-dimensional invariant subspace $\cX$. We introduce another set of coordinates on $\cX$, the so called \emph{action-angle variables}: for some normalizing constant $Z>0$, let
\begin{equation}
	a_j: \sqrt{\frac{1}{Z}(I_j+\xi_j)} \cos(\theta_j)\,, \quad b_j :=-\sqrt{\frac{1}{Z}(I_j+\xi_j)} \sin(\theta_j)\,,\quad  \xi_j>0 \,, \ |I_j| \ll \xi_j\,,
\end{equation}
where $I=(I_1,...,I_{\kappa_0})\in\R^{\kappa_0}$ and $\theta=(\theta_{1},...,\theta_{\kappa_0})\in\T^{\kappa_0}$. Therefore, the function $A:\T^{\kappa_0}\times\R^{\kappa_0}\times\cX_\perp\to\cH$, defined by
\begin{equation}\label{aa_coord}
	\begin{aligned}
		A(\theta,I,z) & := v^\intercal(\theta,I)+z  := \sum_{j=1}^{\kappa_0}\sqrt{\frac{1}{Z}}\begin{pmatrix}
			\sqrt{I_j+\xi_j}\cos(\theta_j) \\ -\sqrt{I_j+\xi_j}\sin(\theta_j)
		\end{pmatrix}\phi_{j,\fm} + z\,,
	\end{aligned}
\end{equation}
is a parametrization of the spatial phase space $\cH$. The symplectic 2-form \eqref{sympl_form} reads in action-angle coordinates as
\begin{equation}\label{sympl_form_aa}
	\cW = \sum_{j=1}^{\kappa_0} (\di \theta_j \wedge \di I_j) \oplus \cW|_{\cX_\perp}\,.
\end{equation}
We also note that the 2-form $\cW$ is exact, namely
\begin{equation}\label{liouville}
	\cW=\di \Lambda\,, \quad \text{where } \ \ \ \Lambda_{(\theta,I,z)}[\wh\theta,\whI,\whz]:= -\sum_{j=1}^{\kappa_0}I_j \wh\theta_j +\tfrac12 (J^{-1} z,\whz)
\end{equation}
is the associated Liouville 1-form. Moreover, given a Hamiltonian $K:\T^{\kappa_0}\times \R^{\kappa_0}\times\cX_\perp$, the associated Hamiltonian vector field, with respect to the symplectic 2-form \eqref{sympl_form_aa}, is defined by
\begin{equation}\label{ham.vf.aa}
	X_K:= (\pa_I K, -\pa_\theta K, J\nabla_z K)\,,
\end{equation}
where $\nabla_z K$ denotes the $L^2$-gradient of $K$ with respect to $z\in\cX_\perp$.
Then, the equations in \eqref{first_order} (recall also the definition of $H_\varepsilon$ in \eqref{Hami}) becomes the Hamiltonian system in the action-angle coordinates $A(\theta,I,z)$  generated by the Hamiltonian
\begin{equation}\label{H_epsilon 0}
	\begin{aligned}
		\cH_{\varepsilon }(\theta,I,z) &:=  H_{\varepsilon}(A(\theta,I,z)) = \cN_{\fm}(I,z)+ \sqrt{\varepsilon} (P_\varepsilon \circ A)(\theta,I,z)\,, \\
		\cN_{\fm}(I,z)&:= \ora{\omega}_{\mathtt m}(\mathtt E) \cdot I +\tfrac12 (z,\big( \begin{smallmatrix}
			-\cL_{\fm}(\mathtt E) & 0 \\ 0 & {\rm Id}
		\end{smallmatrix} \big)  z)_{L^2}\,, \\
		P_{\varepsilon}(\zeta_1)&:=- \int_{-1}^{1} Q_{\varepsilon}(y,\zeta_1)\wrt y\,, \quad \text{where} \quad \partial_\psi Q_\varepsilon(y, \psi) = q_\varepsilon(y, \psi)\,. 
	\end{aligned}
\end{equation}
Now, for a fixed $\tE\in\bar\cK$ as in \eqref{diofantea omega vec E imperturbato nella proof},   we write, for any auxiliary parameter $\tA \in \cJ_{\varepsilon}(\tE)$, with $\cJ_{\varepsilon}(\tE)$ as in \eqref{intro parametro mathtt A dim}, 
$$
\vec \omega_{\mathtt m}(\mathtt E) \cdot I = \vec \omega_{\mathtt m}(\mathtt A) \cdot I + \big( \vec \omega_{\mathtt m}(\mathtt E) - \vec \omega_{\mathtt m}(\mathtt A) \big) \cdot I\,.
$$
Therefore, we rewrite \eqref{H_epsilon 0} as
\begin{equation}\label{H_epsilon}
	\begin{aligned}
		\cH_{\varepsilon }(\tE;\theta,I,z ) & = \cN_{\fm}(\tA,\tE;I,z )+ \sqrt{\varepsilon} {\cal P}_\varepsilon(\tA,\tE;\theta,I,z)\,, \\
		\cN_{\fm}(\tA,\tE;I,z)&:= \ora{\omega}_{\mathtt m}(\mathtt A) \cdot I +\tfrac12 (z,\big( \begin{smallmatrix}
			-\cL_{\fm}(\mathtt E) & 0 \\ 0 & {\rm Id}
		\end{smallmatrix} \big)  z)_{L^2}\,, \\
		{\cal P}_\varepsilon(\tA,\tE;\theta,I,z) & := P_\varepsilon (\tE;A(\theta,I,z)) + \tfrac{1}{\sqrt\varepsilon}\big( \vec \omega_{\mathtt m}(\mathtt E) - \vec \omega_{\mathtt m}(\mathtt A) \big) \cdot I \,.
	\end{aligned}
\end{equation}
We remark that the Hamiltonian $\cH_{\varepsilon}$ does not globally depend on the auxiliary parameter $\tA\in\cJ_{\varepsilon}(\tE)$.
Note that, since the frequency map is analytic and $|\mathtt A - \mathtt E| \leq \sqrt\varepsilon$, we also have
\begin{equation}\label{lip omega mathtt A E}
	\big|\partial_{\mathtt A}^k \big(\vec\omega_{\mathtt m}(\mathtt A) - \vec\omega_{\mathtt m}(\mathtt E) \big) \big| \lesssim_k \sqrt{\varepsilon} \quad \forall\, k\in\N_0 \ \Rightarrow \ | \ora{\omega}_{\fm}(\tA) - \ora{\omega}_{\fm}(\tE) |^{k_0,\upsilon} \lesssim \sqrt{\varepsilon}  \,. 
\end{equation}
This is actually crucial for considering the term $\big(\vec \omega_{\mathtt m}(\mathtt E) - \vec \omega_{\mathtt m}(\mathtt A) \big) \cdot I$ as perturbative of size $O(\sqrt\varepsilon)$ and it is the reason for which we choose the neighbourhood of $\mathtt E$ of size $\sqrt{\varepsilon}$.

The Hamiltonian equations associated to ${\cal H}_\varepsilon$ become 
\begin{equation}\label{system.aa.forced}
	\begin{cases}
		\pa_x \theta -\ora{\omega}_{\mathtt m}(\mathtt A) - \sqrt{\varepsilon}\pa_I {\cal P}_\varepsilon(\tA,\tE;\theta,I,z)=  0 \,, \\
		\pa_x I + \sqrt{\varepsilon} \pa_\theta {\cal P}_\varepsilon (\tA,\tE;\theta,I,z) = 0 \,, \\
				\pa_x z- \bL_{\fm}(\tE) z - \sqrt{\varepsilon} J \nabla_z {\cal P}_\varepsilon (\tA,\tE;\theta,I,z)  = 0 \,.
	\end{cases}
\end{equation}

\subsection{Nash-Moser Theorem with modified hypothetical conjugation}

We look for an embedded invariant torus for the Hamiltonian $\cH_{\varepsilon}$ of the form
\begin{equation}
	i :\T^{\kappa_0}\to \T^{\kappa_0}\times\R^{\kappa_0}\times \cX_\perp\,, \quad \bx\mapsto i (\bx):=(\theta(\bx),I(\bx),z(\bx))
\end{equation}
filled with quasi-periodic solutions with Diophantine frequency $\omega \in\R^{\kappa_0}$. The periodic component of the embedded torus is given by
\begin{equation}\label{ICal}
	\fI(\bx):= i(\bx)- (\bx,0,0):= (\Theta(\bx),I(\bx),z(\bx))\,, \quad \Theta(\bx):= \theta(\bx)-\bx\,,
\end{equation}
The expected quasi-periodic solution of the Hamiltonian equations \eqref{system.aa.forced} will have a slightly shifted frequency vector close to the unperturbed frequency vector $\ora{\omega}_{\fm}(\tE)$ in \eqref{omega.fm.intro}. The strategy that we implement here is a modification of the Th\'eor\`eme de conjugaison hypoth\'etique of Herman presented in \cite{fejoz}, see also \cite{BB}.
Recalling that we fixed $\tE\in\bar\cK$ in \eqref{diofantea omega vec E imperturbato.intro},   and therefore it will not be moved as a parameter in the following, we consider the modified Hamiltonian,  with $\alpha\in\R^{\kappa_0}$, 
\begin{equation}\label{Halpha}
	{\cal H}_{\varepsilon, \alpha} = \cH_{\varepsilon,\alpha}(\tA) :=\cN_{\alpha} + \sqrt{\varepsilon} \,{\cal P}_{\varepsilon}(\tA)\,, \quad \cN_{\alpha}:= \alpha\cdot I +\tfrac12 (z,\big( \begin{smallmatrix}
		-\cL_{\fm}(\tE) & 0 \\ 0 & {\rm Id}
	\end{smallmatrix} \big)  z)_{L^2}\,.
\end{equation}
We look for zeroes of the nonlinear operator
\begin{equation}\label{F_op}
	\begin{aligned}
		{\cal F}({\frak I}, \alpha) \equiv \cF (i,\alpha)  & := \cF (\omega, \mathtt A, \varepsilon;i ,\alpha) := \omega\cdot \pa_\bx i(\bx) - X_{{\cal H}_{\varepsilon, \alpha}}(i(\bx))  \\
		& := \begin{pmatrix}
			\omega\cdot\pa_\bx \theta(\bx) & - \alpha- \sqrt{\varepsilon} \pa_I {\cal P}_\varepsilon(\tA;i(\bx)) \\ 
			\omega\cdot\pa_\bx I(\bx) & +\sqrt{\varepsilon} \pa_\theta {\cal P}_\varepsilon(\tA;i(\bx)) \\
			\omega\cdot\pa_\bx z(\bx)& -\bL_{\fm}(\tE) z(\bx) - \sqrt{\varepsilon} J \nabla_z {\cal P}_\varepsilon(\tA;i(\bx)) 
		\end{pmatrix}\,.
	\end{aligned}
\end{equation}
The parameters of the problem are $\lambda = (\omega, \mathtt A) \in \R^{\kappa_0} \times {\cal J}_\varepsilon(\mathtt E) \subset \R^{\kappa_0} \times \R$, whereas the unknowns of the problem are $\alpha$ and the periodic component of the torus embedding ${\frak I}$.  Solutions of the Hamiltonian equations \eqref{system.aa.forced} are recovered by setting $\alpha=\ora{\omega}_{\fm}(\tA)$.
The Hamiltonian ${\cal H}_{\varepsilon, \alpha}$ is invariant under the involution $\vec\cS$, namely
\begin{equation}\label{invar.rev.Halpha}
	{\cal H}_{\varepsilon, \alpha} \circ \vec\cS = {\cal H}_{\varepsilon, \alpha}\,,
\end{equation}
where $\vec \cS$ is the involution defined in \eqref{invo.aa}, \eqref{invo.std}.
We look for a reversible torus embedding
$ \bx \mapsto  i (\bx) =  ( \theta(\bx), I(\bx), w(\bx)) $, namely satisfying  
\begin{equation}\label{RTTT}
	\vec \cS i(\bx)=  i(-\bx) \,.
\end{equation}
	Recalling \eqref{cX.perp}, 
	let 
	$
	H_\bot := {\rm span}\big\{ \phi_{j, \mathtt m} : j \geq \kappa_0 + 1 \big\}
	$
	and, for any $s, \rho \geq 0$, we define 
	\begin{equation}\label{def H s rho bot}
	\begin{aligned}
	& H^{s, \rho}_\bot := H^s(\T^{\kappa_0}, H_0^\rho([- 1, 1] )\cap H_\bot) \\
	& \equiv\Big\{ u({\bf x}, y) = \sum_{\ell \in \Z^{\kappa_0}} \sum_{j \geq \kappa_0 + 1} u_{\ell, j} e^{\im \ell \cdot {\bf x}} \phi_{j, \mathtt m}(y) \quad \text{with} \quad \| u \|_{s, \rho}^{k_0,\upsilon} < + \infty  \Big\}\,.
	\end{aligned}
	\end{equation}
	Then, we set 
	\begin{equation}\label{prima def cal X s bot}
		{\cal X}^s_\bot  :=  H^{s, 3}_\bot \times H^{s, 1}_\bot  \,, \quad {\cal Y}^s_\bot  := H^{s, 1}_\bot \times H^{s, 1}_\bot\,,
	\end{equation}
	with corresponding norms, for $z=(z_1,z_2)$,
	\begin{equation}
			\| z \|_{{\cal X}^s_\bot}^{k_0, \upsilon} := \| z_1 \|_{s, 3}^{k_0, \upsilon} + \| z_2 \|_{s, 1}^{k_0, \upsilon}\,, \quad \| z \|_{{\cal Y}^s_\bot}^{k_0, \upsilon} := \| z_1 \|_{s, 1}^{k_0, \upsilon} + \| z_2 \|_{s, 1}^{k_0, \upsilon}\,,
	\end{equation}
where the norms $\|\,\cdot\,\|_{s,\rho}^{k_0,\upsilon}$ are defined in \eqref{sobolev.sp}-\eqref{weighted_norm}.
	We also define  the spaces
	\begin{equation}\label{prima def cal X s}
		{\cal X}^s  := H^s(\T^{\kappa_{0}}) \times H^s(\T^{\kappa_{0}}) \times {\cal X}^s_\bot \,, \quad {\cal Y}^s  := H^s(\T^{\kappa_{0}}) \times H^s(\T^{\kappa_{0}}) \times {\cal Y}^s_\bot\,,
	\end{equation}
	with corresponding norms, for ${\frak I} = (\Theta, I, z)$,
	\begin{equation}
			\| {\frak I} \|_{{\cal X}^s}^{k_0, \upsilon}  :=  \| \Theta \|_s^{k_0, \upsilon} + \| I \|_s^{k_0, \upsilon} + \| z \|_{{\cal X}^s_\bot}^{k_0, \upsilon}\,, \quad 	\| {\frak I} \|_{{\cal Y}^s}^{k_0, \upsilon}  :=  \| \Theta \|_s^{k_0, \upsilon} + \| I \|_s^{k_0, \upsilon} + \| z \|_{{\cal Y}^s_\bot}^{k_0, \upsilon}\,.
	\end{equation}
	 Note that 
	\begin{equation}\label{norma Ys Xs}
		\| \cdot \|_{{\cal Y}^s}^{k_0,\upsilon} \leq \| \cdot \|_{{\cal X}^s}^{k_0,\upsilon}\,, \quad \forall \,s \geq 0\,,
	\end{equation}
	and that, for $s \geq 0$, the nonlinear map ${\cal F}$ maps ${\cal X}^{s + 1} \times \R^{\kappa_{0}} $ into $ {\cal Y}^s $. 
	We fix
	\begin{equation}\label{k0.def}
		k_0:=m_0+2\,,
	\end{equation}
	where $k_0$ is the index of non-degeneracy provided in Proposition \ref{prop:trans_un}, which only depends on the linear unperturbed frequencies. Thus $k_0$ is considered as an absolute constant and we will often omit to explicitly write the dependence of the various constants with respect to $k_0$.
	Each frequency vector $\omega=(\omega_1,...,\omega_{\kappa_0})$ will belong to a $\varrho$-neighbourhood (independent of $\varepsilon$)
	\begin{equation}\label{neigh.Omega}
		\t\Omega:=\Big\{ \omega\in\R^{\kappa_0} \,:\,  \dist \Big(\omega,\ora{\omega}_{\fm}({\cal J}_\varepsilon(\mathtt E)) \Big)<\varrho \Big\}\,, \quad \varrho>0\,,
	\end{equation}
	where $\ora{\omega}_{\fm}({\cal J}_\varepsilon(\mathtt E)) = \big\{ \ora{\omega}_\fm(\mathtt A) : \mathtt A \in {\cal J}_\varepsilon(\mathtt E) \big\}$ is the range of the unperturbed linear frequency map $\mathtt A \mapsto \ora{\omega}_{\fm}(\mathtt A)$ defined in \eqref{omega.fm.intro}, restricted to the interval ${\cal J}_\varepsilon(\mathtt E) $ in \eqref{intro parametro mathtt A dim}. 
	\begin{thm}{\bf (Nash-Moser)}\label{NMT}
		Let $\kappa_0\in\N$ be fixed as for Proposition \ref{L_operator}. Let $\tau\geq 1$. There exist positive constants ${\rm a}_0,\varepsilon_0,C$ depending on $\kappa_0,k_0,\tau$ such that, for all $\upsilon=\varepsilon^{\rm a}$, ${\rm a}\in(0,{\rm a}_0)$, and for all $\varepsilon\in(0,\varepsilon_0)$, there exist a $k_0$-times differentiable function
		\begin{equation}\label{alpha_infty}
		\begin{aligned}
			& \alpha_{\infty}: \R^{\kappa_0} \times {\cal J}_\varepsilon(\mathtt E) \to \R^{\kappa_0}\,, \quad (\omega, \mathtt A) \mapsto \alpha_{\infty}(\omega, \mathtt A) : = \omega +r_{\varepsilon}(\omega, \mathtt A) \,, \\
			&  |r_{\varepsilon}|^{k_0,\upsilon}\leq C\sqrt{\varepsilon} \upsilon^{-1}\,,
			\end{aligned}
		\end{equation}
		and a family of reversible embedded tori $i_{\infty}$ defined for all $(\omega, \mathtt A)\in \R^{\kappa_0} \times {\cal J}_\varepsilon(\mathtt E)$ satisfying \eqref{RTTT} and 
		\begin{equation}\label{embed.infty.est}
			\| i_{\infty}(\bx)-(\bx,0,0) \|_{{\cal X}^{s_0}}^{k_0,\upsilon} \leq C\sqrt{\varepsilon} \upsilon^{-1}\,, 
		\end{equation}
		such that, for all $(\omega, \mathtt A) \in \tG^\upsilon \times {\cal J}_\varepsilon(\mathtt E)$ where the Cantor set $\tG^\upsilon$ is defined as 
		\begin{equation}\label{Cset_infty}
			\tG^\upsilon:=\Big\{ \omega\in\t\Omega \, : \, |\omega\cdot \ell | \geq \upsilon \braket{\ell}^{-\tau}, \ \forall\,\ell\in\Z^{\kappa_0}\setminus\{0\} \Big\}\subset\R^{\kappa_0}\,,
		\end{equation}
		the function $i_{\infty}(\bx):= i_{\infty}(\omega,\varepsilon;\bx)$ is a solution of
		\begin{equation}
			\cF(\omega, \mathtt A, \varepsilon; i_{\infty},\alpha_{\infty}(\omega, \mathtt A)) =0\,.
		\end{equation}
		As a consequence, each embedded torus $\bx\mapsto i_{\infty}(\bx)$ is invariant for the Hamiltonian vector field $X_{{\cal H}_{\varepsilon, \alpha_{\infty}(\omega, \mathtt A)}}$ and it is filled by quasi-periodic solutions with frequency $\omega$.
	\end{thm}
	The following theorem  will be proved in  Section \ref{sez.proof.NMT}. The Diophantine condition in \eqref{Cset_infty} is verified for most of the  parameters, see Theorem \ref{MEASEST}.
	
	\section{Transversality of the linear frequencies}\label{deg.KAM}
	
	In this section we apply the  KAM theory approach of Arnold and Rüssmann (see \cite{Russ}), extended to PDEs in \cite{BM}, \cite{BaBM},
	in order to deal with the linear frequencies $ \lambda_{j,\fm}(\tE)$ defined in Proposition \ref{L_operator}.
	We first give the following definition.
	
	\begin{defn}\label{def:non-deg}
		A function $f=(f_1,\dots,f_{\kappa_0}):[\tE_1,\tE_2]\rightarrow\R^{\kappa_0}$ is 
		\emph{non-degenerate} if, for any $c\in \R^{\kappa_0}\setminus\{0\}$, the scalar function $f\cdot c$ is not identically zero on the whole parameter interval $[\tE_1,\tE_2]$.
	\end{defn}
	We recall the vector of the linear frequencies in \eqref{omega.fm.intro}, as $\fm\to \infty$,
	\begin{equation}\label{omega0.vec}
		\ora{\omega}_{\fm}(\tE) :=  \big( \lambda_{1,\fm}(\tE),..., \lambda_{\kappa_0,\fm}(\tE)\big) \to  \ora{\omega}_{\infty}(\tE) := \big( \lambda_{1,\infty}(\tE),..., \lambda_{\kappa_0,\infty}(\tE) \big) \in \R^{\kappa_0}\,, 
	\end{equation}
	where $\lambda_{j,\infty}(\tE) \in (0,\tE)$ is a zero of the secular equation, recalling \eqref{secular_0},
			\begin{equation}\label{secular_0_transv}
			\fF(\lambda):= \lambda \cos\big(\tr \sqrt{\tE^2-\lambda^2}\big)\coth((1-\tr)\lambda) - \sqrt{\tE^2-\lambda^2}\sin\big(\tr \sqrt{\tE^2-\lambda^2}\big) = 0 \,,
	\end{equation}
	 for any $j=1,...,\kappa_0$.
	\begin{rem}\label{remark analitycity}
		All the frequencies involved are analytic with respect to the parameter $\tE\in [\pi(\kappa_{0}+\tfrac14),\infty)$. Indeed, the frequencies $(\lambda_{j, \infty}(\tE))_{j=1}^{\kappa_0}$ are analytic because defined as implicit zeroes of the analytic function \eqref{secular_0_transv},
		 whereas the frequencies $(\lambda_{j,\fm}(\tE))_{j=1}^{\kappa_0}$ are analytic because $-\lambda_{j,\fm}^2(\tE)$ are the negative eigenvalues of the Schrödinger operator $\cL_{\fm}=-\pa_y^2 +Q_{\fm}(y)$ with the analytic potential $Q_{\fm}(y):=Q_{\fm}(\tE,\tr;y)$ under the analytic constrain in \eqref{constrain} (for the physical problem of the channel, we have $\tr\in (0,1]$, from where follows the threshold $\tE\geq (\kappa_{0}+\tfrac14)\pi$, otherwise arbitrary). 
	\end{rem}
	We prove the non-degeneracy of the vector $\ora{\omega}_{\fm}(\tE)$ by first showing that the limit vector $\ora{\omega}_{\infty}(\tE)$ is non-degenerate and then arguing by perturbation. 
	To do so, the key tool is an explicit asymptotic expansion for the frequencies $(\lambda_{j,\infty}(\tE))_{j=1}^{\kappa_0}$.
	\begin{lem}\label{lemma.asympt}
		For any $j=1,...,\kappa_{0}$, let $\lambda_{j,\infty}=\lambda_{j,\infty}(\tE)\in (0,\tE)$ be a zero of \eqref{secular_0_transv}. Then, we have the asymptotic expansion
		\begin{equation}\label{lambdaE.asympt}
			\begin{aligned}
				& \lambda_{j,\infty}(\tE) = \tE \cos\Big(  \pi \big(  \alpha_0(j) + \alpha_2(j) \beta_{j}(\tE)^2 + o(\beta_{j}(\tE)^3)  \big)  \Big)\,, \ \tE \to +\infty\,, \\
				& \beta_{j}(\tE) := \exp(((\kappa_0+\tfrac14)\pi-\tE)\cos(\pi\alpha_0(j)))\,,
			\end{aligned}
		\end{equation} 
		with $\alpha_2(j)=\alpha_2(\alpha_0(j))\in\R\setminus\{0\}$, where  $\alpha_0(j)\in (0,\frac12)$ is a zero of the equation
		\begin{equation}\label{alpha0.implicit}
			\sin(\pi\alpha_0(j)) = \frac{j-\frac12 - \alpha_0(j)}{\kappa_0+\frac14}\,.
		\end{equation}
	\end{lem}
	\begin{proof}
		For sake of simplicity in the notation, in the following proof we will omit to write the explicit dependence of the fixed $j=1,...,\kappa_{0}$ when not needed.
		Therefore, let $\lambda_{j, \infty}(\tE)\equiv \lambda(\tE) \in (0,\tE)$, which implies  $\vs(\lambda(\tE)):=\sqrt{(\tE - \lambda(\tE)^2)}\in (0,\tE)$. We make the following change of variables
		\begin{equation}
			\lambda(\tE):= \tE \cos(\pi \alpha(\tE))\,, \quad \vs(\lambda(\tE)) = \tE \sin(\pi\alpha(\tE)) \,, \quad \alpha(\tE) \in (0,\tfrac12)\,.
		\end{equation}
		Recalling the constrain $\tE\tr = (\kappa_0+\tfrac14)\pi$ in \eqref{constrain}, we have that any zero $\lambda(\tE)$ of \eqref{secular_0_transv} in $(0,\tE)$ correspond to a solution $\alpha(\tE)$ of the following equation
		\begin{align}
			&\cos\big(\pi\alpha(\tE) +(\kappa_0+\tfrac14)\pi \sin (\pi\alpha(\tE))\big)= \label{secular_1} \\
			&=  \cos(\pi\alpha(\tE)) \cos\big((\kappa_0+\tfrac14)\pi\sin(\pi\alpha(\tE))\big) \big(1- \coth\big((\tE-(\kappa_0+\tfrac14)\pi)\cos(\pi\alpha(\tE))\big)  \big) \,, \nonumber
		\end{align}
		where $\beta(\tE)$ is defined in \eqref{lambdaE.asympt}. We search for solutions of the form
		\begin{equation}
			\alpha(\tE) = \alpha_0 + \alpha_1 \beta(\tE) +\alpha_2\beta(\tE)^2+ \tg(\tE)\,, \quad \tg(\tE) = o(\beta(\tE)^3) \,, \quad \tE\to+\infty\,.
		\end{equation}
		By the expansion
		\begin{equation}
			1 - \coth(z) = 1 - \frac{1+e^{-2z}}{1-e^{-2z}} = -2 \sum_{n=1}^{\infty} e^{-2nz}\,, \quad \forall \,z >0\,, 
		\end{equation}
		we note that, in the regime $\tE\to \infty$,
		\begin{align}
	&
	\begin{footnotesize}
		\begin{aligned}
			1-\coth\big((\tE-(\kappa_0+\tfrac14)\pi)\cos(\pi\alpha(\tE))\big)  = -2  \sum_{n=1}^{\infty} \exp\big( 2n((\kappa_0+\tfrac14)\pi-\tE)\cos(\pi\alpha(\tE)) \big) 
		\end{aligned}
	\end{footnotesize}
	\nonumber \\
	& \begin{footnotesize}
		\begin{aligned}
			= -2\sum_{n=1}^{\infty} \exp\Big(2n((\kappa_0+\tfrac14)\pi-\tE)\big(\cos(\pi\alpha_0)  + o(\beta(\tE)) \big)\Big) 
		\end{aligned}
	\end{footnotesize}\nonumber \\
	& \begin{footnotesize}
		\begin{aligned}
			= -2 \sum_{n=1}^{\infty} (\beta(\tE))^{2n}  \exp\big(2n((\kappa_0+\tfrac14)\pi-\tE) o(\beta(\tE))\big) 
		\end{aligned}
	\end{footnotesize}\nonumber  \\
& \begin{footnotesize}
	\begin{aligned}
		= -2 (\beta(\tE))^{2} - \sum_{n=2}^{\infty} (\beta(\tE))^{2n} + 2 \sum_{n=1}^{\infty} (\beta(\tE))^{2n}  \Big( 1-  \exp\big(2n((\kappa_0+\tfrac14)\pi-\tE) o(\beta(\tE))\big)  \Big)\,.
	\end{aligned}
\end{footnotesize} \nonumber
\end{align}
We obtain that
\begin{equation}\label{coth.exp}
	1-\coth\big((\tE-(\kappa_0+\tfrac14)\pi)\cos(\pi\alpha(\tE))\big) = - 2 (\beta(\tE))^{2}  + o\big( \tE (\beta(\tE))^3 \big)\,, \quad \tE\to\infty\,.
\end{equation}
		Since we are interested in the first two powers of $\beta(\tE)$ in the expansion of \eqref{secular_1}, it means that the other two factors on the right hand side of \eqref{secular_1} contribute only with their zeroth orders. We now compute the first two orders of the left hand side of \eqref{secular_1}. We have
		\begin{equation}
			\begin{aligned}
				\alpha(\tE)&+(\kappa_0+\tfrac14)\sin(\pi\alpha(\tE)) = \alpha_0+(\kappa_0+\tfrac14)\sin(\pi\alpha_0)\\ &+\alpha_1(1+(\kappa_0+\tfrac14)\pi\cos(\pi\alpha_0)) \beta(\tE)\\
				&+\big( \alpha_2 (1+(\kappa_0+\tfrac14)\pi\cos(\pi\alpha_0))  -\tfrac{\pi^2}{2}(\kappa_0+\tfrac14)\alpha_1^2 \sin(\pi\alpha_0) \big)(\beta(\tE))^2 + ...
			\end{aligned}
		\end{equation}
		where the dots stand for higher order terms. It follows that
		\begin{align}
			& \cos\big( \pi\alpha(\tE) +(\kappa_0+\tfrac14) \pi \sin(\pi\alpha(\tE)) \big) = \cos(\pi(\alpha_0+(\kappa_0+\tfrac14)\sin(\pi\alpha_0))) \label{expans.lhs} \\
			& -\pi\sin(\pi(\alpha_0+(\kappa_0+\tfrac14)\sin(\pi\alpha_0))) \Big( \alpha_1(1+(\kappa_0+\tfrac14)\pi\cos(\pi\alpha_0)) \beta(\tE) \nonumber \\
			& +\big( \alpha_2 (1+(\kappa_0+\tfrac14)\pi\cos(\pi\alpha_0))  -\tfrac{\pi^2(\kappa_0+\tfrac14)}{2}\alpha_1^2 \sin(\pi\alpha_0) \big)(\beta(\tE))^2 + ... \Big)\nonumber \\
			&-\tfrac{\pi^2}{2}\cos(\pi(\alpha_0+(\kappa_0+\tfrac14)\sin(\pi\alpha_0)))\Big( \alpha_1(1+(\kappa_0+\tfrac14)\pi \cos(\pi\alpha_0))\beta(\tE) +... \Big) +... \,. \nonumber
		\end{align}
		The zeroth order term $O(1)$ in \eqref{secular_1} comes only from its left hand side. Therefore we impose $\alpha_0$ to solve
		\begin{equation}
			\cos(\pi(\alpha_0+(\kappa_0+\tfrac14)\sin(\pi\alpha_0))) = 0\,.
		\end{equation}
		It follows that $\alpha_0=\alpha_0(j)\in (0,\frac12)$ solves, for $j=1,...,\kappa_0$, the implicit equation \eqref{alpha0.implicit}. Furthermore, we note that
		\begin{equation}\label{note.sine.a0}
			\sin\big(\pi(\alpha_0(j)+(\kappa_0+\tfrac14)\sin(\pi\alpha_0(j)))\big) = (-1)^{j-1} \,, \quad \forall\,j=1,...,\kappa_0\,.
		\end{equation}
		Also the contribution to the first order term $O(\beta(\tE))$ in \eqref{secular_1} comes only from its left hand side. Looking at \eqref{expans.lhs}, together with \eqref{note.sine.a0}, we impose
		\begin{equation}\label{a1=0}
			(-1)^{j-1}\pi\alpha_1 (1+(\kappa_0+\tfrac14)\pi\cos(\pi\alpha_0(j))) = 0 \quad \Rightarrow \quad \alpha_1 =0\,.
		\end{equation}
		Second order terms $O((\beta(\tE))^2)$ appear on both sides of \eqref{secular_1}. By \eqref{secular_1}, \eqref{coth.exp}, \eqref{expans.lhs}, \eqref{note.sine.a0} and \eqref{a1=0}, we impose, for any $j=1,...,\kappa_0$,
		\begin{footnotesize}
			\begin{equation}
				(-1)^{j-1} \alpha_2\pi (1+(\kappa_0+\tfrac14)\pi \cos(\pi\alpha_0(j))) =  -2\cos(\pi\alpha_0(j)) \cos((\kappa_0+\tfrac14)\pi\sin(\pi\alpha_0(j)))\,.
			\end{equation}
		\end{footnotesize}
		This linear equation uniquely defines $\alpha_2(j)= \alpha_2(\alpha_0(j))\neq 0$ for any $j=1,...,\kappa_0$. This concludes the proof.
	\end{proof}
	
	We can now prove that the following proposition.
	\begin{prop}\label{prop.nondeg}
		{\bf (Non-degeneracy for $\ora{\omega}_{\infty}(\tE)$). }
		The vector $\ora{\omega}_{\infty}(\tE)$ in \eqref{omega0.vec}
		is non-degenerate in the interval $[\tE_1,\tE_2]$, assuming $\tE_1\gg 1$ sufficiently large.
	\end{prop}
	
	\begin{proof}
		We argue by contradiction. Assume, therefore, that there exists a nontrivial vector $c=(c_1,...,c_{\kappa_0})\in\R^{\kappa_0}\setminus\{0\}$ such that $c\cdot \ora{\omega}_{\infty}(\tE) =0$ on $[\tE_1,\tE_2]$. By Remark \ref{remark analitycity}, the zeroes of the analytic
		function \eqref{secular_0_transv} are analytic in the domain $[\tE_1,\infty)$ as well. Therefore, we have $c\cdot \ora{\omega}_{\infty}(\tE) =0$ on $[\tE_1,\infty)$. By Lemma \ref{lemma.asympt} and dividing by $\tE$, we have, for any $\tE \in [\tE_1,\infty)$,
		\begin{equation}\label{linear.comb}
			c_1 \cos( \pi\vt_1(\tE) ) + ... + c_{\kappa_0} \cos(\pi\vt_{\kappa_0}(\tE)) = 0 \quad \forall\,\tE \in [\tE_1,\infty)\,, 
		\end{equation}
		where
		\begin{equation}\label{vtbeta}
			\begin{aligned}
				&\vt_j(\tE)  := \alpha_0(j) + \alpha_2(j) \beta_j(\tE)^2 + o(\beta_j(\tE)^3)\,, \quad \tE\to+\infty\,,\\
				&\beta_j(\tE) := \exp\big( ((\kappa_0+\tfrac14)\pi - \tE) \cos(\pi\alpha_0(j))  \big) \,.
			\end{aligned}
		\end{equation}
		In particular, we further expand the asymptotic in \eqref{lambdaE.asympt}, obtaining, for any $j=1,...,\kappa_0$, in the limit regime $\tE \to +\infty$,
		\begin{equation}\label{cosine.exp}
			\cos(\pi\vt_j(\tE)) = \cos(\pi\alpha_0(j)) - \pi\alpha_2(j) \sin(\pi\alpha_0(j)) \beta_j(\tE)^2 + o(\beta_j(\tE)^3)\,.
		\end{equation}
		By \eqref{vtbeta}, we have
		\begin{equation}
			\pa_\tE \beta_j(\tE)^2 = -2\cos(\pi\alpha_0(j)) \beta_j(\tE)^2 \,,  \quad \pa_\tE\big( o(\beta_j(\tE)^3)\big) = o(\beta_j(\tE)^3) \,.
		\end{equation}
		By differentiating with respect to $\tE$ in \eqref{linear.comb}, using \eqref{cosine.exp}, we get, in the asymptotic regime $\tE\to + \infty$,
		\begin{equation}\label{linear.derived}
			\begin{aligned}
				&\, c_1 \big( \alpha_2(1) \sin(2\pi\alpha_0(1)) \beta_1(\tE) ^2 + o(\beta_1(\tE)^3)   \big) \\
				+& \,... \,  \\
				+ &	\, c_{\kappa_0-1} \big( \alpha_2(\kappa_0-1) \sin(2\pi\alpha_0(\kappa_0-1)) \beta_{\kappa_0-1}(\tE) ^2 + o(\beta_{\kappa_0-1}(\tE)^3)   \big) \\ 
				+ &	\, c_{\kappa_0} \big( \alpha_2(\kappa_0) \sin(2\pi\alpha_0(\kappa_0)) \beta_{\kappa_0}(\tE) ^2 + o(\beta_{\kappa_0}(\tE)^3)   \big) = 0\,. 
			\end{aligned}
		\end{equation}
		The solution $\alpha_{0}(j)$ of \eqref{alpha0.implicit} are monotone increasing with respect to $j$ because we have $\alpha_{0}'(j)=\big(1+(\kappa_{0}+\frac14)\pi \cos(\pi\alpha_{0}(j))\big)^{-1}>0$, since $\cos(\pi\alpha_{0}(j))\in(0,1)$ (the derivative $\alpha_{0}'(j)$ is of course meant with $j$ as a continuous variable). Therefore, we have $0<\alpha_0(1) < \alpha_0(2) < ... < \alpha_0(\kappa_0-1) < \alpha_0(\kappa_0) < \tfrac12$.
		It follows that
		\begin{equation}
			0 < \exp(-\tE\cos(\pi\alpha_0(1))) < ... < \exp(-\tE\cos(\pi\alpha_0(\kappa_0))) < 1 \,.
		\end{equation}
		This implies that, for any $\tE>(\kappa_{0}+\frac14)\pi$ large enough, recalling \eqref{vtbeta},
		\begin{equation}
			0 < \beta_1(\tE) < \beta_2(\tE) .... < \beta_{\kappa_0-1}(\tE) < \beta_{\kappa_0}(\tE)<1
		\end{equation}
		It means that $\beta_{\kappa_0}(\tE)^2$ is the leading term in \eqref{linear.derived}. Therefore, we multiply all the terms in \eqref{linear.derived} by $\beta_{\kappa_0}(\tE)^2$ and, taking the limit $\tE\to +\infty$, we obtain
		\begin{equation}
			c_{\kappa_0} \alpha_2(\kappa_0)\sin(\pi\alpha_0(\kappa_0)) = 0\,.
		\end{equation}
		Since $\alpha_2(\kappa_0)\sin(\pi\alpha_0(\kappa_0))\neq 0 $ by Lemma \ref{lemma.asympt}, we obtain $c_{\kappa_0}=0$. We insert this constrain in \eqref{linear.derived} and we iterate the procedure with a new leading term at each step. We conclude that we must have $c_{\kappa_0} = c_{\kappa_0-1}=...=c_1=0$, which is a contradiction. The claim is proved.
	\end{proof}
	
	Roughly speaking, the non-degeneracy is an open condition. Therefore, the property extends from the limit vector $\ora{\omega}_{\infty}(\tE)$ to $\ora{\omega}_{\fm}(\tE)$ when $\fm$ is sufficiently large.
	
	\begin{thm}{\bf (Non-degeneracy for $\ora{\omega}_{\fm}(\tE)$).}\label{non deg omega gamma}
		There exists $\bar\fm \equiv \bar\fm(\kappa_0) \gg 1$ such that, for any $\fm\geq \bar\fm$, the vector $\ora{\omega}_{\fm}(\tE)$ in \eqref{omega0.vec}
		is non-degenerate in the interval $[\tE_1,\tE_2]$, assuming $\tE_1>1$ sufficiently large.
	\end{thm}
	\begin{proof}
		By contradiction, assume that for any $\bar\fm\gg 1$ there exists $\fm > \bar\fm$ and a vector $c_{\fm} \in \R^{\kappa_0}\setminus\{0\}$ such that $ \ora{\omega}_{\fm}(\tE) \cdot c_{\fm}= 0$ for any $\tE \in [\tE_1, \tE_2]$. Clearly by defining $d_\fm:= \frac{c_\fm}{|c_\fm|}$ one also has that 
		$$
		\ora{\omega}_{\fm}(\tE) \cdot d_\fm = 0 \quad \forall \,\tE \in [\tE_1, \tE_2]\,. 
		$$
		Since $|d_\fm| = 1$ for any $\fm >\bar\fm$, up to subsequences $d_\fm \to \bar d$, with $|\,\bar d\,| = 1$. Moreover $\sup_{\tE \in [\tE_1, \tE_2]} |\ora{\omega}_{\fm}(\tE) - \ora{\omega}_{\infty}(\tE)| \to 0$ as $\fm\to\infty$ by Proposition \ref{L_operator}. Hence, up to subsequences $\ora{\omega}_{\fm}(\tE) \cdot d_\fm$ converges to $\ora{\omega}_{\infty} (\tE) \cdot \bar d$ uniformly on $[\tE_1, \tE_2]$ as $\fm\to \infty$. This clearly implies that 
		$$
		\ora{\omega}_{\infty}(\tE) \cdot \bar d= 0, \quad \forall \,\tE \,\in [\tE_1, \tE_2]
		$$
		which contradicts the non degeneracy of the vector $\ora{\omega}_{\infty}$ proved  in Proposition \ref{prop.nondeg}. 
		%
		%
		%
		%
		%
	\end{proof}

	The next proposition is the key of the argument.  It provides a quantitative bound from the qualitative non-degeneracy condition in Theorem \ref{non deg omega gamma}.

	\begin{prop} {\bf (Transversality).} \label{prop:trans_un}
		Let $\bar\fm\gg 1$  as in Theorem \ref{non deg omega gamma}. Then, for any $\fm \geq \bar\fm$, there exist $m_0\in\N$ and $\rho_0>0$ such that, for any $\tE\in[\tE_1,\tE_2]$, 
		\begin{equation}
			\max_{0\leq n \leq m_0}
			| \partial_\tE^n \ora{\omega}_{\fm}(\tE)\cdot \ell | \geq \rho_0\braket{\ell} \,, \quad \forall\,\ell\in\Z^{\kappa_0}\setminus\{0\} \,; \label{eq:0_meln}
		\end{equation}
		We call $\rho_0$ the amount of non-degeneracy and $m_0$ the index of non-degeneracy.
	\end{prop}
	\begin{proof}
		Let $\fm\geq \bar\fm$. By contradiction, assume that
		for any $m\in\N$ there exist $\tE_m\in[\tE_1,\tE_2] $ and $\ell_m\in\Z^{\kappa_0}\setminus\{0\}$ such that
		\begin{equation}\label{eq:0_abs_m}
			\Big| \partial_\tE^n \ora{\omega}_{\fm}(\tE_m) \cdot \frac{\ell_m}{\braket{\ell_m}}  \Big| < \frac{1}{\braket{m}} \,, \quad \forall\,0\leq n\leq m \, .
		\end{equation}
		The sequences $(\tE_m)_{m\in\N}\subset[\tE_1,\tE_2]$ and $(\ell_m/\braket{\ell_m})_{m\in\N}\subset \R^{\kappa_0}\setminus\{0\}$ are both bounded. By compactness, up to subsequences
		$\tE_m\to \bar\tE\in[\tE_1,\tE_2]$ and $\ell_m/\braket{\ell_m}\rightarrow\bar c\neq 0$. 
		Therefore, in the limit for $m\rightarrow + \infty$, by  \eqref{eq:0_abs_m} we get $\partial_\tE^n \ora{\omega}_{\fm}(\bar\tE)\cdot \bar c = 0$ for any $n\in\N_0$.
		By the analyticity of $ \ora{\omega}_{\fm}(\tE)$ (see Remark \ref{remark analitycity}), we deduce 
		that the function $ \tE \mapsto \ora{\omega}_{\fm}(\tE)\cdot \bar c$ is identically zero on $[\tE_1,\tE_2]$, which contradicts Proposition \ref{prop.nondeg}.
	\end{proof}
	
	Thanks to Proposition \ref{prop:trans_un}, we can finally prove that the
	 Diophantine non-resonant condition for $\ora{\omega}_{\fm}(\tE)$ holds on a large set of parameters.
	\begin{prop}\label{omega.fm.tE.diophantine}
		Let $(\kappa_{0}+\tfrac14)\pi<\tE_1<\tE_2<\infty$ be given. Let also $\bar\tau\geq m_0\kappa_{0}$  and $\bar\upsilon\in(0,1)$. Then the set
		\begin{equation}\label{diofantea omega vec E imperturbato nella proof}
			\bar\cK = \bar\cK(\bar\upsilon,\bar\tau) := \big\{  \tE\in [\tE_1,\tE_2] \,: \, 	|\vec \omega_{\mathtt m}(\mathtt E) \cdot \ell| \geq  \bar\upsilon \braket{\ell}^{-\bar\tau}, \ \ \forall \,\ell \in \Z^{\kappa_0} \setminus \{ 0 \}\  \big\}\,,
		\end{equation}
		is of large measure with respect to $\bar\upsilon$, namely $|[\tE_1,\tE_2] \setminus \bar\cK| = o(\bar\upsilon^{1/m_0})$.
	\end{prop}
	\begin{proof}
		We write
		\begin{equation}\label{bar cK c}
			\bar\cK^c := [\tE_1,\tE_2] \setminus \bar\cK = \bigcup_{\ell\neq 0} \bar R_{\ell} = \bigcup_{\ell\neq 0} \big\{  \tE\in [\tE_1,\tE_2] \,:\, |\ora{\omega}_{\fm}(\tE)\cdot \ell  | < \bar\upsilon \braket{\ell}^{-\bar\tau} \big\}\,.
		\end{equation}
		We claim that $|\bar R_{\ell}| \lesssim (\bar\upsilon\braket{\ell}^{-(\bar\tau+1)})^{\frac{1}{m_0}}$. We write
		\begin{equation*}
			\bar R_{\ell}= \Big\{ \tE \in [\tE_1,\tE_2] \, : \, |\bar f_{\ell}({\mathtt E})| < \bar\upsilon \braket{\ell}^{-(\bar\tau+1)}  \Big\}\,,
		\end{equation*}
		where $\bar f_{\ell}({\mathtt E}):=\ora{\omega}_{\fm}({\mathtt E})\cdot \frac{\ell}{\braket{\ell}}$. By Proposition \ref{prop:trans_un}, we have $\max_{0\leq n \leq k_0}|\pa_{\mathtt E}^n \bar f_{\ell}({\mathtt E})|\geq \rho_0$ for any ${\mathtt E}\in[\tE_1.\tE_2]$. In addition, by Remark \ref{remark analitycity}, we have  $\max_{0\leq n \leq m_0}|\pa_{\mathtt E}^n \bar f_{\ell}({\mathtt E})|\leq C$ for any ${\mathtt E}\in[\tE_1.\tE_2]$ for some constant $C=C(\tE_1,\tE_2,m_0)>0$. In particular, $\bar f_{\ell}$ is of class $\cC^{k_0-1}=\cC^{m_0+1}$. Thus, Theorem 17.1 in \cite{Russ} applies, whence the claim follows. Finally, we estimate \eqref{bar cK c} by
		\begin{equation}
			|\bar\cK^c| \leq \sum_{\ell\neq 0} |\bar R_{\ell}|  \lesssim \bar\upsilon^{\frac{1}{m_0}} \sum_{\ell \neq 0} \braket{\ell}^{-\frac{\bar\tau+1}{m_0}} \lesssim \bar\upsilon^{\frac{1}{m_0}}
		\end{equation}
		since $\bar\tau > m_0 \kappa_{0} -1$. This concludes the proof.
	\end{proof}

	\section{Proof of Theorem \ref{main.thm1} and measure estimates}\label{subsec:measest}
	
	Assuming that Theorem \ref{NMT} holds, we deduce now Theorem \ref{main.thm1}. 
	By \eqref{alpha_infty}, for any $\mathtt A \in {\cal J}_\varepsilon(\mathtt E)$, the $\mathtt A$-dependent family of functions $\alpha_{\infty}(\cdot, \mathtt A)$ from $\t\Omega$ into their images $\alpha_{\infty}(\t\Omega \times \{ \mathtt A\})$ are invertible and 
	\begin{equation}\label{inv_alpha}
		\begin{aligned}
			& \beta = \alpha_{\infty}(\omega, \mathtt A) = \omega+r_{\varepsilon}(\omega, \mathtt A)\,, \\
			&   \abs{ \breve{r}_{\varepsilon} }^{k_0,\upsilon} \lesssim \sqrt{\varepsilon} \upsilon^{-1}\,,
		\end{aligned} \Leftrightarrow \quad \begin{aligned}
		&  \omega = \alpha_{\infty}^{-1}(\beta, \mathtt A) = \beta+\breve{r}_{\varepsilon}(\beta, \mathtt A)\,,  \\
		&   \abs{ \breve{r}_{\varepsilon} }^{k_0,\upsilon} \lesssim \sqrt{\varepsilon} \upsilon^{-1}\,.
	\end{aligned}
	\end{equation}
	Then, for any $\beta\in\alpha_{\infty}(\tG^\upsilon \times {\cal J}_\varepsilon(\mathtt E))$, Theorem \ref{NMT} proves the existence of an embedded invariant torus filled by quasi-periodic solutions with Diophantine frequency $\omega=\alpha_{\infty}^{-1}(\beta, \mathtt A)$ for the Hamiltonian
	\begin{equation*}
		{\cal H}_{\varepsilon, \beta} = \beta \cdot I+  \tfrac12\big(z, ( \begin{smallmatrix}
			-\cL_{\fm}(\tE) & 0 \\ 0 & {\rm Id}
		\end{smallmatrix} )  z \big)_{L^2} + \sqrt{\varepsilon} \, {\cal P}_\varepsilon \,.
	\end{equation*}
	Consider the curve of the unperturbed tangential frequency vector 
	$ \mathtt A \to \ora{\omega}_{\fm}(\mathtt A)$ in \eqref{omega0.vec}.
	In Theorem \ref{MEASEST} below we prove that, for a density $1$ set of parameters $\mathtt A \in {\cal J}_\varepsilon(\mathtt E)$, the vector $\alpha_{\infty}^{-1}(\ora{\omega}_{\fm}(\mathtt A), \mathtt A)$ is in $\tG^\upsilon$,
	obtaining
	an embedded torus for the Hamiltonian $\cH_{\varepsilon,\beta}$ with $\beta=\alpha_{\infty}(\omega,\tA)=\ora{\omega}_{\fm}(\tA)$, and therefore for the Hamiltonian ${\cal H}_{\varepsilon}$ in  \eqref{H_epsilon}, filled by quasi-periodic solutions with Diophantine frequency vector 
	$\omega = \alpha_{\infty}^{-1}(\ora{\omega}_{\fm}(\mathtt A), \mathtt A) $, denoted  $ \wt \omega $ in Theorem \ref{main.thm1}. Clearly, by the estimates \eqref{lip omega mathtt A E}, \eqref{inv_alpha}, one has that the vector $\tilde \omega$ satisfies 
	\begin{equation}\label{asintotica frequenze vere per main teo}
	\widetilde \omega = \vec\omega_{\mathtt m}(\mathtt A) + O(\sqrt{\varepsilon} \upsilon^{- 1}) = \vec\omega_{\mathtt m}(\mathtt E) + O(\sqrt{\varepsilon} + \sqrt{\varepsilon} \upsilon^{- 1} ) = \vec\omega_{\mathtt m}(\mathtt E) + O( \sqrt{\varepsilon} \upsilon^{- 1} ) \,. 
	\end{equation}
	Thus, the function $ì A(i_{\infty}(\wt \omega x)) = (\zeta_1(\wt \omega x, y), \zeta_2(\wt \omega x, y))$,  where $A$ is defined in \eqref{aa_coord},
	is a quasi-periodic solution of the equation \eqref{first_order} and hence, by recalling Lemma \ref{unforcing.lemma}, \eqref{link eta varepsilon}, \eqref{equazione secondo ordine con h eta riscalata}, \eqref{definizione zeta 1 zeta 2 zeta}, $ h_\eta(y) + \varepsilon \zeta_1(\omega x, y)$ is a quasi-periodic solution of \eqref{elliptic_eq}.
	This proves Theorem \ref{main.thm1} together with the following measure estimate.
	
	\begin{thm} {\bf (Measure estimate).}\label{MEASEST}
		Let
		\begin{equation}\label{param_small_meas}
	\begin{aligned}
		& \upsilon = \varepsilon^{\rm a} \,, \quad  0 <{\rm a}<\min\big\{ {\rm a}_0, \tfrac{\td(\tau)-1}{2(\td(\tau)-\frac{1}{m_0})} \big\}\,, \quad \td(\tau):=\tfrac{\tau +1 -m_0\kappa_{0}}{m_0(\bar\tau +1)}  \\
		&  \tau > {\rm max}\{m_0 \kappa_0-1\,,\, m_0 (\kappa_0 + \bar \tau +1)-1 \} \,,  
	\end{aligned}
\end{equation}
		where $m_0$ is the index of non-degeneracy given in Proposition \ref{prop:trans_un}, $k_0:= m_0+2$, $\bar\tau \geq m_0\kappa_{0}$ is fixed and $\ta_0\in (0,1)$ is defined in \eqref{param.NASH} in Theorem \ref{NASH} . Then, fixed $\tE\in\bar\cK$, with $\bar\cK$ as in \eqref{diofantea omega vec E imperturbato.intro} (see also \eqref{diofantea omega vec E imperturbato nella proof} below), for $ \varepsilon \in (0, \varepsilon_0) $ small enough, the set
		\begin{equation}\label{Gvare}
			\cK_{\varepsilon}=\cK_{\varepsilon}(\tE):= \big\{ \mathtt A \in {\cal J}_\varepsilon(\mathtt E) = [\mathtt E - \sqrt{\varepsilon}, \mathtt E + \sqrt{\varepsilon}]
			\, : \,  \alpha_{\infty}^{-1}( \ora{\omega}_{\fm}({\mathtt A}), \mathtt A)
			\in \tG^\upsilon \big\}
		\end{equation}
		has density $1$ as $\varepsilon \to 0$, namely 
		$$
		\frac{|\cK_{\varepsilon}(\tE)|}{|\cJ_{\varepsilon}(\tE)|}=\frac{| \cK_{\varepsilon}(\tE)|}{2 \sqrt{\varepsilon}} \rightarrow 1\quad \text{as} \quad \varepsilon \to 0\,, \quad  \text{uniformly in } \, \tE\in\bar\cK\,.
		$$ 
	\end{thm}

	The rest of this section is devoted to prove Theorem \ref{MEASEST}.  
	The key point to compute the density of the set ${\cal K}_\varepsilon(\tE)$ is that  the unperturbed frequency $\vec\omega_{\mathtt m}({\mathtt E})$ is Diophantine with constants $\bar \upsilon \in (0, 1)$ and $\bar \tau$ stronger than the ones in \eqref{Cset_infty}, namely with $\upsilon \ll \bar\upsilon$ and $\tau \gg \bar \tau$, see Proposition \ref{omega.fm.tE.diophantine}.
	
	By \eqref{inv_alpha}
	we have that, for $\mathtt A \in {\cal J}_\varepsilon(\mathtt E)$, 
	\begin{equation}\label{Om-per}
		\ora{ \omega}_{\varepsilon} ({\mathtt A}):= \alpha_{\infty}^{-1}(\ora{\omega}_{\fm}(\mathtt A), \mathtt A) =
		\ora{\omega}_{\fm}(\mathtt A) +\ora{r}_{\varepsilon}(\mathtt A) \,,
	\end{equation}
	where $\ora{r}_\varepsilon(\mathtt A) 
	:= \breve{r}_\varepsilon(\ora{\omega}_{\fm}(\mathtt A), \mathtt A) $
	satisfies 
	\begin{equation}\label{eq:tang_res_est}
		|\partial_{\mathtt A}^k {\vec r}_{\varepsilon} ({\mathtt A})| \leq C \sqrt{\varepsilon}\upsilon^{-(1+k)} \,, \quad \forall\,0 \leq k \leq k_0 \,, \  \ \text{uniformly on } \ {\cal J}_\varepsilon(\mathtt E)\,.
	\end{equation}
	By \eqref{Cset_infty}, the Cantor set $\cK_{\varepsilon}(\tE)$ in \eqref{Gvare} becomes
	\begin{equation}\label{Kvare}
		\cK_{\varepsilon}(\tE):= \Big\{ \mathtt A \in {\cal J}_\varepsilon(\mathtt E) \,:\, |\ora{\omega}_{\varepsilon}({\mathtt A}) \cdot\ell| \geq \upsilon\braket{\ell}^{-\tau} \,, \ \forall\,\ell\in\Z^{\kappa_0}\setminus\{0\}  \Big\}.
	\end{equation}
	We estimate the measure of the complementary set
	\begin{equation}\label{kvare.compl}
		\begin{aligned}
		 &	\cK_{\varepsilon}^c(\tE)\!:=\!{\cal J}_\varepsilon(\mathtt E)\setminus\cK_{\varepsilon}(\tE) \!:=\! \bigcup_{\ell\neq 0}R_{\ell}(\mathtt E)\!:=\! \bigcup_{\ell\neq 0} \Big\{ {\mathtt A} \in \!{\cal J}_\varepsilon(\mathtt E) :  |\ora{\omega}_{\varepsilon} ({\mathtt A})\cdot\ell|\!<\!\upsilon\braket{\ell}^{-\tau} \! \Big\}\,.
		\end{aligned}
	\end{equation}
	To estimate the measure of the sets $R_{\ell}(\mathtt E)$ in \eqref{kvare.compl}, the key point is to show that the perturbed linear frequencies satisfy the similar lower bound in \eqref{eq:0_meln} in Proposition \ref{prop:trans_un}. The transversality property actually holds not only in the vicinity of the fixed $\mathtt E$, but  also on the full parameter set $[\mathtt E_1, \mathtt E_2]$. 
	
	\begin{lem} {\bf (Perturbed transversality)} \label{lem:pert_trans}
		For $\varepsilon\in(0,\varepsilon_0)$ small enough and for all ${\mathtt A}\in[\tE_1,\tE_2]$, 
		\begin{equation}
			\max_{0\leq n \leq m_0} | \partial_{\mathtt A}^n \ora{\omega}_{\varepsilon}({\mathtt A})\cdot \ell | \geq \frac{\rho_0}{2}\braket{\ell} \,, \quad \forall\,\ell\in\Z^{\kappa_0}\setminus\{0\} \,; \label{eq:0_meln_pert}
		\end{equation}
		here $\rho_0$ is the amount of non-degeneracy that has been defined in Proposition \ref{prop:trans_un}. In particular, the same estimate \eqref{eq:0_meln_pert} holds for any $\tA\in \cJ_{\varepsilon}(\tE)$, with the constant $\rho_0$ independent of $\varepsilon>0$.
	\end{lem}
	\begin{proof}
		The estimate \eqref{eq:0_meln_pert} follows directly from \eqref{Om-per}-\eqref{eq:tang_res_est} provided $\sqrt\varepsilon\upsilon^{-(1+m_0)}\leq \rho_0/(2C)$, which, by \eqref{k0.def} and \eqref{param_small_meas}, is satisfied for $\varepsilon$ sufficiently small.
	\end{proof}
	As an application of Rüssmann Theorem 17.1 in \cite{Russ}, we deduce the following lemma.
	\begin{lem} {\bf (Estimates of the resonant sets)} \label{lem:meas_res}
		The measure of the sets $R_{\ell}(\mathtt E)$ in \eqref{kvare.compl} satisfies $|R_{\ell}(\mathtt E)| \lesssim (\upsilon\braket{\ell}^{-(\tau+1)})^{\frac{1}{m_0}}$ for any $\ell\neq 0$.
	\end{lem}
	\begin{proof}
		We write
		\begin{equation*}
			R_{\ell}({\mathtt E}) = \Big\{ {\mathtt A} \in \cJ_{\varepsilon}(\tE)\, : \, |f_{\ell}({\mathtt A})| < \upsilon \braket{\ell}^{-(\tau+1)}  \Big\}\,,
		\end{equation*}
		where $f_{\ell}({\mathtt A}):=\ora{\omega}_{\varepsilon}({\mathtt A})\cdot \frac{\ell}{\braket{\ell}}$. By \eqref{eq:0_meln_pert}, we have $\max_{0\leq n \leq k_0}|\pa_{\mathtt A}^n f_{\ell}({\mathtt A})|\geq \rho_0/2$ for any ${\mathtt A}\in[\tE_1.\tE_2]$. In addition,  \eqref{Om-per}-\eqref{eq:tang_res_est} imply that $\max_{0\leq n \leq m_0}|\pa_{\mathtt A}^n f_{\ell}({\mathtt A})|\leq C$ for any ${\mathtt A}\in[\tE_1.\tE_2]$, provided $\sqrt\varepsilon\upsilon^{-(1+m_0)}$ is small enough, namely, by \eqref{param_small_meas}, for $\varepsilon$ small enough. In particular, $f_{\ell}$ is of class $\cC^{k_0-1}=\cC^{m_0+1}$. Thus, Theorem 17.1 in \cite{Russ} applies, whence the lemma follows.
	\end{proof}

\begin{lem}\label{risonanti bassi zero}
There exists $C = C(\bar\upsilon) > 0$ such that, if $0 < |\ell| \leq C (\upsilon\varepsilon^{-\frac12})^{ \frac{1}{(\bar \tau + 1)}}$, then $R_\ell(\tE) = \emptyset$.
\end{lem}
\begin{proof}
 By the estimates \eqref{lip omega mathtt A E}, \eqref{inv_alpha}, we deduce
$
|\omega_\varepsilon({\mathtt A}) - \vec\omega_{\mathtt m}({\mathtt E})| \lesssim \sqrt{\varepsilon} \upsilon^{- 1}
$ for any $\mathtt A \in \cJ_{\varepsilon}(\tE)$.
Hence, by Proposition \ref{omega.fm.tE.diophantine}, we get, for some constant $C \geq 0$,
$$
\begin{aligned}
|\omega_\varepsilon({\mathtt A}) \cdot \ell | & \geq |\vec\omega_{\mathtt m}(\mathtt E) \cdot \ell| - C \sqrt{\varepsilon} \upsilon^{- 1} |\ell| \geq \frac{\bar \upsilon}{| \ell |^{\bar \tau}} - C \sqrt{\varepsilon} \upsilon^{- 1} |\ell|  \geq \frac{\bar \upsilon}{2 | \ell |^{\bar \tau}}
\end{aligned}
$$
provided that the condition stated in the statement holds. Hence, for $\upsilon \ll \bar \upsilon$ and $\tau\gg \bar\tau$, this implies that $R_\ell(\tE) = \emptyset$.
\end{proof}
	
	\begin{proof}[Proof of Theorem \ref{MEASEST} completed]
		The series $\sum_{\ell \neq 0} |\ell|^{-\frac{\tau + 1}{m_0}}$ is convergent because $\frac{\tau+1}{m_0}>\kappa_0$ by \eqref{param_small_meas}. Hence, by Lemmata \ref{lem:meas_res}, \ref{risonanti bassi zero}, the measure of the set $\cK_{\varepsilon}^c(\tE)$ in \eqref{kvare.compl} is estimated by
		\begin{equation*}
	\begin{aligned}
		|\cK_{\varepsilon}^c(\tE)| & \leq \sum_{|\ell| > C (\upsilon\varepsilon^{-\frac12})^{\frac{1}{\bar \tau + 1}} } |R_{\ell}(\tE)|  
		\lesssim  \upsilon^{\frac{1}{m_0}}\sum_{|\ell| > C (\upsilon\varepsilon^{-\frac12})^{\frac{1}{\bar \tau + 1}}  } \frac{1}{|\ell|^{\frac{\tau + 1}{m_0}}} \\
		& 
		\lesssim \upsilon^{\frac{1}{m_0}} \Big( \dfrac{\sqrt{\varepsilon}}{\upsilon} \Big)^{\frac{\tau + 1 - m_0 \kappa_0 }{m_0(\bar \tau + 1)}}   \simeq \upsilon^{\frac{1}{m_0}} \Big( \dfrac{\sqrt{\varepsilon}}{\upsilon} \Big)^{\alpha (\tau - \beta)} \,, 
	\end{aligned}
\end{equation*}
	where $\alpha := (m_0(\bar\tau+1))^{-1}$ and 
		$\beta:= m_0 \kappa_{0}-1$.
		Therefore, we have
		$$
		\frac{|{\cal K}_\varepsilon^2(\tE)|}{2 \sqrt{\varepsilon}} \lesssim {\varepsilon}^{\mathtt p_1} \upsilon^{- \mathtt p_2}\,, \quad 	\mathtt p_1 := \frac{\alpha( \tau - \beta )- 1}{2} \,,  \quad \mathtt p_2 :=\frac{m_0\alpha(\tau -\beta) -1}{m_0} \,.
				$$
		Since $\tau > m_0 (\kappa_{0}+\bar\tau+1)$ by \eqref{param_small_meas}, we have that $\tp_1>0$ and, consequently, also $\tp_2>0$.
		Then, for $\upsilon = \varepsilon^{\mathtt a}$ with $$0 < \mathtt a < \frac{\mathtt p_1}{\mathtt p_2}=\frac{\alpha(\tau-\beta)-1}{2(\alpha(\tau-\beta)-\tfrac{1}{m_0})}<\frac12 <1\,,$$ we have $(2\sqrt\varepsilon)^{-1}| \cK_{\varepsilon}^c(\tE)|\lesssim \varepsilon^{\tp_1-\ta \tp_2}\to 0$ as $\varepsilon\to 0$.
		It implies $(2\sqrt\varepsilon)^{-1}|\cK_{\varepsilon}(\tE)|\geq 1- C\varepsilon^{\frac{{\rm a}}{m_0}}$ and the proof of Theorem \ref{MEASEST} is concluded.
	\end{proof}

	\section{Approximate inverse}\label{sec:approx_inv}
In order to implement a convergent Nash-Moser scheme that leads to a solution of $\cF(i,\alpha)=0$, where $ \cF (i, \alpha) $ is the nonlinear operator  defined in \eqref{F_op},  
we construct the \emph{approximate right inverses} of the linearized operators
\begin{equation*}
	\di_{i,\alpha}\cF(i_{0},\alpha_{0})[\whi,\wh\alpha] = \omega\cdot \pa_\bx \whi - \di_i X_{{\cal H}_{\varepsilon, \alpha}}\left( i_{0}(\bx) \right)[\whi] - \left(\wh\alpha,0,0\right) \,.
\end{equation*}
Note that $\di_{i,\alpha}\cF (i_{0},\alpha_{0})=\di_{i,\alpha}\cF(i_{0})$ is independent of $\alpha_{0}$.
We assume that the torus $ i_{0} (\bx) = ( \theta_{0} (\bx), I_{0} (\bx), z_{0} (\bx)) $ 
is  reversible,  according to  \eqref{RTTT}.

In the sequel we shall assume  the smallness condition,  
$$
\sqrt{\varepsilon}\upsilon^{-1} \ll 1 \, . 
$$

\smallskip

First of all, we state  tame estimates for the composition operator induced by the Hamiltonian vector field $X_{{\cal P}_\varepsilon}= ( \pa_I  {\cal P}_\varepsilon , - \pa_\theta {\cal P}_\varepsilon, J \nabla_{z} {\cal P}_\varepsilon )$ in \eqref{F_op}.
\begin{lem}{\bf (Estimates of the perturbation ${\cal P}_\varepsilon$).} \label{XP_est}
	Let $S \geq 2(s_0+\mu)$, with $\mu >0$ as in Lemma \ref{estimates f eta g eta}-$(ii)$. Let $\fI(\bx)$ in \eqref{ICal} satisfy $\norm{ \fI }_{\cX^{s_0}}^{k_0,\upsilon}\leq 1$. Then, for any $s_0 \leq s \leq S/2 - \mu$ and any $\varepsilon \ll 1$, we have
	$	\norm{ X_{{\cal P}_\varepsilon}(i) }_{\cY^{s}}^{k_0,\upsilon} \lesssim_s  1 + \norm{ \fI }_{\cX^{s}}^{k_0,\upsilon} $, 
	and, for all $\whi:= (\wh\theta,\whI,\whz)$, $\whi_1:= (\wh\theta_1,\whI_1,\whz_1)$, $\whi_2:= (\wh\theta_2,\whI_2,\whz_2)$
	\begin{align*}
		\norm{ \di_i X_{{\cal P}_\varepsilon}(i)[\whi] }_{\cY^{s}}^{k_0,\upsilon} &\lesssim_s \norm{\whi}_{{\cal Y}^s}^{k_0,\upsilon} + \norm{ \fI }_{\cX^{s}}^{k_0,\upsilon}\norm{ \whi}_{{\cal Y}^{s_0}}^{k_0,\upsilon} \,, \\
		\norm{ \di_i^2 X_{{\cal P}_\varepsilon}(i)[\whi_1,\whi_2] }_{\cY^{s}}^{k_0,\upsilon} &\lesssim_s \norm{\whi_1}_{{\cal Y}^s}^{k_0,\upsilon}\norm{\whi_2}_{{\cal Y}^{s_0}}^{k_0,\upsilon} + \norm{\whi_1}_{{\cal Y}^{s_0}}^{k_0,\upsilon}\norm{\whi_2}_{{\cal Y}^s}^{k_0,\upsilon} + \norm{ \fI }_{\cX^{s}}^{k_0,\upsilon} ( \norm{\whi}_{{\cal Y}^{s_0}}^{k_0,\upsilon} )^2 \,.
	\end{align*}
\end{lem}
\begin{proof}
	From \eqref{H_epsilon}, \eqref{Hami} and \eqref{aa_coord}, \eqref{ham.vf.aa}, the Hamiltonian vector field for ${\cal P}_\varepsilon=P_{\varepsilon}\circ A + \tfrac{1}{\sqrt\varepsilon}\big( \ora{\omega}_{\fm}(\tE)-\ora{\omega}_{\fm}(\tA) \big)\cdot I$ is given by
	\begin{equation}
		X_{{\cal P}_\varepsilon} = \begin{pmatrix}
			[\pa_I v^\intercal(\theta,I)]^T \pa_{I} P_{\varepsilon}(A(\theta,I,z)) + \tfrac{1}{\sqrt{\varepsilon}}\big( \ora{\omega}_{\fm}(\tE)-\ora{\omega}_{\fm}(\tA) \big)\\
			-[\pa_\theta v^\intercal(\theta,I)]^T \pa_{\theta} P_{\varepsilon}(A(\theta,I,z)) \\
			\Pi_\perp J^{-1}\nabla_{z} P_{\varepsilon}(A(\theta,I,z))
		\end{pmatrix}\,,
	\end{equation}
	where $\Pi^\perp$ is the projection onto the subspace $\cX_\perp$ in \eqref{cX.perp}. Since $\nabla_{z} P_\varepsilon(\zeta_1, \zeta_2) = q_\varepsilon(y, \zeta_1)$, the claimed estimates follow by Lemma \ref{stime cal Q eta}, by the definition of $v^\intercal$ and $A$ in \eqref{aa_coord}, by the interpolation inequality \eqref{prod} (recall also the definitions given in \eqref{prima def cal X s}) and by the estimate in \eqref{lip omega mathtt A E}.
\end{proof}

\subsection{Invertibility of the linearized operator}

Along this section, we assume the following hypothesis, which is verified by the approximate solutions obtained at each step of the Nash-Moser Theorem \ref{NASH}. We recall the definitions of the spaces ${\cal X}^s_\bot, {\cal Y}^s_\bot, {\cal X}^s, {\cal Y}^s$ in \eqref{prima def cal X s bot}-\eqref{prima def cal X s} that we shall use in the whole section. 
\begin{itemize}
	\item {\sc ANSATZ.} The map $\lambda\mapsto \fI_0(\lambda) = i_0(\lambda;\bx)- (\bx,0,0)$ is $k_0$-times differentiable with respect to the parameter $\lambda=(\omega,\tA)\in \R^{\kappa_0}\times \cJ_{\varepsilon}(\tE)$ and, for some $\sigma :=\sigma(\kappa_0, k_0, \tau) \gg 0$, $\upsilon\in (0,1)$,
	\begin{equation}\label{ansatz}
		\norm{\fI_0}_{{\cal X}^{s_0+\sigma}}^{k_0,\upsilon} + \abs{ \alpha_0-\omega }^{k_0,\upsilon} \leq C \sqrt{\varepsilon} \upsilon^{-1} \,, \quad \sqrt{\varepsilon} \upsilon^{- 1} \ll 1\,. 
	\end{equation} 
\end{itemize}
We remark that, in the sequel, we denote by $\sigma \equiv \sigma(\kappa_0,k_0, \tau) \gg 0$ constants, 
which may increase from lemma to lemma, that represent ``loss of derivatives''.

As in \cite{BB,BM,BBHM}, we first modify the approximate torus $i_0 (\bx) $ to obtain a nearby isotropic torus $i_\delta (\bx) $, namely such that the pull-back 1-form  $i_\delta^*\Lambda $  is closed, 
where $\Lambda$ is the Liouville 1-form defined in 
\eqref{liouville}.  
We first consider the pull-back $ 1$-form 
\begin{equation}\label{ak}
	i_0^*\Lambda  = \sum_{k=1}^{\nu} a_k(\bx) \di \bx_k \, , \ \ 
	a_k(\bx) := -\big( [ \pa_\bx \theta_0(\bx) ]^\top I_0(\bx) \big)_k +\tfrac12 
	\big( J^{-1} z_0(\bx), \pa_{\bx_k} z_0(\bx) \big)_{L^2} \,, 
\end{equation} 
and its exterior differential 
\begin{equation}\label{Akj}
	i_0^*\cW  = \di i_0^*\Lambda = \sum_{1\leq k < j \leq \nu} A_{kj} \di \bx_k \wedge \di \bx_j \,, 
	\quad 
	A_{kj}(\bx)  := \pa_{\bx_k} a_j(\bx) - \pa_{\bx_j}a_k(\bx) \, .
\end{equation}
By the formula given in Lemma 5.3 in \cite{BB}, we deduce that for any $s \geq s_0$, 
if $\omega$ belongs to $\mathtt G^\upsilon$ (see \eqref{Cset_infty}), the estimate (assuming the ansatz \eqref{ansatz}), 
\begin{equation}\label{stimaAjk}
	\begin{aligned}
		& \norm{ A_{kj} }_{s}^{k_0,\upsilon} \lesssim_s  \upsilon^{-1}\big( \norm{ Z }_{{\cal Y}^{s+\sigma}}^{k_0, \upsilon} +  \norm{ \fI_0 }_{{\cal X}^{s+ \sigma}}^{k_0,\upsilon} \norm{ Z }_{{\cal Y}^{s_0+ \sigma}}^{k_0,\upsilon} \big) \, , \\
		& \| A_{k j}\|_{s}^{k_0, \upsilon}  \lesssim_s \| {\frak I}_0 \|_{{\cal X}^{s + 1}}^{k_0, \upsilon}\,. 
	\end{aligned}
\end{equation}
for some $\sigma \equiv \sigma(k_0, \tau ) \gg 0$ large enough, 
where $Z(\bx)   $  is the  ``error function'' 
\begin{equation}\label{ZError}
	Z(\bx)   
	:= \cF(i_0,\alpha_{0})
	:= \omega\cdot \pa_{\bx} i_0(\bx) - X_{H_{\alpha_0}}(i_0(\bx))\,.
\end{equation}
Note that, if $ Z (\bx) = 0 $, the torus $ i_0 (\bx) $ is invariant for $ X_{H_{\alpha_0}} $ and
the  1-form $ i_0^* \Lambda $ is closed, namely the torus $ i_0 (\bx) $ is isotropic.
We denote below the Laplacian $\Delta_\bx:= \sum_{k=1}^{\nu}\pa_{\bx_k}^2$.
\begin{lem} {\bf (Isotropic torus)} \label{torus_iso}
	The torus $i_\delta(\bx):= ( \theta_0(\bx),I_\delta(\bx),w_0(\bx) )$, defined by
	\begin{equation}\label{Idelta}
		I_\delta(\bx)\!:=\! I_0(\bx) + [ \pa_\bx \theta_0(\bx) ]^{-\top}\rho(\bx) \,, 
		\quad \rho = (\rho_j)_{j=1, \ldots,\kappa_0} \, , \quad \rho_j(\bx)\!
		:= \Delta_\bx^{-1} \sum_{k=1}^{\kappa_0}\pa_{\bx_k}A_{kj}(\bx)\,,
	\end{equation}
	is isotropic. 
	Moreover, there is $\sigma:= \sigma(\kappa_0, k_0, \tau) \gg 0$ such that, for $S \geq 2(s_0+\sigma)$ if \eqref{ansatz} holds, then for all $ s_0 \leq s \leq S/2 - \sigma $, 
	\begin{align}
		\| I_\delta-I_0 \|_{s}^{k_0,\upsilon} &\lesssim_s \norm{\fI_0}_{{\cal X}^{s+1}}^{k_0,\upsilon} \, ,    \label{ebb1} \\
		\| I_\delta-I_0 \|_{s}^{k_0,\upsilon} & \lesssim_s \upsilon^{-1}
		\big( \norm{Z}_{{\cal Y}^s}^{k_0,\upsilon} +\norm{Z}_{{\cal Y}^s}^{k_0,\upsilon} \norm{ \fI_0 }_{{\cal X}^{s + \sigma}}^{k_0,\upsilon} \big) \,,\label{ebb2}  \\
		\| \cF(i_\delta,\alpha_0) \|_{{\cal Y}^s}^{k_0,\upsilon} &\lesssim_s  \norm{Z}_{{\cal Y}^{s + \sigma}}^{k_0,\upsilon} +\norm{Z}_{{\cal Y}^{s + \sigma}}^{k_0,\upsilon} \norm{ \fI_0 }_{{\cal X}^{s + \sigma}}^{k_0,\upsilon} \,. \label{ebb4} 
	\end{align}
	Furthermore  $i_\delta(\bx)$
	is  a reversible  torus,  cfr.  \eqref{RTTT}.
\end{lem}
\begin{proof}
	The estimates \eqref{ebb1}-\eqref{ebb4} follow e.g. as in Lemma 5.3 in \cite{BBHM}.
\end{proof}

In order to find an approximate inverse of the linearized operator $\di_{i,\alpha}\cF(i_\delta)$, we introduce the symplectic diffeomorphism $G_\delta:(\phi,\fy,\tz) \rightarrow (\theta,I,z)$ of the phase space $ \T^{\kappa_0}\times \R^{\kappa_0} \times \cX_\perp$, 
\begin{equation}\label{Gdelta}
	\begin{pmatrix}
		\theta \\ I \\ z
	\end{pmatrix} := G_\delta \begin{pmatrix}
		\phi \\ \fy \\ \tz
	\end{pmatrix} := \begin{pmatrix}
		\theta_0(\phi) \\ 
		I_\delta(\phi) + \left[ \pa_\phi \theta_0(\phi) \right]^{-\top}\fy + \left[(\pa_\theta\wtz_0)(\theta_0(\phi))  \right]^\top J^{-1} \tz \\
		z_0(\phi) + \tz
	\end{pmatrix}\,,
\end{equation}
where $\wtz_0(\theta):= z_0(\theta_0^{-1}(\theta))$.
It is proved in Lemma 2 of \cite{BB} that $G_\delta$ is symplectic, because the torus $i_\delta$ is isotropic (Lemma \ref{torus_iso}). In the new coordinates, $i_\delta$ is the trivial embedded torus $(\phi,\fy,\tz)=(\phi,0,0)$.
The diffeomorphism $G_\delta$ in \eqref{Gdelta} is reversibility preserving.

Under the symplectic diffeomorphism $G_\delta $, the Hamiltonian vector field $X_{H_\alpha}$ changes into
\begin{equation}\label{Kalpha}
	X_{K_\alpha} = \left(DG_\delta  \right)^{-1} X_{H_\alpha} \circ G_\delta 
	\qquad {\rm where} \qquad K_\alpha := H_\alpha \circ G_\delta \,.
\end{equation}
We have  that $ K_\alpha $ is reversibility  preserving, in the sense that 
\begin{equation}\label{Ka.prop}
	K_\alpha\circ \vec \cS = K_\alpha \,.
\end{equation} 
The Taylor expansion of $K_\alpha$ at the trivial torus $(\phi,0,0)$ is
\begin{equation}\label{taylor_Kalpha}
	\begin{aligned}
		K_\alpha(\phi,\fy,\tz) =& \ K_{00}(\phi,\alpha) + K_{10}(\phi,\alpha) \cdot \fy + 
		( K_{01}(\phi,\alpha),\tz )_{L^2} + \tfrac12 K_{20}(\phi) \fy\cdot \fy \\
		& + ( K_{11}(\phi)\fy,\tz )_{L^2} + \tfrac12 ( K_{02}(\phi)\tz,\tz )_{L^2} + K_{\geq 3}(\phi,\fy,\tz)\,,
	\end{aligned}
\end{equation}
where $K_{\geq 3}$ collects all terms at least cubic in the variables $(\fy,\tz)$. By \eqref{Halpha} and \eqref{Gdelta}, the only Taylor coefficients that depend on $\alpha$ are $K_{00}\in \R$, $K_{10}\in\R^{\kappa_0}$ and $K_{01}\in \cX_\perp $, whereas the $ \kappa_0 \times \kappa_0 $ symmetric matrix $K_{20} $, $K_{11}\in\cL ( \R^{\kappa_0},\cX_\perp)$ and the linear self-adjoint operator $ K_{02} $,  acting on 
$\cX_\perp$, 
are independent of it. 

Differentiating the identities in \eqref{Ka.prop} at $(\phi,0,0)$, we have
\begin{align}
	& K_{00}(-\phi) = K_{00}(\phi)\,, \quad  K_{10}(-\phi) = K_{10}(\phi)\,, \quad K_{20}(-\phi) = K_{20}(\phi)\,, \label{Ka.rev} \\
	&  \cS \circ K_{01}(-\phi)  = K_{01}(\phi)\,, \quad \cS \circ K_{11}(-\phi) = K_{11}(\phi)\,,  \quad K_{02}(-\phi)\circ \cS =  \cS \circ K_{02}(\phi)\,. \nonumber
\end{align}
The Hamilton equations associated to \eqref{taylor_Kalpha} are
\begin{equation}\label{hameq_Kalpha}
	\begin{cases}
		\dot\phi =   K_{10}(\phi,\alpha) + K_{20}(\phi)\fy + [K_{11}(\phi)]^\top \tz + \pa_{\fy} K_{\geq 3}(\phi,\fy,\tz) \\
		\dot \fy =   - \pa_\phi K_{00}(\phi,\alpha) - [\pa_\phi K_{10}(\phi,\alpha)]^\top \fy - [\pa_\phi K_{01}(\phi,\alpha)]^\top \tz  \\
		\ \ \ \ \ - \pa_\phi\left( \tfrac12 K_{20}(\phi)\fy\cdot \fy + \left( K_{11}(\phi)\fy,\tz \right)_{L^2} + \tfrac12 \left( K_{02}(\phi)\tz,\tz \right)_{L^2} + K_{\geq 3}(\phi,\fy,\tz) \right) \\
		\dot\tz =  J\, \left( K_{01}(\phi,\alpha)+ K_{11}(\phi)\fy + K_{02}(\phi)\tz + \nabla_{\tz} K_{\geq 3}(\phi,\fy,\tz) \right) 
	\end{cases}\,,
\end{equation}
where $\pa_\phi K_{10}^\top $ is the $\kappa_0\times\kappa_0$ transposed matrix and $\pa_\phi K_{01}^\top , K_{11}^\top: \cX_\perp \rightarrow\R^{\kappa_0}$ are defined by the duality relation 
$ (\pa_\phi K_{01}[\wh\phi],\tz  )_{L^2}=\wh\phi\cdot [\pa_\phi K_{01} ]^\top \tz $ 
for any $\wh\phi\in\R^{\kappa_0}$, $\tz\in\cX_\perp$. 
The transpose
$	K_{11}^\top (\phi) $ is defined similarly. 
On an exact solution (that is $Z=0$), the terms $K_{00}, K_{01}$  in the Taylor expansion \eqref{taylor_Kalpha} vanish and $K_{10}= \omega$. 
More precisely, arguing as  in Lemma 5.4 in \cite{BBHM} (with minor adaptations), we have 

\begin{lem}\label{Kcoeff_est}
	There is $ \sigma :=  \sigma (\kappa_0, k_0,  \tau) > 0 $, such that if $S \gg 2(s_0 + \sigma)$ and \eqref{ansatz} holds then for all $ s_0 \leq s \leq S/2 - \sigma$, one has 
	\begin{align}
		& \| \pa_\phi K_{00}(\cdot , \alpha_0)\|_{s}^{k_0, \upsilon} + \| K_{10}(\cdot ,\alpha_0)-\omega\|_{s}^{k_0, \upsilon} + \| K_{01}( \cdot ,\alpha_0) \|_{{\cal Y}^s_\bot}^{k_0, \upsilon} \nonumber \\
		& \ \ \ \ \ \ \ \ \ \ \ \ \ \ \ \ \ \lesssim_s  \norm{Z}_{{\cal Y}^{s+\sigma}}^{k_0,\upsilon} + \norm{Z}_{{\cal Y}^{s+\sigma}}^{k_0,\upsilon}  \norm{ \fI_0 }_{{\cal Y}^{s+\sigma}}^{k_0,\upsilon}  \,, \label{face1} \\
		&\norm{\pa_\alpha K_{00}}_{s}^{k_0,\upsilon} + \norm{ \pa_\alpha K_{10}-{\rm Id} }_{s}^{k_0,\upsilon}  + \norm{ \pa_\alpha K_{01} }_{{\cal Y}^s_\bot}^{k_0,\upsilon}  \lesssim_s \| \fI_0 \|_{{\cal X}^{s + \sigma}}^{k_0,\upsilon}  \,,\label{face2} \\
		& \norm{ K_{20} }_{s}^{k_0,\upsilon}\lesssim_s \sqrt{\varepsilon} ( 1 + \norm{ \fI_0 }_{{\cal X}^{s+\sigma}}^{k_0,\upsilon} )\,,  \label{face3}\\
		& \norm{ K_{11}y }_{{\cal Y}^s_\bot}^{k_0,\upsilon} \lesssim_s \sqrt{\varepsilon} ( \norm{ y }_s^{k_0,\upsilon}+ \norm{y }_{s_0}^{k_0,\upsilon}\norm{\fI_0 }_{{\cal X}^{s+\sigma}}^{k_0,\upsilon} )\,,   \label{face5} \\
		&  \norm{ K_{11}^\top \tz  }_{s}^{k_0,\upsilon} \lesssim_s \sqrt{\varepsilon} ( \norm{\tz}_{{\cal Y}^s_\bot}^{k_0,\upsilon} + \norm{\tz}_{{\cal Y}^{s_0}_\bot}^{k_0,\upsilon}\norm{\fI_0}_{{\cal X}^{s+\sigma}}^{k_0,\upsilon}) \, . \label{face6}
	\end{align}
\end{lem}
Under the linear change of variables
\begin{equation}\label{DGdelta}
	DG_\delta(\bx,0,0)\begin{pmatrix}
		\wh\phi \\ \wh\fy \\ \wh\tz
	\end{pmatrix}:= \begin{pmatrix}
		\pa_\phi \theta_0(\bx) & 0 & 0 \\ \pa_\phi I_\delta(\bx) & [\pa_\phi \theta_0(\bx)]^{-\top} &  [(\pa_\theta\wtw_0)(\theta_0(\bx))]^\top  J^{-1} \\\pa_\phi w_0(\bx) & 0 & {\rm Id}
	\end{pmatrix}\begin{pmatrix}
		\wh\phi \\ \wh\fy \\ \wh\tz
	\end{pmatrix} \,,
\end{equation}
the linearized operator $\di_{i,\alpha}\cF(i_\delta)$ is approximately transformed into the one obtained when one linearizes the Hamiltonian system \eqref{hameq_Kalpha} at $(\phi,\fy,\tz) = (\bx,0,0)$, differentiating also in $\alpha$ at $\alpha_0$ and changing $\pa_x \rightsquigarrow \omega\cdot \pa_\bx$, namely
\begin{equation}\label{lin_Kalpha}
	\begin{footnotesize}
		\begin{pmatrix}
			\widehat \phi  \\
			\widehat \fy    \\ 
			\widehat \tz \\
			\widehat \alpha
		\end{pmatrix} \mapsto
		\begin{pmatrix}
			\omega\cdot \pa_\bx \wh\phi - \pa_\phi K_{10}(\bx)[\wh\phi] - \pa_\alpha K_{10}(\bx)[\wh\alpha] - K_{20}(\bx)\wh\fy - [K_{11}(\bx)]^\top \wh\tz  \\
			\omega\cdot \pa_\bx\wh\fy + \pa_{\phi\phi}K_{00}(\bx)[\wh\phi]+ \pa_\alpha\pa_\phi K_{00}(\bx)[\wh\alpha] + [\pa_\phi K_{10}(\bx)]^\top \wh\fy + [\pa_\phi K_{01}(\bx)]^\top  \wh\tz  \\
			\omega\cdot \pa_\bx \wh\tz - J \,  \big( \pa_\phi K_{01}(\bx)[\wh\phi] + \pa_\alpha K_{01}(\bx)[\wh\alpha] + K_{11}(\bx) \wh\fy + K_{02}(\bx) \wh\tz \big)   
		\end{pmatrix}. 
	\end{footnotesize}
\end{equation}
In order to construct an  approximate inverse of \eqref{lin_Kalpha}, we need the operator
\begin{equation}\label{Lomegatrue}
	\cG_\omega := \Pi^\perp \left( \omega\cdot \pa_\bx - J K_{02}(\bx) \right)|_{\cX_\perp}
\end{equation}
to be invertible (on reversible tori), where we recall that $\Pi^\perp$ denotes the projection on the invariant subspace $\cX_{\perp}$ in \eqref{cX.perp}.

	\begin{lem}\label{lemma_inv_infty}
		There exists $\sigma := \sigma(k_0, \kappa_0, \tau) > 0$ such that, if \eqref{ansatz} holds, then, for $S\geq 2(s_0+\sigma)$, for any $s_0 \leq s \leq S/2 -\sigma$ and
		for any $h\in {\cal Y}_\bot^{s+1}$, there exists a solution $f := {\cal G}_\omega^{- 1} h \in \cX_\perp^{s}$ of the equation $\cG_\omega f = h$ satisfying
		\begin{equation}\label{Ginv_est}
			\normk{(\cG_\omega)^{-1} h}{{\cal X}^s_\bot}\lesssim_{s} \normk{h}{{\cal Y}^{s + 1}_\bot}+ \| {\frak I}_0 \|_{{\cal X}^{s +  \sigma}}^{k_0, \upsilon} \| h \|_{{\cal Y}^{s_0 + 1}_\bot}^{k_0, \upsilon}\,.
		\end{equation}
		Moreover, if $h$ is anti-reversible, then $f$ is reversible.
	\end{lem}
	
	We postpone the proof of this lemma to Section \ref{sec:proof.AI}.
	%
	To find an approximate inverse of the linear operator in \eqref{lin_Kalpha} (and so of $\di_{i,\alpha}\cF(i_\delta)$), it is enough to  invert the operator
	\begin{equation}\label{Dsys}
		\D\big[ \wh\phi,\wh\fy,\wh\tz,\wh\alpha \big]:=\begin{pmatrix}
			\omega\cdot \pa_\bx\wh\phi - \pa_\alpha K_{10}(\bx)[\wh\alpha] - K_{20}(\bx)\wh\fy- K_{11}^\top(\bx)\wh\tz \\
			\omega\cdot \pa_\bx \wh\fy +\pa_\alpha\pa_\phi K_{00}(\bx)[\wh\alpha] \\
			\cG_\omega \wh\tz -  J  \left( \pa_\alpha K_{01}(\bx)[\wh\alpha] + K_{11}(\bx)\wh\fy \right)
		\end{pmatrix}
	\end{equation}
	obtained neglecting in \eqref{lin_Kalpha} the terms $\pa_\phi K_{10}$, $\pa_{\phi\phi}K_{00}$, $\pa_\phi K_{00}$, $\pa_\phi K_{01}$, 
	(as they vanish at an exact solution).
	For $(\omega,\tA)\in \tG^{\upsilon} \times \cJ_{\varepsilon}(\tE) $  (recall \eqref{Cset_infty} and \eqref{intro parametro mathtt A dim}), we look for an inverse of $\D$ by solving the system
	\begin{equation}\label{Dsys_tos}
		\D\big[ \wh\phi,\wh\fy,\wh\tz,\wh\alpha \big] = \begin{pmatrix}
			g_1 \\ g_2 \\ g_3
		\end{pmatrix}\,,
	\end{equation}
	where $ (g_1, g_2, g_3)  $ is an anti-reversible torus, i.e. 
	\begin{align}\label{g_revcond}
		&g_1(\bx) = g_1(- \bx) , \qquad 
		g_2(\bx) = - g_2(- \bx) , \qquad 
		\cS g_3(\bx) = - g_3(- \bx) \, .
	\end{align}
	We start with the second equation in \eqref{Dsys}-\eqref{Dsys_tos}, that is $\omega\cdot\pa_\bx\wh\fy = g_2-\pa_\alpha\pa_\phi K_{00}(\bx)[\wh\alpha]$. By  \eqref{g_revcond} and \eqref{Ka.rev},  the right hand side of this equation is odd in $ \bx $.
	In particular it has  
	zero average and so, for $\omega \in \mathtt G^\upsilon$ (recall \eqref{Cset_infty}), one can set
	\begin{equation}\label{why_sol}
		\wh\fy := (\omega\cdot\pa_\bx )^{-1}
		( g_2 -\pa_\alpha\pa_\phi K_{00}(\bx)[\wh\alpha] ) \, .
	\end{equation}
	Next, we consider the third equation $\cG_\omega \wh\tz = g_3 + 
	J ( \pa_\alpha K_{01}(\bx)[\wh\alpha]+ K_{11}(\bx)\wh\fy )$.
	By  Lemma \ref{lemma_inv_infty}, there is an anti-reversible solution 
	\begin{equation}\label{whw_sol}
		\wh\tz := ( \cG_\omega )^{-1} \big(  g_3 +  J
		( \pa_\alpha K_{01}(\bx)[\wh\alpha]+ K_{11}(\bx)\wh\fy ) \big) \, . 
	\end{equation}
	Finally, we solve the first equation in \eqref{Dsys_tos}, which, inserting \eqref{why_sol} and \eqref{whw_sol}, becomes
	\begin{equation}\label{whphi_eq}
		\omega\cdot\pa_\bx \wh\phi = g_1 + M_1(\bx)[\wh\alpha]+ M_2(\bx)g_2 + M_3(\bx)g_3\,,
	\end{equation}
	where
	\begin{align}
		M_1(\bx) & := \pa_\alpha K_{10}(\bx) - M_2(\bx)\pa_\alpha\pa_\phi K_{00}(\bx) + M_3(\bx) J\, \pa_\alpha K_{01} (\bx) \,,\notag  \\
		M_2(\bx) & := K_{20}(\bx) (\omega\cdot\pa_\bx )^{-1} + K_{11}^\top (\bx)\left(\cG_\omega \right)^{-1} J\, K_{11}(\bx)(\omega\cdot\pa_\bx)^{-1}\,, \notag\\
		M_3(\bx) & := K_{11}^\top(\bx)\left( \cG_\omega \right)^{-1} \,.\notag
	\end{align}
	In order to solve \eqref{whphi_eq}, we  choose $\wh\alpha$ such that the average  in $\bx$ of the right hand side is zero. 
	The 
	$ \bx $-average of the matrix $ M_1 $ satisfies $\braket{M_1}_\bx = {\rm Id} + O(\sqrt\varepsilon \upsilon^{-1})$. 
	Then, for $\sqrt\varepsilon\upsilon^{-1}$ small enough, $\braket{M_1}_\bx$ is invertible and $\braket{M_1}_\bx^{-1} = {\rm Id} + O(\sqrt\varepsilon\upsilon^{-1})$. Thus we define
	\begin{equation}\label{whalpha_sol}
		\wh\alpha := -\braket{M_1}_\bx^{-1}\big( \braket{g_1}_\bx + \braket{M_2g_2}_\bx + \braket{M_3 g_3}_\bx \big) \, , 
	\end{equation} 
	and
	the solution of equation \eqref{whphi_eq} 
	\begin{equation}\label{whphi_sol}
		\wh\phi := ( \omega\cdot\pa_\bx )^{-1}\big(  g_1 + M_1(\bx)[\wh\alpha] + M_2(\bx)g_2 + M_3(\bx)g_3 \big)\,
	\end{equation}
	for $\omega \in \mathtt G^\upsilon$. 
	Moreover, using \eqref{g_revcond}, \eqref{Ka.rev},  
	the fact that 
	$ J $ and $ \cS $ anti-commutes and Lemma \ref{lemma_inv_infty}, one checks that 
	$ (\wh \phi, \wh\fy, \wh\tz )$  is  reversible, i.e. 
	\begin{equation}\label{rev-def-TW}
		\wh\phi  (\bx) = - \wh\phi  (- \bx) , \quad 
		\wh\fy(\bx) =  \wh\fy (- \bx) , \quad 
		\cS \wh\tz (\bx) = \wh\tz (- \bx) \, . 
	\end{equation}
	In conclusion, we have obtained a  solution 
	$ ( \wh\phi,\wh\fy,\wh\tz,\wh\alpha )$ of the linear system \eqref{Dsys_tos}, and,
	denoting the norm 
	$ \normk{(\phi,\fy,\tz,\alpha)}{{\cal X}^s}:= \max \big\{ \normk{(\phi,\fy,\tz)}{{\cal X}^s},\abs{\alpha}^{k_0,\upsilon} \big\} $ (where ${\cal X}^s$ is defined in \eqref{prima def cal X s}), we have:  
	\begin{prop}\label{Dsystem}
		For all $(\omega,\tA)\in \tG^{\upsilon}\times \cJ_{\varepsilon}(\tE)$ and for any anti-reversible torus variation $ g =(g_1,g_2,g_3)$ (i.e. satisfying \eqref{g_revcond}),
		the linear system \eqref{Dsys_tos} has a solution $\D^{-1}g:= ( \wh\phi,\wh\fy,\wh\tz,\wh\alpha)$, with $ ( \wh\phi,\wh\fy,\wh\tz,\wh\alpha)$ defined in \eqref{whphi_sol}, \eqref{why_sol}, \eqref{whw_sol}, \eqref{whalpha_sol}, 
		where $( \wh\phi,\wh\fy,\wh\tz)$ is a reversible torus variation,
		satisfying, for any $s_0\leq s\leq S/2 - \sigma$
		\begin{equation}
			\label{est.D-1g}
			\normk{\D^{-1}g}{{\cal X}^s} \lesssim_{S} \upsilon^{-1}\big( \normk{g}{{\cal Y}^{s + \sigma}}+\normk{g}{{\cal Y}^{s_0+ \sigma}}\normk{\fI_0}{{\cal X}^{s+\sigma}} \big) \,. 
		\end{equation}
	\end{prop}
	
	
	Finally, we prove that the operator
	\begin{equation}\label{bT0}
		\bT_0 := \bT_0(i_0):=  ( D\wtG_\delta )(\bx,0,0) \circ \D^{-1} \circ (D G_\delta ) (\bx,0,0)^{-1}
	\end{equation}
	is an approximate right inverse for $\di_{i,\alpha}\cF(i_0)$, where
	$	\wtG_\delta(\phi,\fy,\tz,\alpha) := \left( G_\delta(\phi,\fy,\tz),\alpha \right) $
	is the identity on the $\alpha$-component.  
	\begin{thm} {\bf (Approximate inverse)} \label{alm.approx.inv}
		Assume {\rm (AI)}.  There is $\bar\sigma :=\bar\sigma(\tau,\kappa_0,k_0)>0$ such that, if \eqref{ansatz} holds with $\sigma=\bar\sigma$, then, for all $(\omega,\tA)\in \tG^{\upsilon}\times \cJ_{\varepsilon}(\tE)$ (recall \eqref{Cset_infty} and \eqref{intro parametro mathtt A dim}) and for any anti-reversible torus variation $g:=(g_1,g_2,g_3)$ (i.e. satisfying \eqref{g_revcond}), the operator $\bT_0$ defined in \eqref{bT0} satisfies, for $S\geq 2(s_0+\bar\sigma)$ and for all $s_0 \leq s \leq S/2 - \overline\sigma$,
		\begin{equation}\label{tame-es-AI}
			\begin{aligned}
				& \normk{\bT_0 g}{{\cal X}^s} \lesssim_{s} \upsilon^{-1} \big( \normk{g}{{\cal Y}^{s + \bar\sigma}} +\normk{g}{{\cal Y}^{s_0 + \bar \sigma}}\normk{\fI_0}{{\cal X}^{s + \bar \sigma}}  \big)\,, \\
				& \normk{\bT_0 g}{{\cal Y}^s} \lesssim_{s} \upsilon^{-1} \big( \normk{g}{{\cal Y}^{s + \bar\sigma}} +\normk{g}{{\cal Y}^{s_0 + \bar \sigma}}\normk{\fI_0}{{\cal X}^{s + \bar \sigma}}  \big)
			\end{aligned}
		\end{equation}
		Moreover, the first three components of $\bT_0 g $  form a reversible 
		torus variation (i.e. satisfy \eqref{rev-def-TW}).
		Finally, $\bT_0$ is an approximate right inverse of $\di_{i,\alpha}\cF(i_0)$, namely
		\begin{equation*}
			\di_{i,\alpha}\cF(i_0) \circ \bT_0 - {\rm Id} = \cP(i_0) 
		\end{equation*}
		where, for any variation $ g \in {\cal Y}^{s_0 + \bar\sigma}$, one has 
		\begin{equation}\label{pfi1}
			\begin{aligned}
				\normk{\cP(i_0) g}{{\cal Y}^{s_0}} & \lesssim  \upsilon^{-1}  \normk{\cF(i_0,\alpha_0)}{{\cal Y}^{s_0 + \bar \sigma}}\normk{g}{{\cal Y}^{s_0+\bar\sigma}}   
				%
			\end{aligned}
		\end{equation}
	\end{thm}
	\begin{proof}
		The claim that the first three components of $\bT_0 g$, with $\bT_0$ as in \eqref{bT0}, form a reversible torus variation follows from the facts that $DG_\delta(\bx,0,0)$, $DG_\delta(\bx,0,0)^{-1}$ are reversibility preserving and that $\D^{-1}$ maps anti-reversible torus variations into reversible torus variations, see Proposition \ref{Dsystem}.
		
		First, we note that $DG_\delta(\bx,0,0)$, defined in \eqref{DGdelta}, satisfies the estimates, by Lemma \ref{torus_iso} and \eqref{ansatz}, with $\cZ^{s} = \cX^{s}$ or $\cZ^{s}= \cY^{s}$,
		\begin{align}
			& \! \! \!\!\!\!\! \| DG_\delta(\bx,0,0)[\whi] \|_{{\cal Z}^s}^{k_0,\upsilon}\!  +\!  \| DG_\delta(\bx,0,0)^{-1}[\whi] \|_{{\cal Z}^s}^{k_0,\upsilon} \lesssim_{s} \| \whi \|_{{\cal Z}^s}^{k_0,\upsilon} + \|\fI_0 \|_{{\cal X}^{s+\sigma}}^{k_0,\upsilon}\|\whi \|_{{\cal Z}^{s_0}}^{k_0,\upsilon}, \label{chiove1}  \\
			&\! \! \!\!\!\!\!\| D^2G_\delta(\bx,0,0)[\whi_1,\whi_2] \|_{{\cal Z}^s}^{k_0,\upsilon} \lesssim_{s} \| \whi_1 \|_{{\cal Z}^s}^{k_0,\upsilon} \| \whi_2 \|_{{\cal Z}^{s_0}}^{k_0,\upsilon} +  \| \whi_1 \|_{{\cal Z}^{s_0}}^{k_0,\upsilon} \| \whi_2 \|_{{\cal Z}^{s}}^{k_0,\upsilon} \notag \\
			& \quad \quad \quad \quad \quad \quad \quad \quad \quad\quad \ \ \ \ \ + \| \fI_0 \|_{{\cal X}^{s+\sigma}}^{k_0,\upsilon}  \| \whi_1 \|_{{\cal Z}^{s_0}}^{k_0,\upsilon}  \| \whi_2 \|_{{\cal Z}^{s_0}}^{k_0,\upsilon}\,,\label{chiove2}
		\end{align}
		for any variation $\whi:=(\wh\phi,\wh\fy,\wh\tz)$ and for some $\sigma=\sigma(k_0, \kappa_0, \tau)>0$. These latter estimates, together with Proposition \ref{Dsystem} imply the first estimate in \eqref{tame-es-AI}. The second estimate easily follows from the first one by recalling the property \eqref{norma Ys Xs}.

		We now compute the operator $\cP$ and prove the estimate \eqref{pfi1}. By \eqref{F_op} and Lemma \ref{torus_iso}, since $X_\cN$ is independent of the action $I$, we have
		\begin{equation}\label{part0}
			\di_{i,\alpha} \cF(i_0) - \di_{i,\alpha}\cF(i_\delta) 
			= \sqrt{\varepsilon} \int_{0}^{1} \pa_I \di_i X_{{\cal P}_\varepsilon}(\theta_0,I_\delta
			+ \lambda(I_0-I_\delta),z_0)
			[I_0-I_\delta,\Pi[\,\cdot\,]]    \wrt \lambda =: \cE_0 \,, 
		\end{equation}
		where $\Pi$ throughout this proof denotes the projection $(\whi,\wh\alpha)\to\whi$. Denote by $\tu:=(\phi,\fy,\tz)$ the symplectic coordinates induced by $G_\delta$ in \eqref{Gdelta}. Under the symplectic map $G_\delta$, the nonlinear operator $\cF$ in \eqref{F_op} is transformed into 
		\begin{equation}
			\cF(G_\delta(\tu(\bx)),\alpha) = DG_\delta(\tu(\bx))\big(\omega\cdot\pa_{\bx} \tu(\bx) -X_{K_\alpha} (\tu(\bx),\alpha) \big) \,,
		\end{equation}
		with $K_\alpha= H_\alpha\circ G_\delta$ as in \eqref{Kalpha}. By differentiating at the trivial torus $\tu_\delta(\bx):= G_{\delta}^{-1}(i_\delta)(\bx) = (\bx,0,0)$ and at $\alpha=\alpha_0$, we get
		\begin{equation}\label{part1}
			\di_{i,\alpha}\cF(i_\delta) =  DG_\delta(\tu_\delta(\bx))\big(\omega\cdot\pa_{\bx}
			-\di_{\tu,\alpha}X_{K_\alpha} (\tu_\delta(\bx),\alpha_0) \big) D\wtG_\delta(\tu_\delta)^{-1} + \cE_1 \,,
		\end{equation}
		where
		\begin{equation}\label{cE1}
			\cE_1:= D^2 G_\delta(\tu_\delta) [DG_\delta (\tu_\delta)^{-1} \cF(i_\delta,\alpha_0), DG_\delta(\tu_\delta)^{-1}\Pi[\,\cdot\,]  ]\,.
		\end{equation}
		Furthermore, by \eqref{lin_Kalpha}, \eqref{Lomegatrue}, \eqref{Dsys}, we split
		\begin{equation}\label{split.field.delta}
			\omega\cdot\pa_{\bx}  -\di_{\tu,\alpha}X_{K_\alpha} (\tu_\delta(\bx),\alpha_0) = \D + R_Z\,,
		\end{equation}
		where
		\begin{equation}\label{RZ}
			R_Z[\wh\phi,\wh\fy,\wh\tz,\wh\alpha]:=\begin{pmatrix}
				- \pa_\phi K_{10}(\bx)[\wh\phi]   \\
				\pa_{\phi\phi}K_{00}(\bx)[\wh\phi] + [\pa_\phi K_{10}(\bx)]^\top \wh\fy + [\pa_\phi K_{01}(\bx)]^\top  \wh\tz  \\
				- J  \pa_\phi K_{01}(\bx)[\wh\phi] 
			\end{pmatrix}.
		\end{equation}
		Summing up \eqref{part0}, \eqref{part1} and \eqref{split.field.delta}, we get the decomposition
		\begin{equation}\label{part2}
			\di_{i,\alpha}\cF(i_0) = DG_\delta(\tu_\delta)\circ \D\circ D\wtG_\delta(\tu_\delta)^{-1} + \cE\,,
		\end{equation}
		where
		\begin{equation}\label{cE.total}
			\cE:= \cE_0 + \cE_1 + DG_\delta(\tu_\delta)\circ R_Z\circ D\wtG_\delta(\tu_\delta)^{-1}\,,
		\end{equation}
		with $\cE_0$, $\cE_1$ and $R_Z$ defined in \eqref{part0}, \eqref{cE1}, \eqref{RZ}, respectively. Applying $\bT_0$ defined in \eqref{bT0} to the right of \eqref{part2}, since $\D\circ\D^{-1}={\rm Id}$ by Proposition \ref{Dsystem}, we get
		\begin{equation}
			\di_{i,\alpha}\cF(i_0)\circ \bT_0 - {\rm Id} = \cP\,, \quad \cP:= \cE \circ \bT_0\,.
		\end{equation}
		
		By \eqref{cE.total}, Lemmata \ref{XP_est}, \ref{torus_iso} and \eqref{ansatz}, \eqref{chiove1}, \eqref{chiove2} and using the ansatz \eqref{ansatz}, we obtain the estimate
		\begin{equation}\label{chiove3}
			\| \cE[\whi,\wh\alpha] \|_{{\cal Y}^{s_0}}^{k_0,\upsilon} \lesssim \| Z \|_{{\cal Y}^{s_0+\sigma}}^{k_0,\upsilon} \| \whi \|_{{\cal Y}^{s_0+\sigma}}^{k_0,\upsilon}\,,
		\end{equation}
		where $Z=\cF(i_0,\alpha_0)$, recall \eqref{ZError}. The estimate on ${\cal P}$ then follows by \eqref{chiove3} and the second estimate in \eqref{tame-es-AI},  assuming \eqref{ansatz} for some $\bar \sigma \gg \sigma \gg 0$ large enough.
	\end{proof}


	\subsection{Invertibility of the operator $\cG_{\omega}$ and proof of Lemma \ref{lemma_inv_infty}}\label{sec:proof.AI}
	
	In this section we prove the invertibility of the operator ${\cal G}_\omega$ as a bounded operator $\cX_\perp^{s + 1} \to {\cal Y}^s_\bot$ (recall their definitions in \eqref{prima def cal X s bot}). 
	First, we write an explicit expression of the linear operator $\cG_{\omega}$, defined in \eqref{Lomegatrue}, as the projection on the normal directions of the linearized Hamilton equation \eqref{F_op} at the approximate solution, up to a remainder with finite rank, therefore bounded and small.
	
	\begin{lem}\label{lem:K02}
		The Hamiltonian operator $\cG_\omega$ in \eqref{Lomegatrue}, acting on 
		the subspace $ \cX_\perp^s$,  has the form
		\begin{equation}\label{cLomega_again}
			\cG_\omega = \Pi^\perp
			(\cG - \sqrt{\varepsilon} J R)|_{\cX^s_\perp}  \,,
		\end{equation}
		where:
		\\[1mm]
		\noindent	$(i)$ $ \cG $ is the Hamiltonian operator 
		\begin{equation}\label{cL000}
			\cG := \omega\cdot \pa_\bx  - J \pa_\zeta\nabla_\zeta H_\varepsilon(y,T_\delta(\phi)) \, ,
		\end{equation}
		where  
		$H_\varepsilon$  is the Hamiltonian in \eqref{Hami} evaluated at 
		\begin{equation}\label{Tdelta}
			T_\delta(\phi):=   A ( i_\delta (\phi) ) =   A\left( \theta_0(\phi),I_\delta(\bx),z_0(\phi) \right) = v^\intercal\left( \theta_0(\phi),I_\delta(\phi) \right) + z_0(\phi)\,,
		\end{equation}
		the torus  $i_\delta(\phi):= ( \theta_0(\phi),I_\delta(\phi),z_0(\phi) )$ is as in Lemma 
		\ref{torus_iso} 
		and  $A(\theta,I,z) $, $ v^\intercal(\theta,I)$ in \eqref{aa_coord};
		\\[1mm]
		\noindent $(ii)$
		$ R (\phi) $ is Hamiltonian and it has the finite rank form  $ R(\phi)[h] = \sum_{j=1}^{\kappa_0} \left( h,g_j \right)_{L^2} \chi_j $ for any $ h({\bf x}, y)$,
		for functions $g_j,\chi_j \in {\cal Y}^s_\bot$ that satisfy, for some $\sigma:= \sigma(\tau,\kappa_0, k_0) > 0 $, for $S\geq 2(s_0+\sigma)$,  for all $ j = 1, \ldots, \kappa_0 $ and for all $s_0\leq s \leq S/2 - \sigma$, 
		\begin{equation}
			\label{gjchij_est}
			\begin{aligned}
				\norm{ g_j }_{{\cal Y}^s_\bot}^{k_0,\upsilon} + \norm{ \chi_j }_{{\cal Y}^s_\bot}^{k_0,\upsilon} & \lesssim_s 1 + \norm{ \fI_0 }_{{\cal X}^{s + \sigma}}^{k_0,\upsilon} \,.
			\end{aligned}
		\end{equation}
		Furthermore, the operator $ \cG_\omega $, ${\cal G}$, $R$ are reversible.
	\end{lem}
	
	\begin{proof}
		The claims are proved as in Lemma 7.1 in \cite{BFM} and Lemma 6.1 in \cite{BM}. We refer to these articles for more details.
	\end{proof}

	The operator $\cG$ in \eqref{cL000}
	is obtained by linearizing the original system \eqref{first_order} at any torus 
	\begin{equation}\label{small_embed}
		\begin{aligned}
			\bar\zeta(\bx)=(\bar\zeta_1(\bx),\bar\zeta_2(\bx))= T_\delta(\bx)\,,
		\end{aligned}
	\end{equation}
	with $T_\delta(\bx)$ as in \eqref{Tdelta}.
	Explicitly, we have
	\begin{equation}\label{cG0}
		\cG = \omega\cdot\pa_\bx - \begin{pmatrix}
			0 & {\rm Id} \\ \cL_{\fm} + a(\bx,y) & 0
		\end{pmatrix}\,, \quad a(\bx, y) := \sqrt{\varepsilon}q_\varepsilon(y, \bar\zeta_1( \bx, y))
	\end{equation}

	\begin{proof}[Proof of Lemma \ref{lemma_inv_infty}.]
		First, we estimate the norm of the function $a(\bx,y)$. By Moser composition estimates in Lemma \ref{compo_moser}, we have that the norm of $\bar{\zeta}$ satisfies $\| \bar\zeta\|_{s, 1}^{k_0,\upsilon} \lesssim_s 1+\| \fI_0\|_{{\cal X}^s}^{k_0,\upsilon}$ and hence by the ansatz \eqref{ansatz}, $\| \bar\zeta \|_{s_0, 1}^{k_0, \upsilon} \lesssim 1$. Thus, by  estimates \eqref{stime nonlinearita riscalata}, we get, for any $s_0 \leq s \leq S/2 - \sigma$,
		\begin{equation}\label{a0_est}
			\| a \|_{s, 1}^{k_0,\upsilon} \lesssim_s \sqrt{\varepsilon}   (1+\| \fI_0\|_{{\cal X}^{s + \sigma}}^{k_0,\upsilon}) \,.
		\end{equation}
		Now, we want to solve the equation $\cG_{\omega} f=h$. By \eqref{cLomega_again} \eqref{cG0} we write $	\cG_\omega  = {\cal D}_\omega +  {\mathcal R}_\varepsilon$, where
		\begin{equation}\label{cDomega}
			\begin{aligned}
				{\cal D}_\omega :=  \Pi^{\perp} \Big( \omega\cdot\pa_\bx -\begin{pmatrix}
					0 & {\rm Id} \\ \cL_{\fm} & 0
				\end{pmatrix} \Big) \Pi^{\perp}, \ \
				{\cal R}_\varepsilon := - \Pi^\perp \Big( \begin{pmatrix}
					0 & 0 \\ a & 0
				\end{pmatrix} + \sqrt{\varepsilon} J R \Big)\Pi^{\perp},
			\end{aligned}
		\end{equation}
		where $JR(\phi)$ is the finite rank operator  in Lemma \ref{lem:K02}-$(ii)$.
		 In order to invert $\cG_{\omega}$, we conjugate it with the symmetrizing invertible transformation
		\begin{equation}\label{cM.transf}
			\begin{aligned}
				\cM_{\perp} := \frac{\Pi^{\perp}}{\sqrt 2} \begin{pmatrix}
					\cL_{\fm,\perp}^{- 1/2} & \cL_{\fm,\perp }^{- 1/2} \\ -{\rm Id} & {\rm Id}
				\end{pmatrix}\Pi^{\perp}\,, \quad \cM_{\perp}^{- 1} :=
				\frac{\Pi^{\perp}}{\sqrt 2} \begin{pmatrix}
					\cL_{\fm,\perp}^{1/2}& -{\rm Id} \\ \cL_{\fm,\perp}^{1/2} & {\rm Id}
				\end{pmatrix}\Pi^{\perp}\,.
			\end{aligned}
		\end{equation}
		where, by the standard functional calculus on the self-adjoint operator $\cL_{\fm}$ in Proposition \ref{L_operator}, for any $\alpha\in \R\setminus\{0\}$ we defined
	\begin{equation}\label{funct.calculus}
		f(y) = \sum_{j=1}^{\infty} f_j \phi_{j,\fm}(y) \ \mapsto \ \cL_{\fm,\perp}^\alpha f(y) := \cL_{\fm}^\alpha \Pi^\perp f(y) = \sum_{j\geq \kappa_{0}+1} f_j \lambda_{j,\fm}^{2\alpha} \phi_{j,\fm}(y)\,.
	\end{equation}
	It is immediate to verify that $\big( {\cal L}_{\mathtt m} f , f \big)_{L^2} = \sum_{j \geq \kappa_0 + 1} \lambda_{j, \mathtt m}^2 |f_j|^2  \geq \lambda_{\kappa_0 + 1}^2 \| f \|_{L^2}^2 $ and
	\begin{equation}\label{bla spec 1}
	\begin{aligned}
	 \big( {\cal L}_{\mathtt m} f , f \big)_{L^2} &\lesssim \| \partial_y f \|_{L^2}^2 + \| Q_{\mathtt m} \|_{L^\infty} \| f \|_{L^2}^2 \lesssim \| \partial_y f \|_{L^2}^2 + \| f \|_{L^2}^2\,, \\
	 \| \partial_y f \|_{L^2}^2 & \lesssim  \big( {\cal L}_{\mathtt m} f , f \big)_{L^2} + |\big( Q_{\mathtt m} f , f \big)_{L^2}|  \lesssim \big( {\cal L}_{\mathtt m} f , f \big)_{L^2} + \| f \|_{L^2}^2 \lesssim \big( {\cal L}_{\mathtt m} f , f \big)_{L^2}\,. 
	\end{aligned}
	\end{equation}
	This implies that, for any function $  f(y) = \sum_{j \geq \kappa_0 + 1} f_j \phi_{j, \mathtt m}(y) \in H_0^1([- 1, 1])$,
	\begin{equation}\label{bla spec 2}
	\begin{aligned}
	& \| f \|_{H^1}^2 \simeq \big( {\cal L}_{\mathtt m, \bot} f, f \big)_{L^2}  = \sum_{j \geq \kappa_0 + 1} \lambda_j^2 |f_j|^2 \,.
	\end{aligned}
	\end{equation}
	By the latter inequality, we easily deduce that, for any $f \in H^{s, 0}_\bot$ and $0 \leq \rho, \alpha \leq 1$, 
	\begin{equation}\label{bla spec 3}
	\begin{aligned}
	& \| {\cal L}_{\mathtt m, \bot}^{- 1} f \|_{s, \rho} \lesssim \| f \|_{s, 0}\,, \quad 
	 \| {\cal L}_{\mathtt m, \bot}^\alpha f \|_{s, 0} \lesssim \| f \|_{s, 1},, \quad 
	\| {\cal L}_{\mathtt m, \bot}^{- \alpha} f \|_{s, \rho} \lesssim \| f \|_{s, \rho}\,.
	\end{aligned}
	\end{equation}
		By \eqref{cDomega} and \eqref{cM.transf}, we have
		\begin{equation}\label{bGomega}
			\begin{aligned}
				\bG_{\omega} & := \cM_{\perp}^{-1} \cG_{\omega} \cM_{\perp} = \bD_{\omega} + \bR_{\varepsilon} \,, \\
				\bD_{\omega} &:=  \Pi^{\perp}\begin{pmatrix}
					\omega \cdot \pa_{\bx} + \cL_{\fm,\perp}^{1/2} & \\ 0 & \omega\cdot\pa_{\bx} -\cL_{\fm,\perp}^{1/2}
				\end{pmatrix} \Pi^{\perp}\,, \quad \bR_{\varepsilon} := \cM_{\perp}^{-1} \cR_{\varepsilon} \cM_{\perp}\,.
			\end{aligned}
		\end{equation}
		We split the rest of the proof in several steps.
		\\[1mm]
		\noindent
		{\sc Step 1) Analysis of ${\bf R}_\varepsilon$.} We split the remainder ${\bf R}_\varepsilon$ in \eqref{bGomega} as 
		\begin{equation}\label{split cal R epsilon R 1 R2}
			\begin{aligned}
				& {\bf R}_\varepsilon  = {\bf R}_1 + {\bf R}_2 \,, \\
				&{\bf R}_1 := - {\cal M}_\bot^{- 1}\Pi_\bot \begin{pmatrix}
					0 & 0 \\
					a & 0
				\end{pmatrix} \Pi_\bot {\cal M}_\bot \,, \quad {\bf R}_2 =  - \sqrt{\varepsilon} {\cal M}_\bot^{- 1}  J R {\cal M}_\bot \,.
			\end{aligned}
		\end{equation}
		By a direct calculation and using that $J R$ is a finite rank operator (see Lemma \ref{lem:K02}), the operators ${\bf R}_1, {\bf R}_2$ have the following form: 
		\begin{equation}\label{forma bf R1 bf R2}
			\begin{aligned}
				{\bf R}_1& =  \frac{\Pi_\bot }{2} \begin{pmatrix}
					a {\cal L}_{\mathtt m,\perp}^{- 1/2} & a {\cal L}_{\mathtt m,\perp}^{- 1/2} \\
					- a {\cal L}_{\mathtt m,\perp}^{- 1/2} & - a {\cal L}_{\mathtt m,\perp}^{- 1/2}
				\end{pmatrix} \Pi_\bot\,, \\
				{\bf R}_2[u] & =-  \sqrt{\varepsilon}\sum_{j \in S} \big( p_j , u  \big)_{L^2_x} q_j\,, \quad p_j := {\cal M}_\bot^\top[g_j]\,, \quad q_j := {\cal M}_\bot^{- 1}[\chi_j] \,.
			\end{aligned}
		\end{equation}
		The latter formula, together with the estimate \eqref{a0_est}, the tame estimate \eqref{prod2}, the properties \eqref{bla spec 3} and the ansatz \eqref{ansatz} implies that ${\bf R}_1$ satisfies the estimate
		\begin{equation}\label{stima bf R1}
			\| {\bf R}_1 u\|_{s, 0}^{k_0, \upsilon} \lesssim_s \sqrt{\varepsilon}\big( \| u \|_{s, 0}^{k_0, \upsilon} + \| {\frak I} \|_{{\cal X}^s}^{k_0, \upsilon} \| u \|_{s_0, 0}^{k_0, \upsilon} \big)\,, \quad \forall\, s_0 \leq s \leq S/2 - \sigma\,.
		\end{equation}
		We now estimate ${\bf R}_2.$ Since $g_j = (g_j^{(1)}, g_j^{(2)})$, $\chi_{j} = (\chi_{j}^{(1)}, \chi_{j}^{(2)}) \in {\cal Y}^s_\bot$, by \eqref{gjchij_est}, we have, for any $s_0 \leq s \leq S/2 - \sigma$, 
		$$
		\| g_j^{(1)} \|_{s, 1}^{k_0, \upsilon}, \, \| \chi_{j}^{(1)} \|_{s, 1}^{k_0, \upsilon}, \,\| g_j^{(2)} \|_{s, 1}^{k_0, \upsilon}, \, \| \chi_{j}^{(2)} \|_{s, 1}^{k_0, \upsilon} \lesssim_s  1 + \| {\frak I}_0 \|_{s + \sigma}^{k_0, \upsilon}\,.
		$$
		Moreover, since we have, by \eqref{stima bf R1}, \eqref{cM.transf} that
		$$
		p_j = \frac{1}{\sqrt{2}}\begin{pmatrix}
			{\cal L}_{\mathtt m,\perp}^{- 1/2} g_j^{(1)} - g_j^{(2)}\,, \\
			{\cal L}_{\mathtt m,\perp}^{- 1/2} g_j^{(1)} + g_j^{(2)}
		\end{pmatrix}\,, \quad q_j = \frac{1}{\sqrt{2}}\begin{pmatrix}
			{\cal L}_{\mathtt m,\perp}^{1/2} \chi_j^{(1)} - \chi_j^{(2)}\,, \\
			{\cal L}_{\mathtt m,\perp}^{1/2} \chi_j^{(1)} + \chi_j^{(2)}
		\end{pmatrix}\,,
		$$
		we obtain, together with the estimates \eqref{bla spec 3}, that
		$$
		\| p_j \|_{{\cal Y}^s_\bot}^{k_0, \upsilon}\,,\, \| q_j \|_{s, 0}^{k_0, \upsilon} \lesssim_s 1 + \| {\frak I}_0 \|_{\cX^{s+\sigma}}^{k_0, \upsilon}, \quad \forall \, s_0 \leq s \leq S/2 - \sigma\,. 
		$$
		These latter estimates, together with the formula of ${\bf R}_2$ in \eqref{forma bf R1 bf R2} and Lemma \ref{prod.lemma}-$(ii)$
		 to estimates the terms $(p_j, u)_{L^2} q_j$, imply that ${\bf R}_2$ satisfies the same estimate as ${\bf R}_1$ in \eqref{stima bf R1}. We conclude that
		\begin{equation}\label{stima cal R eps inversione}
			\| {\bf R}_\varepsilon u\|_{s, 0}^{k_0, \upsilon} \lesssim_s \sqrt{\varepsilon}\big( \| u \|_{s, 0}^{k_0, \upsilon} + \| {\frak I} \|_{{\cal X}^{s + \sigma}}^{k_0, \upsilon} \| u \|_{s_0, 0}^{k_0, \upsilon} \big)\,, \quad \forall \,s_0 \leq s \leq S/2 - \sigma\,.  
		\end{equation}
		\noindent
		{\sc Step 2) Inversion of ${\bf D}_\omega$.} The operators $  \Pi^\perp \big(\omega\cdot\pa_{\bx}\pm \cL_{\fm,\perp}^{1/2}\big) \Pi^\perp$ are invertible on their range, with bounded and smoothing inverse from 
		$  H_{\bot}^{s,\rho}$ to $H_{\bot}^{s,\rho + 1}$ for any $\rho \geq 0$, where the spaces are defined in \eqref{def H s rho bot}. Indeed, recalling Proposition \ref{L_operator} and \eqref{funct.calculus}, the operators $\big(\omega\cdot\pa_\bx\pm \cL_{\fm,\perp}^{1/2}\big)^{-1}$ act on elements of the basis of the form $e^{\im \ell\cdot\bx}\phi_{j,\fm}(y)$, with  $j\geq \kappa_0+1$ and $\ell\in\Z^{\kappa_0}$, as
		\begin{equation}\label{inverse.cA}
			\big(\omega\cdot\pa_\bx\pm \cL_{\fm,\perp}^{1/2}\big)^{-1} e^{\im \ell\cdot\bx}\phi_{j,\fm}(y) =  \frac{1}{\im\,\omega\cdot\ell \pm \lambda_{j,\fm}} e^{\im \ell\cdot\bx}\phi_{j,\fm}(y)\,.
		\end{equation}
		By Proposition \ref{L_operator}, It is clear that $|\im \,\omega\cdot\ell  \pm \lambda_{j,\fm}|> \lambda_{j,\fm}>0$ for any $j\geq \kappa_0+1$ and $\ell\in\Z^{\kappa_0}$. Therefore, by recalling \eqref{bla spec 2}, $\bD_{\omega}$ is invertible and it satisfies the following bounds, for any $u = (u_1, u_2) \in H^{s, \rho}_\bot \times H^{s, \rho}_\bot$ and any $ s, \rho \geq 0 $,
		\begin{equation}\label{bound bf D omega inv}
			\begin{aligned}
				& 	\| {\bf D}_\omega^{- 1} u \|_{s, 0}^{k_0,\upsilon} \lesssim \| {\bf D}_\omega^{- 1} u \|_{s, 1}^{k_0,\upsilon} \lesssim_s \| u \|_{s, 0}^{k_0,\upsilon}.
			\end{aligned}
		\end{equation}
		\noindent
		{\sc Step 3) Inversion of ${\bf G}_\omega$.} We write $\bG_{\omega}$ in \eqref{bGomega} as
		\begin{equation}\label{bGomega.re}
			\bG_{\omega} =  \big( {\rm Id} +  \bR_{\varepsilon}\bD_{\omega}^{-1} \big) \bD_{\omega}\,.
		\end{equation}
		The estimates in \eqref{stima cal R eps inversione} and \eqref{bound bf D omega inv} imply that 
		\begin{equation}\label{stima resto Neumann 2}
			\big\|  \bR_{\varepsilon} \bD_{\omega}^{-1} u \big\|_{s, 0}^{k_0, \upsilon} \lesssim_s \sqrt{\varepsilon} \big( \| u \|_{s, 0}^{k_0, \upsilon} + \| {\frak I} \|_{{\cal X}^{s + \sigma}}^{k_0, \upsilon}  \| u \|_{s_0, 0}^{k_0, \upsilon}  \big), \quad \forall s_0 \leq s \leq S/2 - \sigma\,. 
		\end{equation}
		Thus, by the ansatz \eqref{ansatz}, for $\varepsilon \ll 1$ small enough, the operator ${\bf G}_\omega =  \big({\rm Id} +  \bR_{\varepsilon} \bD_{\omega}^{-1} \big) {\bf D}_\omega$ is invertible as an operator from ${\cal Y}^s_\bot = H^{s, 1}_\bot \times H^{s, 1}_\bot \to H^{s, 0}_\bot \times H^{s, 0}_\bot$ and  its inverse satisfies the estimate, for all $ s_0 \leq s \leq S/2 - \sigma$, 
		\begin{equation}\label{stima I D omega R Neumann}
			\big\| {\bf G}_\omega^{- 1} u \big\|_{{\cal Y}^s_\bot}^{k_0, \upsilon} \lesssim_s  \big( \| u_1 \|_{s, 0}^{k_0, \upsilon}  + \| u_2 \|_{s, 0}^{k_0, \upsilon}  \big) + \| {\frak I}_0\|_{{\cal X}^{s + \sigma}}^{k_0, \upsilon} \big( \| u_1 \|_{s_0, 0}^{k_0, \upsilon}  + \| u_2 \|_{s_0, 0}^{k_0, \upsilon}  \big)  \,. 
		\end{equation}
		\noindent
		{\sc Step 4) Inversion of ${\cal G}_\omega$ on ${\cal Y}^s_\bot$.} In order to estimate ${\cal G}_\omega^{- 1}$, we observe that, by \eqref{bGomega}, one has ${\cal G}_\omega^{- 1} = {\cal M}_\bot {\bf G}_\omega^{- 1} {\cal M}_\bot^{- 1}$. By the expressions of ${\cal M}_\bot, {\cal M}_\bot^{- 1}$ in \eqref{cM.transf} and the estimates in \eqref{bla spec 3}, one easily deduces that
		$$
		\| {\cal M}_\bot u  \|_{{\cal Y}^s_\bot}^{k_0,\upsilon} \lesssim \| u \|_{{\cal Y}^s_\bot}^{k_0,\upsilon} \quad \text{and} \quad  \| ({\cal M}_\bot^{- 1} u)_1 \|_{s, 0}^{k_0,\upsilon} +\| ({\cal M}_\bot^{- 1} u )_2\|_{s, 0}^{k_0,\upsilon}  \lesssim \| u \|_{{\cal Y}^s_\bot}^{k_0,\upsilon}\,.
		$$
		Therefore, by this latter estimate, together with \eqref{stima I D omega R Neumann} we deduce that, for any $s_0 \leq s \leq S/2 - \sigma$, given $f = (f_1, f_2) \in {\cal Y}^s_\bot$, there exists a solution $h = (h_1, h_2) := {\cal G}_\omega^{- 1} f\in {\cal Y}^s_\bot$ of the equation ${\cal G}_\omega h = f$ satisfying the tame estimates 
		\begin{equation}\label{prima stima cal G omega inverse}
			\| h \|_{{\cal Y}^s_\bot}^{k_0, \upsilon} = \| {\cal G}_\omega^{- 1} f \|_{{\cal Y}^s_\bot}^{k_0, \upsilon} \lesssim_s \|  f \|_{{\cal Y}^s_\bot}^{k_0, \upsilon} + \| {\frak I} \|_{{\cal X}^{s + \sigma}}^{k_0, \upsilon} \| f \|^{k_0, \upsilon}_{{\cal Y}^{s_0}_\bot}, \quad \forall s_0 \leq s \leq S/2 - \sigma\,. 
		\end{equation}
		Moreover, if $f$ is anti-reversible, then $h$ is reversible.
		\\[1mm]
		\noindent
		{\sc Step 5) Estimate of ${\cal G}_\omega^{- 1}$ on ${\cal X}^s_\bot$.} It remains to show that $h = (h_1, h_2) \in {\cal X}^s_\bot$, for any $s_0 \leq s \leq S/2 - \sigma$, namely $h_1 \in H^{s, 3}_\bot$ for any $s_0 \leq s \leq S/2 - \sigma$. This follows essentially by an elliptic regularity argument. Indeed, using that ${\cal L}_{\mathtt m,\perp} = \Pi_\bot (- \partial_y^2 + Q_{\mathtt m}(y)) \Pi_\bot$, by \eqref{cDomega}, \eqref{funct.calculus}, the vector $h = (h_1, h_2)$ solves the system
		\begin{equation}\label{sistema per regolarita ellittica}
			\begin{cases}
				\omega \cdot \partial_{\bf x} h_1 - h_2 =f_1 -  {\cal R}^{(1)}_\varepsilon[h_1, h_2]\,, \\
				\omega \cdot \partial_{\bf x} h_2  + \pa_{y}^2 h_1 - \Pi_\bot Q_{\mathtt m}(y)  h_1 = f_2 - {\cal R}^{(2)}_\varepsilon[h_1, h_2]\,,
			\end{cases}
		\end{equation}
		where we write the remainder ${\cal R}_\varepsilon$ as
		$$
		{\cal R}_\varepsilon[h_1, h_2] = \begin{pmatrix}
			{\cal R}^{(1)}_\varepsilon[h_1, h_2] \\
			{\cal R}^{(2)}_\varepsilon[h_1, h_2]
		\end{pmatrix}  \,.
		$$
		By Lemma \ref{lem:K02}-$(ii)$,  the estimates \eqref{a0_est},  \eqref{prod} (use also the ansatz \eqref{ansatz}), \eqref{prima stima cal G omega inverse} and the form of ${\cal R}_\varepsilon$ in \eqref{cDomega}, one deduces, for any $s_0 \leq s \leq S/2 - \sigma$, 
		\begin{equation}\label{stima cal R eps fine proof}
			\begin{aligned}
				\| {\cal R}^{(2)}_\varepsilon[h_1, h_2] \|_{s, 1}^{k_0, \upsilon} & \leq \| {\cal R}_\varepsilon h \|_{{\cal Y}^s_\bot}^{k_0, \upsilon} \lesssim_s \sqrt{\varepsilon} \big( \| h \|_{{\cal Y}^s_\bot}^{k_0, \upsilon} + \|{\frak I}_0 \|_{{\cal X}^{s + \sigma}}^{k_0, \upsilon} \| h \|_{{\cal Y}^{s_0}_\bot}^{k_0, \upsilon} \big) \\
				& \lesssim_s \sqrt{\varepsilon} \big( \| f \|_{{\cal Y}^s_\bot}^{k_0, \upsilon} + \|{\frak I}_0 \|_{{\cal X}^{s + \sigma}}^{k_0, \upsilon} \| f \|_{{\cal Y}^{s_0}_\bot}^{k_0, \upsilon} \big) \\ 
				\| Q_{\mathtt m} h_1 \|_{s, 1}^{k_0, \upsilon} & \lesssim \| h_1 \|_{s, 1}^{k_0, \upsilon}  \lesssim_s  \big( \| f \|_{{\cal Y}^s_\bot}^{k_0, \upsilon} + \|{\frak I}_0 \|_{{\cal X}^{s + \sigma}}^{k_0, \upsilon} \| f \|_{{\cal Y}^{s_0}_\bot}^{k_0, \upsilon} \big)
			\end{aligned}
		\end{equation}
	where in the second estimate we have used that $Q_{\mathtt m}(y)$ is a smooth potential, with $\| Q_{\mathtt m} \|_{H^1} \leq C(\mathtt E_1 , \mathtt E_2)$ for some constant $C(\mathtt E_1 , \mathtt E_2) >0$. Then, by the second equation in \eqref{sistema per regolarita ellittica}, since $  -\pa_{y}^2 : H^3_0([- 1, 1]) \to H_0^1([- 1, 1])$ is invertible and its inverse $(- \pa_{y}^2)^{- 1}$ gains two derivatives with respect to $y$, namely $\| (- \pa_{y}^2)^{- 1}  u \|_{s, 3}^{k_0, \upsilon} \lesssim \| u \|_{s, 1}^{k_0, \upsilon}$ for any $s \geq 0$, $u \in H^{s, 1}_\bot$, one has 
		$$
		h_1 = ( -\pa_{y}^2)^{- 1}\big[  \omega \cdot \partial_{\bf x} h_2 - \Pi_\bot Q_{\mathtt m}(y)  h_1  - f_2 + {\cal R}^{(2)}_\varepsilon[h_1, h_2] \big]\,.
		$$
		Therefore, by \eqref{prima stima cal G omega inverse}, \eqref{stima cal R eps fine proof} and  for $\varepsilon>0$ small enough, we have,   for any $s_0 \leq s \leq S/2 - \sigma$ with an eventually larger $\sigma$, 
		$$
		\begin{aligned}
			\| h_1 \|_{s, 3}^{k_0, \upsilon} & \leq \big\|  \omega \cdot \partial_{\bf x} h_2 - \Pi_\bot Q_{\mathtt m}(y)  h_1  - f_2 + {\cal R}^{(2)}_\varepsilon[h_1, h_2]  \big\|_{s, 1}^{k_0, \upsilon } \\
			& \leq \| h_2 \|_{s + 1, 1}^{k_0, \upsilon} + \| f_2 \|_{s, 1}^{k_0, \upsilon} + \| Q_{\mathtt m}(y)  h_1 \|_{s, 1}^{k_0, \upsilon} + \| {\cal R}^{(2)}_\varepsilon[h_1, h_2]  \|_{s, 1}^{k_0, \upsilon} \\
			& \lesssim_s \|  f \|_{{\cal Y}^{s + 1}_\bot}^{k_0, \upsilon} + \| {\frak I} \|_{{\cal X}^{s +  \sigma}}^{k_0, \upsilon} \| f \|^{k_0, \upsilon}_{{\cal Y}^{s_0}_\bot}\,. 
		\end{aligned}
		$$
		This latter estimate, together with the estimate \eqref{prima stima cal G omega inverse}, implies the claimed bound of Lemma \ref{lemma_inv_infty}. The proof is then concluded. 
	\end{proof}

\section{Proof of Theorem \ref{NMT}}\label{sez.proof.NMT}

Theorem \ref{NMT} is a consequence of Theorem \ref{NASH} below. Recalling \eqref{prima def cal X s}, we define the spaces ${\cal X}^\infty$ and ${\cal Y}^\infty$ as
$	{\cal X}^\infty := \bigcap_{s \geq 0} {\cal X}^s$, $ {\cal Y}^\infty := \bigcap_{s \geq 0} {\cal Y}^s $
and, for any $\tn \in\N_0$, we define the superexponential scale 
\begin{equation}\label{scales}
	K_\tn: = K_0^{\chi^\tn} \,, \quad \chi= 3/2\, .
\end{equation}
We consider the finite dimensional subspaces
\begin{equation*}
	E_\tn := \big\{ 
	\fI(\bx)= (\Theta,I,z)(\bx) \in {\cal X}^\infty  \ :  \ 
	\Theta = \Pi_\tn \Theta\,, \ I=\Pi_\tn I \,, \ z = \Pi_\tn z \big\}
\end{equation*}
where $\Pi_\tn z := \Pi_{K_\tn} z $   
is defined  as in \eqref{pro:N} with 
$ K_n $ in 
\eqref{scales}, 
and we denote with the same symbol $\Pi_\tn g(\bx)  
:= \sum_{\abs\ell\leq K_\tn} g_\ell e^{\im\ell\cdot\bx}$.  
Note that the  projector $\Pi_{\tn}$ maps (anti)-reversible  variations into (anti)-reversible  variations. We introduce some constants needed to implement the Nash-Moser iteration. Let $\bar\sigma\equiv \bar \sigma(\tau, k_0) \gg 0$ be the largest loss of derivatives coming from the construction of the approximate inverse of the linearized operator in Theorem \ref{alm.approx.inv} and let $S \gg s_0$ be the smoothness of our nonlinearity. Then, we define the following parameters
\begin{equation}\label{costanti Nash Moser}
	\begin{aligned}
		& \bar\mu := 3 \bar\sigma + 1\,, \quad \mu_1 := 3(\bar \mu + 1) + 1\,, \quad  \ta_1:= 3(\bar \mu + 1) + 1 \,, \\
		& \mathtt b_1 := \mathtt a_1 + \bar \mu + \tfrac23 \mu_1 + 3\,, \quad \mathtt a_2 := \mathtt a_1 - 3 \bar \sigma\,, \quad S := 2(s_0 + \mathtt b_1 + \bar \sigma)\,. 
	\end{aligned}
\end{equation}
\begin{rem}
	The constant $\ta_1$ is the exponent in \eqref{P2.2}. The constant $\ta_2$ is the exponent in \eqref{P1.2}. The constant $\tb_1$ is the largest increase of derivatives we need to control with respect to the low regularity $s_0$ and the constant $\mu_1$ is the exponent in the control of the high norm in \eqref{P3.1}. 
\end{rem}

\begin{thm}{\bf (Nash-Moser).} \label{NASH}
	There exist $\delta_0, C_*>0$ such that, if
	\begin{equation}\label{param.NASH}
		\begin{aligned}
			&	K_0^{\bar \mu + \mathtt a_1} \sqrt{\varepsilon}\upsilon^{-1} < \delta_0 \,, \quad
			K_0 := \upsilon^{-1}\,,  \quad \\
			& \upsilon:= \varepsilon^{\rm a}\,, \quad   0< {\rm a} < {\rm a}_0 := (2(1 + \bar \mu + \mathtt a_1))^{-1}\,,
		\end{aligned}
	\end{equation}
	then, for all $\tn\geq 0$:
	\\[1mm]
	\noindent
	$	(\cP 1)_\tn$ There exists a $k_0$-times differentiable function $\wtW_\tn:\R^{\kappa_0}\times \cJ_{\varepsilon}(\tE) \rightarrow E_{\tn-1}\times \R^{\kappa_0}$, $\lambda=(\omega,\tA)\mapsto \wtW_\tn(\lambda):= (\wt\fI_\tn, \wt\alpha_\tn-\omega)$, for $\tn \geq 1 $, and $\wtW_0:=0$, satisfying
	\begin{equation}\label{P1.1}
		\normk{\wtW_\tn}{{\cal X}^{s_0+\bar\sigma}} \leq C_* \sqrt{\varepsilon}\upsilon^{-1} \,.
	\end{equation}
	Let $\wtU_\tn:= U_0+\wtW_\tn$, where $U_0:= (\bx,0,0,\omega)$. The difference $\wtH_\tn:= \wtU_\tn-\wtU_{\tn-1}$, for $\tn \geq 1 $, satisfies
	\begin{equation}\label{P1.2}
		\begin{aligned}
			& \normk{\wtH_1}{{\cal X}^{s_0+\bar\sigma}}\leq C_*  \sqrt{\varepsilon} \upsilon^{-1}\,, \quad
			\normk{\wtH_\tn}{{\cal X}^{s_0+\bar\sigma}} \leq C_*  \sqrt{\varepsilon} \upsilon^{-1} K_{\tn-1}^{-\ta_2}\,, \ \forall\, \tn\geq 2 \,.
		\end{aligned}
	\end{equation}
	The torus embedding $ \wti_\tn := (\bx,0,0) + \wt\fI_\tn $ 
	is  reversible, 
	i.e.  \eqref{RTTT} holds.
	\\[1mm]
	\noindent
	$(\cP 2)_\tn$
	For all $\omega\in \tG^{\upsilon}$ (see \eqref{Cset_infty}),
	setting $K_{-1}:=1$, we have
	\begin{equation}\label{P2.2}
		\normk{\cF(\wtU_\tn)}{{\cal Y}^{s_0}} \leq C_* \sqrt{\varepsilon}  K_{\tn-1}^{-\ta_1} \,.
	\end{equation}
	\noindent
	$(\cP 3)_\tn$ {\sc (High norms)} 
	For any  $\lambda=(\omega,\tA) \in \R^{\kappa_{0}}\times \cJ_{\varepsilon}(\tE)$, we have
	\begin{equation}\label{P3.1}
		\normk{\wtW_\tn}{{\cal X}^{s_0+\tb_1}} \leq C_*  \sqrt{\varepsilon} \upsilon^{-1} K_{\tn-1}^{\mu_1}\,.	
	\end{equation} 
\end{thm}

\begin{proof}
	We argue by induction.
	\\[1mm]
	{\sc STEP 1: Proof of $(\cP 1,2,3)_0$}. By \eqref{F_op}, Lemma \ref{XP_est} and Proposition \ref{prop.F.etax}, we deduce
	\begin{equation}\label{fU0.est}
		\normk{\cF(U_0)}{\cY^{s}} =O( \sqrt{\varepsilon})\,.
	\end{equation}
	The claims then follow by \eqref{fU0.est},
	taking $C_*$ large enough and by noting that $i_0:=(\bx,0,0)$ is clearly reversible.
	\\[1mm]
	{\sc STEP 2: Assume  $(\cP 1,2,3)_\tn$ for some $\tn\in\N_0$ and prove $(\cP 1,2,3)_{\tn+1}$}.
	We are going to define the successive approximation $\wtU_{\tn+1}$ by a modified Nash-Moser scheme and prove by induction that the approximate torus $\wti_{\tn+1}$ is reversible. For that, we prove the almost-approximate invertibility of the linearized operator
	$ G_\tn = G_\tn(\lambda):= \di_{i,\alpha} \cF(\wti_\tn(\lambda)) $.
	We apply Theorem \ref{alm.approx.inv} to $G_\tn(\lambda)$.
	It implies, for $\lambda=(\omega,\tA)\in \tG^{\upsilon}\times \cJ_{\varepsilon}(\tE)$, the existence of an almost-approximate inverse $\bT_\tn := \bT_\tn(\lambda,\wti_\tn(\lambda))$ of the linearized operator $G_{\tn}=\di_{i, \alpha}\cF(\wti_\tn)$ which satisfies, for any anti-reversible variation $g$ and for any $s_0\leq s \leq s_0+\tb_1$,
	\begin{align}
		& \normk{\bT_\tn g}{{\cal X}^s} \lesssim_{s_0+\tb_1} \upsilon^{-1} (\normk{g}{{\cal Y}^{s+\bar\sigma}}+\normk{\wt\fI_\tn}{{\cal X}^{s+\bar\sigma}} \normk{g}{{\cal Y}^{s_0+\bar\sigma}})\,,\label{Ttn.1} \\
		& \normk{\bT_\tn g}{{\cal X}^{s_0}} \lesssim_{s_0+\tb_1} \upsilon^{-1} \normk{g}{{\cal Y}^{s_0+\bar\sigma}} \,.\label{Ttn.2}
	\end{align}
	Moreover, the first three components of $\bT_\tn g$ form a reversible variation.
	For all $\lambda \in \tG^{\upsilon}\times \cJ_{\varepsilon}(\tE)$
	we define the successive approximation
	\begin{equation}\label{succ.approx}
		\begin{aligned}
			&U_{\tn+1} := \wtU_\tn + H_{\tn+1} \,, \\
			&H_{\tn+1} := (\wh\fI_{\tn+1}, \wh\alpha_{\tn+1}) := -\b\Pi_\tn \bT_\tn \Pi_n \cF(\wtU_\tn) \in E_\tn \times \R^{\kappa_0} \,,
		\end{aligned}
	\end{equation}
	where $\b\Pi_\tn$ is defined for any $(\fI,\alpha)$
	by 
	\begin{equation}\label{bPitn}
		\b\Pi_\tn(\fI,\alpha) := ( \Pi_\tn\fI,\alpha) \,, \quad 
		\b\Pi_\tn^\perp:= (\Pi_\tn^\perp \fI,0)\,.
	\end{equation}
	Since 
	$ \wti_{\tn} $ is reversible by induction assumption, we have that 
	$ \cF(\wtU_\tn) = \cF(\wti_\tn, \widetilde \alpha_n) $ is 
	anti-reversible, i.e. \eqref{g_revcond} holds. Thus the first three components of 
	$\bT_\tn \cF(\wtU_\tn) $ form a reversible variation, as well as
	$ \b\Pi_\tn \bT_\tn \Pi_\tn \cF(\wtU_\tn) $. 
	We now show that the iterative scheme in \eqref{succ.approx} is rapidly converging. We write
	\begin{equation*}
		\cF(U_{\tn+1}) = \cF(\wtU_\tn) + G_\tn H_{\tn+1} + Q_{\tn} \,,
	\end{equation*}
	where $G_\tn := \di_{i,\alpha} \cF(\wti_\tn)$ and 
	\begin{equation}\label{Qn}
		\begin{aligned}
			&Q_\tn := Q(\wtU_\tn, H_{\tn+1}) 
			:= \cF(\wtU_\tn+H_{\tn+1}) - \cF(\wtU_\tn) - G_\tn H_{\tn+1}
			\,.
		\end{aligned}
	\end{equation}
	Then, by the definition of $H_{\tn+1}$ in \eqref{succ.approx}, we have
	\begin{equation}\label{FUn+1.split}
		\begin{aligned}
			\cF(U_{\tn+1}) & = \cF(\wtU_\tn) - G_{\tn} \b\Pi_\tn \bT_\tn \Pi_\tn  \cF(\wtU_\tn) + Q_{\tn} \\
			& = ( {\rm Id} - G_{\tn} \bT_\tn  )\cF(\wtU_\tn) + G_{\tn} \b\Pi_\tn^\perp \bT_\tn  \cF(\wtU_\tn) + Q_{\tn} \\
			& =P_{\tn} + R_{\tn} +  Q_{\tn}\,,
		\end{aligned}
	\end{equation}
	where $Q_{\tn}$ is as in \eqref{Qn} and, according also to Theorem \ref{alm.approx.inv},
	\begin{equation}\label{RnPn}
		\begin{aligned}
			&  P_{\tn}:= ({\rm Id}-G_{\tn}\bT_\tn) \Pi_n\cF(\wtU_\tn)= -\cP(\wti_{\tn}) \Pi_n \cF(\wtU_\tn) \,, \\
			& R_{\tn}:= G_{\tn}\b\Pi_\tn^\perp \bT_\tn \Pi_n\cF(\wtU_\tn) + \Pi_\tn^\bot {\cal F}(\widetilde U_\tn)\,. 
		\end{aligned}
	\end{equation}
	First, by  \eqref{F_op} ,  \eqref{fU0.est}, \eqref{P1.1} and Lemma \ref{XP_est}, we have, for any $\lambda \in \R^{\kappa_0}\times \cJ_{\varepsilon}(\tE)$,
	\begin{equation}\label{still.piccolo.s}
		\begin{aligned}
			\| \cF(\wtU_\tn) \|_{{\cal Y}^s}^{k_0,\upsilon} &\leq 	\| \cF(U_0) \|_{{\cal Y}^s}^{k_0,\upsilon} + 	\| \cF(\wtU_\tn) - \cF(U_0)\|_{{\cal Y}^s}^{k_0,\upsilon} \lesssim_{s}  \sqrt{\varepsilon}  + \| \wtW_{\tn} \|_{{\cal X}^{s+\bar\sigma}}^{k_0,\upsilon} \,.
		\end{aligned}
	\end{equation}
	The latter estimate, together with \eqref{param.NASH}, \eqref{P1.1}, implies that
	\begin{equation}\label{piccolo.s0}
		\upsilon^{-1} \| \cF(\wtU_\tn) \|_{{\cal Y}^{s_0}}^{k_0,\upsilon} \leq 1\,.
	\end{equation}
	We start with the estimates for $H_{\tn+1}$ since we will need them for the other estimates. By \eqref{succ.approx}, \eqref{bPitn}, \eqref{SM12}, \eqref{Ttn.1}, \eqref{Ttn.2}, \eqref{still.piccolo.s}, \eqref{piccolo.s0}, we have
	\begin{align}
		\| H_{\tn+1} \|_{{\cal X}^{s_0 + \tb_1}}^{k_0,\upsilon} & \lesssim_{s_0+\tb_1} K_{\tn}^{2\bar\sigma} \| \bT_\tn \cF(\wtU_{\tn}) \|_{{\cal X}^{s_0+\tb_1-2\bar\sigma}} \notag \\ 
		&\lesssim_{s_0+\tb_1} \upsilon^{-1} K_{\tn}^{2\bar\sigma} \big(\| \cF(\wtU_{\tn})\|_{{\cal Y}^{s_0+\tb_1-\bar\sigma}}^{k_0,\upsilon}+\|\wt\fI_\tn\|_{{\cal X}^{s_0+\tb_1-\bar\sigma}}^{k_0,\upsilon}  \| \cF(\wtU_{\tn}) \|_{{\cal Y}^{s_0+\bar\sigma}}^{k_0,\upsilon}\big) \notag  \\
		& \lesssim_{s_0+\tb_1} \upsilon^{-1} K_{\tn}^{2\bar\sigma}\big(  \sqrt{\varepsilon}   + \| \wtW_{\tn} \|_{{\cal X}^{s_0+\tb_1}}^{k_0,\upsilon} \big)\,, \label{Hn+1.est1} \\
		\| H_{\tn+1} \|_{{\cal X}^{s_0}}^{k_0,\upsilon} & \lesssim_{s_0} \upsilon^{-1} K_{\tn}^{\bar\sigma} \| \cF(\wtU_{\tn}) \|_{{\cal Y}^{s_0}}^{k_0,\upsilon}\,, \label{Hn+1.est2} \\
		\| H_{\tn+1} \|_{{\cal X}^{s_0 + \bar \sigma}}^{k_0,\upsilon} & \lesssim  \upsilon^{-1} K_{\tn}^{2 \bar\sigma} \| \cF(\wtU_{\tn}) \|_{{\cal Y}^{s_0}}^{k_0,\upsilon} \,. \label{Hn+1.est3}
	\end{align}
	We estimate $P_{\tn}$, $Q_{\tn}$ and $R_{\tn}$ with respect to the Sobolev norms in the low regularity $s_0$. By the definition of $Q_{\tn}$ in \eqref{Qn}, together with \eqref{F_op}, Lemma \ref{XP_est}, \eqref{P1.1}, \eqref{succ.approx}, \eqref{SM12}, \eqref{Ttn.2}, \eqref{Hn+1.est2}, \eqref{Hn+1.est3}, we have the quadratic estimate, 
	\begin{align}
		\| Q_{\tn} \|_{{\cal Y}^{s_0}}^{k_0,\upsilon} 
		& \lesssim_{s_0} \sqrt{\varepsilon}\big( \| H_{\tn+1} \|_{{\cal X}^{s_0}}^{k_0,\upsilon} \big)^2  \lesssim_{s_0} \sqrt{\varepsilon} \upsilon^{- 2} K_\tn^{2 \bar \sigma} \big(\| {\cal F}(\wtU_\tn) \|_{{\cal Y}^{s_0}}^{k_0, \upsilon} \big)^2  \,. \label{Qn.est}
	\end{align}
	Before estimating $P_{\tn}$, by \eqref{still.piccolo.s} (applied with $s = s_0 + \bar \sigma$), we have
	\begin{equation}\label{stima.interm.1}
		\begin{aligned}
			\| \cF(\wtU_{\tn}) \|_{{\cal Y}^{s_0+\bar\sigma}}^{k_0,\upsilon} & \leq \| \Pi_\tn {\cal F}(\widetilde U_\tn) \|_{{\cal Y}^{s_0 + \bar \sigma}}^{k_0, \upsilon} + \| \Pi_\tn^\bot {\cal F}(\widetilde U_\tn) \|_{{\cal Y}^{s_0 + \bar \sigma}}^{k_0, \upsilon} \\
			& \stackrel{\eqref{SM12}}{\lesssim} K_\tn^{\bar\sigma} \| {\cal F}(\widetilde U_\tn) \|_{{\cal Y}^{s_0}}^{k_0, \upsilon} + K_\tn^{- \mathtt b_1 + 2 \bar \sigma} \| {\cal F}(\widetilde U_\tn)\|_{{\cal Y}^{s_0 + \mathtt b_1 - \bar \sigma}}^{k_0, \upsilon}   \\
			&  \stackrel{\eqref{still.piccolo.s}}{\lesssim} K_\tn^{\bar\sigma} \| {\cal F}(\widetilde U_\tn) \|_{{\cal Y}^{s_0}}^{k_0, \upsilon} + K_\tn^{- \mathtt b_1 + 2 \bar \sigma} \big(\sqrt{\varepsilon} + \| \widetilde W_\tn \|^{k_0, \upsilon}_{{\cal X}^{s_0 + \mathtt b_1}} \big) \,.
		\end{aligned}
	\end{equation}
	By \eqref{RnPn}, Theorem \ref{alm.approx.inv}, \eqref{param.NASH}, \eqref{still.piccolo.s}, \eqref{piccolo.s0}, \eqref{stima.interm.1} and, by induction assumption, \eqref{P1.1} at the step $\tn$, we have
		\begin{align}
			\| P_{\tn} \|_{{\cal Y}^{s_0}}^{k_0,\upsilon} & \lesssim_{s_0} \upsilon^{-1}
			\| {\cal F}(\widetilde U_\tn) \|_{{\cal Y}^{s_0 + \bar\sigma}}^{k_0, \upsilon} \| \Pi_\tn {\cal F}(\widetilde U_\tn) \|_{{\cal Y}^{s_0 + \bar\sigma}}^{k_0, \upsilon}\,, \nonumber \\
			& \lesssim \upsilon^{- 1}K_\tn^{2 \bar\sigma} \big( \| {\cal F}(\widetilde U_\tn) \|_{{\cal Y}^{s_0}}^{k_0, \upsilon} \big)^2 + K_\tn^{- \mathtt b_1 + 3 \bar \sigma} \big(\sqrt{\varepsilon} + \| \widetilde W_\tn \|^{k_0, \upsilon}_{{\cal X}^{s_0 + \mathtt b_1}} \big) \upsilon^{- 1} \| {\cal F}(\widetilde U_\tn) \|_{{\cal Y}^{s_0}}^{k_0, \upsilon} \nonumber \\
			& \stackrel{\eqref{piccolo.s0}}{\lesssim} \upsilon^{- 1}K_\tn^{2 \bar\sigma} \big( \| {\cal F}(\widetilde U_\tn) \|_{{\cal Y}^{s_0}}^{k_0, \upsilon} \big)^2 + K_\tn^{- \mathtt b_1 + 3 \bar \sigma} \big(\sqrt{\varepsilon} + \| \widetilde W_\tn \|^{k_0, \upsilon}_{{\cal X}^{s_0 + \mathtt b_1}} \big)	\,.\label{Pn.est}
		\end{align}
	We now estimate $R_{\tn}$. By \eqref{SM12}, \eqref{F_op}, Lemma \ref{XP_est}, \eqref{P1.1}, \eqref{RnPn}, \eqref{still.piccolo.s}, \eqref{piccolo.s0} and Theorem \ref{alm.approx.inv}, one gets 
	\begin{equation}\label{Rn.est}
		\begin{aligned}
			\| R_{\tn} \|_{{\cal Y}^{s_0}}^{k_0, \upsilon}& \leq \| \Pi_\tn^\bot {\cal F}(\widetilde U_\tn) \|_{{\cal Y}^{s_0}}^{k_0, \upsilon} + \| G_{\tn}\b\Pi_\tn^\perp \bT_\tn \Pi_n\cF(\wtU_\tn) \|_{{\cal Y}^{s_0}}^{k_0, \upsilon}  \\
			& \lesssim K_n^{\bar \sigma - \mathtt b_1} \big(\sqrt{\varepsilon} + \| \widetilde W_\tn \|_{{\cal X}^{s_0 + \mathtt b_1}}^{k_0, \upsilon} \big) + \| \b\Pi_\tn^\perp \bT_\tn \Pi_n\cF(\wtU_\tn) \|_{{\cal X}^{s_0 + 1}}^{k_0, \upsilon} \\
			& \lesssim K_n^{\bar \sigma - \mathtt b_1} \big(\sqrt{\varepsilon} + \| \widetilde W_\tn \|_{{\cal X}^{s_0 + \mathtt b_1}}^{k_0, \upsilon} \big) + K_\tn^{- \mathtt b_1} \|  \bT_\tn \Pi_n\cF(\wtU_\tn) \|_{{\cal X}^{s_0 + \mathtt b_1 + 1}}^{k_0, \upsilon}  \\
			& \lesssim K_n^{\bar \sigma - \mathtt b_1} \big(\sqrt{\varepsilon} + \| \widetilde W_\tn \|_{{\cal X}^{s_0 + \mathtt b_1}}^{k_0, \upsilon} \big)  \\
			&  + K_\tn^{\bar \sigma- \mathtt b_1} \upsilon^{- 1} \Big( \| \Pi_n\cF(\wtU_\tn)  \|_{{\cal Y}^{s_0 + \mathtt b_1 + 1}}^{k_0, \upsilon}  + \| \widetilde W_\tn \|_{{\cal X}^{s_0 + \mathtt b_1 + 1 }}^{k_0, \upsilon} \| \Pi_n {\cal F}(\widetilde U_\tn) \|_{{\cal Y}^{s_0 }}^{k_0, \upsilon}\Big)  \\
			& \lesssim  K_n^{\bar \sigma - \mathtt b_1} \big(\sqrt{\varepsilon} + \| \widetilde W_\tn \|_{{\cal X}^{s_0 + \mathtt b_1}}^{k_0, \upsilon} \big) + K_\tn^{\bar \sigma + 1- \mathtt b_1} \upsilon^{- 1}  \| \Pi_n\cF(\wtU_\tn)  \|_{{\cal Y}^{s_0 + \mathtt b_1 }}^{k_0, \upsilon} \\
			& \lesssim K_\tn^{2 \bar\sigma + 1 - \mathtt b_1} \upsilon^{- 1} \big( \sqrt{\varepsilon} + \| \widetilde W_\tn \|_{{\cal X}^{s_0 + \mathtt b_1}}^{k_0, \upsilon} \big)\,.
		\end{aligned}
	\end{equation}
	By \eqref{FUn+1.split}, \eqref{Qn.est}, \eqref{Pn.est}, \eqref{Rn.est}, \eqref{still.piccolo.s}, \eqref{piccolo.s0}, we finally estimate $\cF(U_{\tn+1})$ by
	\begin{equation}\label{final.est1}
		\begin{aligned}
			&	\| \cF(U_{\tn+1})\|_{{\cal Y}^{s_0}}^{k_0,\upsilon}  \lesssim  \upsilon^{-1} K_{\tn}^{\bar \mu-\tb_1} (\sqrt{\varepsilon} + \| \wtW_{\tn} \|_{{\cal X}^{s_0+\tb_1}}^{k_0,\upsilon} )\\
			& \quad \quad + \upsilon^{-1} K_{\tn}^{\bar\mu} \big( \| \cF(\wtU_{\tn}) \|_{{\cal Y}^{s_0}}^{k_0,\upsilon} \big)^2 + \sqrt{\varepsilon} \upsilon^{-2} K_\tn^{\bar\mu} (\| \cF(\wtU_{\tn})\|_{{\cal Y}^{s_0}}^{k_0,,\upsilon})^2 \\
			& \stackrel{\sqrt{\varepsilon} \upsilon^{- 1} \ll 1}{\lesssim}  \upsilon^{-1} K_{\tn}^{\bar \mu-\tb_1} (\sqrt{\varepsilon} + \| \wtW_{\tn} \|_{{\cal X}^{s_0+\tb_1}}^{k_0,\upsilon} ) + \upsilon^{-1} K_{\tn}^{\bar\mu} \big( \| \cF(\wtU_{\tn}) \|_{{\cal Y}^{s_0}}^{k_0,\upsilon} \big)^2 \,,
		\end{aligned}
	\end{equation}
	with $\bar \mu:=3\bar\sigma+1$.Moreover, by \eqref{succ.approx}, \eqref{Ttn.1}, \eqref{fU0.est}, we have
	\begin{equation}\label{final.est2}
		\| W_{1} \|_{{\cal X}^{s_0+\tb_1}}^{k_0,\upsilon} = \| H_1 \|_{{\cal X}^{s_0+\tb_1}}^{k_0,\upsilon} \lesssim \upsilon^{-1} \| \cF(U_0) \|_{{\cal Y}^{s_0+\bar\sigma+\tb_1}}^{k_0,\upsilon} \lesssim \sqrt{\varepsilon} \upsilon^{-1}\,,  
	\end{equation}
	and, noting that $W_{\tn+1}=\wtW_\tn+ H_{\tn+1}$ for $\tn\geq 1$, we have, by \eqref{Hn+1.est1},
	\begin{equation}\label{final.est3}
		\| W_{\tn+1} \|_{{\cal X}^{s_0+\tb_1}}^{k_0,\upsilon} \lesssim\upsilon^{-1} K_{\tn}^{\bar \mu} (\sqrt{\varepsilon} + \| \wtW_{\tn} \|_{{\cal X}^{s_0+\tb_1}}^{k_0,\upsilon})\,.
	\end{equation}
	We extend $H_{\tn+1}$ in \eqref{succ.approx}, defined for $\omega\in \tG^{\upsilon}$, to $\wtH_{\tn+1}$ defined for all parameters $\lambda\in \R^{\kappa_{0}}\times \cJ_{\varepsilon}(\tE)$ with equivalent $\| \,\cdot\, \|_{s}^{k_0,\upsilon}$ norms and we set $\wtU_{\tn+1}:= \wtU_{\tn} + \wtH_{\tn+1}$. Therefore, by \eqref{final.est1}, \eqref{final.est2}, \eqref{final.est3}, the induction assumptions, the choice of the constants in \eqref{costanti Nash Moser} and the smallness condition in \eqref{param.NASH}, we conclude that \eqref{P1.1}, \eqref{P1.2}, \eqref{P2.2}, \eqref{P3.1} hold at the step $\tn+1$. Finally, by \eqref{succ.approx}, \eqref{F_op}, \eqref{Halpha}, \eqref{invar.rev.Halpha},  Theorem \ref{alm.approx.inv} and the induction assumption on $\wtU_{\tn}$, we have that $\wh\fI_{\tn+1}$ satisfies \eqref{RTTT} and so $\wtU_{\tn+1}$ is a reversible embedding. This concludes the proof.
\end{proof}

\noindent
{\bf Proof of Theorem \ref{NMT}} Let $\upsilon=\varepsilon^{\ta}$, with $0<\ta< \ta_0$ (see \eqref{param.NASH}). Then, there exists $\varepsilon_0>0$ small enough such that the smallness condition \eqref{param.NASH} holds and Theorem \ref{NASH} applies. By \eqref{P1.2}, the sequence $\wtW_\tn = \wtU_\tn - (\bx,0,0,(\omega,\tA))= (\fI_{\tn},\alpha_{\tn}-\omega)$ converges to a function $W_\infty:\R^{\kappa_0}\times \cJ_{\varepsilon}(\tE)\to H_\bx^{s_0} \times H_\bx^{s_0} \times H^{s_0,1}\times \R^{\kappa_0}$ and we define
\begin{equation}
	U_\infty:=(i_\infty,\alpha_{\infty}) = (\bx,0,0,(\omega,\tA)) + W_\infty\,.	
\end{equation}
The torus $i_\infty$ is reversible, namely it satisfies \eqref{RTTT}. Moreover, by \eqref{P1.1}, \eqref{P1.2}, we deduce
\begin{equation}
	\| U_\infty - U_0 \|_{s_0+\bar\sigma} \leq C_* \sqrt\varepsilon\upsilon^{-1} \,, \quad \| U_\infty -\wtU_{\tn} \|_{s_0+\bar\sigma}^{k_0,\upsilon} \leq C_* \sqrt\varepsilon\upsilon^{-1} K_{\tn}^{-\ta_2}\,, \ \forall\, \tn\geq 1\,. 
\end{equation}
In particular, \eqref{alpha_infty}, \eqref{embed.infty.est} hold. By Theorem \ref{NASH}-$(\cP 2)_{\tn}$, we deduce $\cF(\omega,U_\infty(\omega))=0$ for any $(\omega,\tA)\in \mathtt{DC}(\upsilon, \tau)\times \cJ_{\varepsilon}(\tE)$ and hence also for $(\omega,\tA) \in \tG^{\upsilon}\times \cJ_{\varepsilon}(\tE)$ (see \eqref{Cset_infty}, \eqref{intro parametro mathtt A dim}), where the set $\t\Omega$ in \eqref{Cset_infty} is the $\varrho$-neighbourhood of the unperturbed linear frequencies in \eqref{neigh.Omega}. This concludes the proof of Theorem \ref{NMT}.

\begin{footnotesize}
	
\end{footnotesize}

\bigskip

\begin{flushright}
	\textbf{Luca Franzoi}
	
	\smallskip
	
	NYUAD Research Institute
	
	New York University Abu Dhabi
	
	NYUAD Saadiyat Campus
	
	129188 Abu Dhabi, UAE
	
	\smallskip 
	
	\texttt{lf2304@nyu.edu}
	
	\bigskip
	
	\textbf{Nader Masmoudi}
	
	\smallskip
	
	NYUAD Research Institute
	
	New York University Abu Dhabi
	
	NYUAD Saadiyat Campus
	
	129188 Abu Dhabi, UAE
	
	\smallskip
	
	Courant Institute of Mathematical Sciences
	
	New York University
	
	251 Mercer Street
	
	10012, New York, NY, USA
	
	\smallskip 
	
	\texttt{masmoudi@cims.nyu.edu}

	\bigskip
	
	\textbf{Riccardo Montalto}
	
	\smallskip
	
	Dipartimento di Matematica ``Federigo Enriques''
	
	Universit\`a degli Studi di Milano
	
	Via Cesare Saldini 50
	
	20133 Milano, Italy
	
	\smallskip
	
	\texttt{riccardo.montalto@unimi.it}
\end{flushright}

\end{document}